\documentclass[10pt]{article}

\usepackage{amsthm, amsmath,amssymb,bbm}
\DeclareMathAlphabet{\mathbbm}{U}{bbm}{m}{n}
\usepackage{natbib}
\usepackage{multirow}
\usepackage{subfigure}
\usepackage{makecell}
\usepackage{booktabs}
\usepackage{array}
\usepackage{tabularx}
\usepackage{tabulary}
\usepackage{caption}
\usepackage{booktabs}
\usepackage{url}
\usepackage{algorithm}
\usepackage{algpseudocode}
\usepackage{bm}
\usepackage{wrapfig}
\usepackage{lipsum}
\usepackage{mathrsfs}
\usepackage{dsfont}
\usepackage{titling}
\usepackage{relsize}
\usepackage{enumitem}
\usepackage{setspace}
\usepackage[colorlinks,
linkcolor=tpurple,
anchorcolor=tpurple,
citecolor=tpurple
]{hyperref}

\usepackage{tikz}
\usetikzlibrary{positioning}

\usepackage[nameinlink,capitalize]{cleveref}

\definecolor{myred}{HTML}{ae1908}
\definecolor{myblue}{HTML}{05348b}
\definecolor{myorange}{HTML}{ec813b}
\definecolor{mylightblue}{HTML}{9acdc4}
\definecolor{tpurple}{HTML}{743096}
\definecolor{myyellow}{HTML}{e5a84b}
\definecolor{mygreen}{HTML}{6bb392}

\definecolor{revisecolor}{HTML}{000000}

\usepackage{fullpage}

\usepackage[font={footnotesize,it}]{caption}

\numberwithin{equation}{section}

\usepackage{smile}
\newcommand{\indep}{\perp \!\!\! \perp}
\newcommand*{\supp}{\mathrm{supp}}
\usepackage{etoolbox}
\newtoggle{vaos}
\togglefalse{vaos}
\newcommand{\aosversion}[2]{\iftoggle{vaos}{#1}{#2}}

\newcommand{\Rom}[1]{\text{\uppercase\expandafter{\romannumeral #1\relax}}}

\newcommand{\blue}[1]{{\color{blue} #1}}

\newcommand{\myred}[1]{{\color{myred} #1}}
\newcommand{\myblue}[1]{{\color{myblue} #1}}
\newcommand{\mylightblue}[1]{{\color{mylightblue} #1}}
\newcommand{\myorange}[1]{{\color{myorange} #1}}

\newcommand{\myyellow}[1]{{\color{myyellow} #1}}

\begin{document}
\setcounter{tocdepth}{2}

\title{Environment Invariant Linear Least Squares}
\author{ 
	Jianqing Fan
	\qquad Cong Fang
	\qquad Yihong Gu
	\qquad Tong Zhang}
\date{Princeton University, Peking University, and \\ The Hong Kong University of Science and Technology}
\maketitle

\begin{abstract}
\setstretch{1}
    This paper considers a multi-environment linear regression model in which data from multiple experimental settings are collected. The joint distribution of the response variable and covariates may vary across different environments, yet the conditional expectations of the response variable, given the unknown set of important variables, are invariant. Such a statistical model is related to the problem of endogeneity, causal inference, and transfer learning. The motivation behind it is illustrated by how the goals of prediction and attribution are inherent in estimating the true parameter and the important variable set. We construct a novel {\it environment invariant linear least squares (EILLS)} objective function, a multi-environment version of linear least squares regression that leverages the above conditional expectation invariance structure and heterogeneity among different environments to determine the true parameter. Our proposed method is applicable without any additional structural knowledge and can identify the true parameter under a near-minimal identification condition related to the heterogeneity of the environments. We establish non-asymptotic $\ell_2$ error bounds on the estimation error for the EILLS estimator in the presence of spurious variables. Moreover, we further show that the $\ell_0$ penalized EILLS estimator can achieve variable selection consistency in high-dimensional regimes. These non-asymptotic results demonstrate the sample efficiency of the EILLS estimator and its capability to circumvent the curse of endogeneity in an algorithmic manner without any additional prior structural knowledge.  To the best of our knowledge, this paper is the first to realize statistically efficient invariance learning in the general linear model.
\end{abstract} 
\noindent{\bf Keywords}: Least Squares, Endogeneity, Multiple Environments, Invariance, Heterogeneity, Structural Causal Model, Invariant Risk Minimization.
{
\section{Introduction}

The development of statistical regression methods dates back to the least squares proposed in the early nineteenth century \citep{legendre1805leastsquares, gauss1809leastsquares}. The ordinary linear least squares method, also known as {\it the combination of observations}, uses observations (data) to fit a model predicting the {\it response variable} as a linear function of several designated {\it explanatory variables}. At that time, the fitted models for specific tasks, for example, determining lunar motion, were successfully deployed in the real world for predicting the unseen future and were of great commercial and military significance \citep{stigler1986history}. For example, such a prediction of lunar liberation helps determine the ship's position and facilitates navigating the ocean during the Age of Discovery. This kind of prediction requires only a strong correlation with the response variable. Since then, one ultimate goal of fitting a regression model is to discover the law, which is {\it invariant} across time and space or more broadly {\it environments} to some extent, from the data and then ground it in the real world for prediction. The latter requires understanding causation.  

With the rapidly growing amount of high-dimension data in the era of big data, there is a surge in demand for making predictions based on numerous explanatory variables \citep{fan2014challenges,  wainwright2019high, fan2020statistical}. Compared to the circumstances in the nineteenth century, in which the target is to predict the response variable using fixed, carefully chosen explanatory variables, we are now in the stage where the algorithms such as Lasso \citep{tibshirani1997lasso}, SCAD \citep{fan2001variable}, Dantzig selector \citep{candes2007dantzig}, to name a few, can automatically select tens of important variables responsible for the response variable out of thousands or more of candidates. There is considerable literature on the theoretical analysis of these methods regarding the estimation error and variable selection property \citep{zhao2006model, candes2007dantzig, bickel2009simultaneous, buhlmann2011statistics, wainwright2019high, fan2020statistical}, demonstrating their successes and promising prospects.

Let us use a thought experiment to illustrate the underlying risks and potential remedies when incorporating more candidate variables. Suppose our task is to fit a model to classify cows and camels based on extracted hierarchical features from some ``oracle'' agent. We find a dataset $\mathcal{D}$ containing 10k images of cows and camels from the Internet and use 70\% of them to train a classifier on top of two features provided by the agent: the back shape $x_1$, and the background color $x_2$. It is contemplated that cows often appear on the grass while most camels appear on the sand. This indicates that we can use $x_2$ to build a classifier that works well on the training data and the remaining $30\%$ test data by classifying whether the background color is green or yellow. Moreover, incorporating $x_2$ can also increase the accuracy of the classifier built on $x_1$. However, introducing $x_2$ is not what we expected: it may result in a disaster when we deploy it in real-world applications, for example, a detector in a place farming camels and cows, in which the background color is fixed. Problems of a similar nature arise readily in realistic applications \citep{torralba2011unbiased, geirhos2020shortcut} and easily in high dimensions \citep{fan2014endogeneity}. A natural question is whether there are any purely data-driven methods to address such an endogeneity problem. Consider the case where we have another dataset $\tilde{\mathcal{D}}$ in which an association between the background color and object label still exists yet slightly perturbs; for example, 60\%/90\% of the camels stand on sand in the two datasets, respectively. Intuitively, we can combine two different associations between the background color and the object label in these two datasets and infer that $x_2$ may be a ``spurious'' variable for prediction or causation.

The above thought experiment demonstrates that we may suffer from the ``curse of endogeneity'', that the conditional expectation of the response given all the explanatory variables may diverge from the law of interests, when including a lot of variables besides the true important variables before estimation \citep{fan2014challenges}. Such a problem will deviate our route toward building a decent prediction model grounded in the real world and yield non-robust predictions in other environments. Meanwhile, a potential data-driven strategy is to utilize the \emph{heterogeneity} across datasets. This paper implements the above intuition to the linear regression model in statistical modeling, methodology, and theory. We propose a multi-environment version of linear least squares, whose key idea can be summarized as \emph{the combination of combinations of observations} under heterogeneous environments: it combines the linear least squares (\emph{combination of observations}) solutions across different datasets and uses their differences to determine the true parameter $\bbeta^*$ in a completely data-driven manner.

\subsection{The Problem under Study}

In this work, we are interested in predicting the response variable $y \in \mathbb{R}$ with a linear function of the explanatory variable $\bx \in \mathbb{R}^p$ using data from \emph{multiple environments}. Suppose we have collected data from multiple resources/environments. Let $\mathcal{E}$ be the set of environments. For each \emph{environment} $e\in \mathcal{E}$, we observe $n$ i.i.d. $(\bx_1^{(e)}, y_1^{(e)}), \ldots, (\bx_n^{(e)}, y^{(e)}_n) \sim \mu^{(e)}$, typically assumed from the linear model\footnote{Our main theoretical results still hold if the conditional expectation is replaced by $\mathbb{E}[\varepsilon^{(e)} \bx_{S^*}^{(e)}]=0$, in which $\bbeta^*$ is the best linear predictor on important variables.}
\begin{align}
\label{eq:intro-dgp}
    y^{(e)} = (\bbeta^*_{S^*})^\top \bx_{S^*}^{(e)} + \varepsilon^{(e)} ~~~~~ \text{with} ~~~~ \mathbb{E}[\varepsilon^{(e)}|\bx_{S^*}^{(e)}] \equiv 0,
\end{align} where the unknown set of important variables $S^* = \{j: \beta_j^* \neq 0\}$ and the model parameters $\bbeta^*$ are the same, or \emph{invariant}, across different environments, while $\mu^{(e)}$, the distribution of $(\bx^{(e)},y^{(e)})$, may vary. We aim to estimate $\bbeta^*$ and $S^*$ using the $n\cdot |\mathcal{E}|$ data $\{(\bx_i^{(e)}, y_i^{(e)})\}_{e\in \mathcal{E}, i\in \{1,\ldots, n\}}$. The model assumptions of multiple environments resemble and slightly relax the assumptions in this paper's predecessors, for example, \cite{peters2016causal, rojas2018invariant, pfister2021stabilizing, yin2021optimization}.

The main challenge behind identifying $\bbeta^*$ is that the exogeneity condition \citep{engle1983exogeneity} on all predictors, i.e., $\mathbb{E}[\varepsilon^{(e)}|\bx^{(e)}] \equiv 0$ or at least $\mathbb{E}[\varepsilon^{(e)} \bx^{(e)}] = 0$, no longer holds for each single $e\in \mathcal{E}$.  This arises easily in the high-dimensional settings as argued in \cite{fan2014endogeneity} and \cite{fan2014challenges} where a different data-driven strategy is proposed to solve the problem. Instead, for each environment $e\in \mathcal{E}$, the exogeneity condition only holds on the important variables according to \eqref{eq:intro-dgp}. The only assumption available now is
\begin{align}
\label{eq:intro-invariance}
    \forall e\in \mathcal{E}, ~~~~~~ \mathbb{E}[y^{(e)}|\bx^{(e)}_{S^*}] = (\bbeta^*_{S^*})^\top \bx^{(e)}_{S^*},
\end{align} while $\mathbb{E}[y^{(e)}|\bx^{(e)}]$ is not necessarily equal to $(\bbeta^*)^\top \bx^{(e)}$ for each single environment $e\in \mathcal{E}$. This further indicates that $\bbeta^*$ may not be the best linear predictor for environment $e\in \mathcal{E}$, and the gap can be potentially large. Recall when the covariance matrix $\mathbb{E}[\bx^{(e)} (\bx^{(e)})^\top]$ is positive definite, the best linear predictor for environment $e\in \mathcal{E}$ can be written as
\begin{align}
\label{eq:intro-lse-sol}
    \bbeta^{(e)} = \bbeta^* + \left(\mathbb{E}[\bx^{(e)} (\bx^{(e)})^\top]\right)^{-1} \mathbb{E}[\bx^{(e)} \varepsilon^{(e)}].
\end{align} This implies that when only one environment $e\in \mathcal{E}$ is taken into consideration, one can obtain a reduced mean squared error by incorporating those \emph{linear spurious variables}, defined as the variables $x_j$ satisfying $\mathbb{E}[x_j^{(e)} \varepsilon^{(e)}]\neq 0$, and other variables correlated to these variables. They help predict $ \varepsilon^{(e)}$ and reduce prediction error under this environment.   An example is the variable measuring ``background color'' in the above thought experiment. However, these variables are unstable because the corresponding association between these variables and $y$ may vary or even be adversarial in other or unseen environments, leading to biased predictions.

Two ultimate goals of fitting a statistical model using data are prediction and attribution. Regarding ``prediction'', we hope our fitted model can make decent predictions on unseen data grounded in the real world rather than only on the ``demo'' test data. Regarding ``attribution'' \citep{efron2020prediction}, we wish to attribute the outcome/response variable to the significant variables of the fitted model such that the fitted model can lead to true scientific claims. See also Chapter 1 of \cite{fan2020statistical}. We illustrate that the goal of estimating $\bbeta^*$ and $S^*$ for the model \eqref{eq:intro-dgp} unifies the above two seemingly separate goals, thereby expounding the motivation and importance of the multi-environment linear regression model.

\noindent \textbf{Prediction.} Consider the case where the potential distribution of the unseen data $\tilde{\mu}$ may be different from those $\{\mu^{(e)}\}_{e\in \mathcal{E}}$ yet shares the same conditional expectation structure $\mathbb{E}_{\tilde{\mu}}[y|\bx_{S^*}] = (\bbeta^*_{S^*})^\top \bx_{S^*}$ as those of observations in \eqref{eq:intro-dgp}. Without informing of the unseen data distribution ahead, we can define the out-of-sample $L_2$ risk in an adversarial manner as
\begin{align*}
    \mathsf{R}_{\mathtt{oos}}(\bbeta) = \sup_{\substack{\mathbb{E}_\mu[y|\bx_{S^*}] = (\bbeta^*_{S^*})^\top \bx_{S^*} \\ \mathrm{Var}_{\mu}[y|\bx_{S^*}] \lor \max_{1\le j\le p}\mathbb{E}_{\mu}[x_j^2] \le \sigma^2}}  \mathbb{E}_{\mu}\left[|\bbeta^\top \bx - y|^2\right].
\end{align*} 
It follows from \cref{prop:oos-definition-property} in \cref{subsec:problem-formulation} that $\bbeta^*$ minimizes $\mathsf{R}_{\mathtt{oos}}(\bbeta)$ though it may not be the best linear predictor for a specific environment $\tilde{\mu}$. This demonstrates that $\bbeta^*$ is the optimal linear predictor robust to all potential distribution shifts on the unseen data to some extent. Moreover, \cref{prop:oos-definition-property} also implies that the $\ell_2$ error $\|\bbeta - \bbeta^*\|_2$ can be treated as an surrogate of the excess out-of-sample $L_2$ risk. This connects the estimation error of $\bbeta^*$ to the ``adversarial'' mean squared error on potential unseen data.

\noindent \textbf{Attribution.} Let us restrict the multi-environment linear regression model to a specific instance -- data with different experimental settings \citep{didelez2012direct, peters2016causal} in the context of causal inference. To be specific, suppose there exists an environment $e_0\in \mathcal{E}$ with observational data, and the rest are environments with interventional data \citep{he2008active}, in which some interventions are performed on the variables other than the response variable $y$. The distribution of the variables in each environment can be encoded as a Structural Causal Model (SCM) \citep{glymour2016causal}. Under the modularity \citep{scholkopf2012causal} assumption for SCMs that the intervention on variable $x_j$ only changes the distribution of $\mathbb{P}(x_j|\bx_\mathsf{pa(j)})$ where $\mathsf{pa}(j)$ is the set of direct causes of $x_j$, we can see that $S^*$ in this instance is exactly the ``direct cause'' of the response variable. In this case, inferring $S^*$ from data coincides with discovering the direct cause of the variable $y$, while estimating $\bbeta^*$ is exactly estimating the true causal coefficients characterizing the mechanism $\mathbb{P}(y|\bx_{\mathsf{pa(y)}})=\mathbb{P}(y|\bx_{S^*})$.  See \cref{subsec:causal} for additional details.

\subsection{Related Works}

Multi-environment regression is common in many applications \citep{meinshausen2016methods, vcuklina2021diagnostics}. There is considerable literature proposing methods to estimate $\bbeta^*$ and $S^*$ starting from the pioneering work of \cite{peters2016causal}, in which they propose the ``Invariant Causal Prediction'' (ICP) to do causal discovery. The key idea behind it is the modularity assumption of SCMs, which is also referred to as invariance, autonomy \citep{haavelmo1944probability, aldrich1989autonomy}, and stability \citep{dawid2010identifying}. To be specific, \cite{peters2016causal} considers the multiple environments setting, in which the intervention may be applied to unknown variables other than $y$, leading to the following multi-environment linear regression model,
\begin{align}
\label{eq:intro-icp-cond}
    y^{(e)} = (\bbeta^*)^\top\bx^{(e)} + \varepsilon^{(e)} ~~~~~ \text{with} ~~~~~ \varepsilon^{(e)} \sim \mu_\varepsilon, ~\varepsilon^{(e)} \indep \bx_{S^*}^{(e)}, ~\text{and}~\mathbb{E}[\varepsilon^{(e)}]=0.
\end{align} They exploit the conditional distribution invariance structure, i.e., $\varepsilon^{(e)}|\bx_{S^*}^{(e)}$, and propose a hypothesis testing procedure which can guarantee the selected set $\hat{S}$ satisfies $\mathbb{P}(\hat{S} \subseteq S^*) \ge 1-\alpha$ for given Type-I error $\alpha>0$. There is considerable literature extending this idea to other models such as \cite{heinze2018invariant, pfister2019invariant}. Though the Type-I error is guaranteed for their method, these procedures may collapse to conservative solutions, such as $\hat{S}=\emptyset$, and there is a lack of guarantee in the power of the test.

The conservative nature of ICP methods has sparked the development of numerous optimization-based methods. Built upon the invariance principle, there is also considerable literature \citep{ghassami2017learning, rothenhausler2019causal, rothenhausler2021anchor} proposing provably sample-efficient regression methods for estimating the causal parameter $\bbeta^*$. However, they rely on additional, restrictive structures that simplify the original problem considerably. For instance, the Causal Dantzig \citep{rothenhausler2019causal} presumes a linear SCM model with the heterogeneity of environments resulting from additive interventions. Implementing these methods in practical scenarios necessitates expert domain knowledge to validate these structures before estimation. This requirement introduces potential risks of model misspecification, as these methods are specifically tailored to the assumed structures.

There is also a considerable literature designing methods for the generic linear model \eqref{eq:intro-icp-cond}, for example, \cite{ rojas2018invariant, pfister2021stabilizing, yin2021optimization}. However, these methods tend to be heuristic, and finite sample guarantees are lacking. Inspired by the goal of achieving out-of-distribution generalization, \cite{arjovsky2019invariant} introduced another heuristic and model-agnostic approach. Their method, Invariant Risk Minimization (IRM), seeks a data representation such that the optimal predictors based on it are invariant across all environments. The ideas of IRM and its variants are widely applied in many machine learning tasks \citep{sagawa2019distributionally, zhang2020invariant, krueger2021out, lu2021nonlinear}. Nevertheless, the theoretical understanding of invariance learning remains sparse, and the performance improvement of these methods over the standard empirical risk minimization is not clear \citep{rosenfeld2020risks, kamath2021does}. 

Another critical issue associated with these invariance learning methods is the inadequacy of theoretical insights into their identification conditions, which characterizes when it is possible to identify $\bbeta^*$ in the model \eqref{eq:intro-icp-cond} using infinite data from {\it finite} environments. While \cite{peters2016causal} delves into this issue, providing sufficient conditions for specific intervention and SCM structure types, it falls short of offering general criteria. Similarly, \cite{arjovsky2019invariant} also has some preliminary discussions on linear models, yet it stipulates an impractical requirement: $|\mathcal{E}| \ge d$. The understanding of the identification condition for a developed method is of great significance because it elucidates the method's sample efficiency in terms of the number of environments $|\mathcal{E}|$ required: a stronger identification condition may necessitate a potentially increased number of environments $|\mathcal{E}|$ to recover $\bbeta^*$.

It is worth noting that the invariance structure we can exploit depends on the perturbations we expect in real-world scenarios. Due to its clear causal interpretation, the idea of invariance is also widely adopted in domain adaptation, transfer learning, and out-of-distribution generalization. Numerous invariance forms have been proposed based on various expected perturbations beyond the residual invariance \eqref{eq:intro-dgp}. There is also considerable literature designing methods to leverage these invariance structures \citep{muandet2013domain, gong2016domain, heinze2021conditional} using data from multiple environments. A notable example in classification tasks is the invariance of the label conditional distribution, expressed as $\mathbb{P}^{(e)}(\phi(\bx)|y)\equiv q(\bx, y)$, where $\phi(\cdot)$ is an unknown data representation. We refer readers to \cite{chen2021domain, wang2023causal} for an overview.

\subsection{New Contributions and Comparison with Predecessors}

The significance of recovering $S^*$ and $\bbeta^*$ in the model \eqref{eq:intro-dgp}, the sample inefficiency of previous methods in terms of $n$ and $|\mathcal{E}|$, and the lack of theoretical understanding in invariance learning raise the following question:
\begin{align*}
    \begin{split}
    &\text{\it Is provably sample-efficient estimation of $\bbeta^*$ and $S^*$ in the model \eqref{eq:intro-dgp}}\\
    &\text{\it possible under a general, minimal identification condition?} \\
    \end{split}
\end{align*}

This paper provides an affirmative ``yes'' to the above question. In this paper, we propose the \emph{environment invariant linear least squares (EILLS)} estimator that regularizes linear least squares with a \emph{focused linear invariance regularizer} which promotes the invariance or exogeneity on selected variables. In particular, the population-level EILLS objective over the environments $\mathcal{E}$ is defined as
\begin{align}
\label{eq:intro-eills-obj-popu}
    \mathsf{Q}_\gamma(\bbeta) = \underbrace{\sum_{e\in \mathcal{E}} \mathbb{E} \left[|y^{(e)} - \bbeta^\top \bx^{(e)}|^2\right]}_{\mathsf{R}(\bbeta)} + \gamma \underbrace{\sum_{j=1}^p \mathds{1}\{\beta_j\neq 0\} \times  \sum_{e\in \mathcal{E}} \left|\mathbb{E} [(y^{(e)} - \bbeta^\top \bx^{(e)}) x_j^{(e)}]\right|^2 }_{\mathsf{J}(\bbeta)}
\end{align} with some hyper-parameter $\gamma>0$. From a high-level viewpoint, the introduction of the regularizer $\mathsf{J}(\cdot)$ leverages the invariance structure \eqref{eq:intro-invariance} while the sum of $L_2$ loss $\mathsf{R}(\cdot)$ requires a good overall solution and prevents it from collapsing to conservative solutions such as $\bbeta=0$.  See also \cite{fan2014endogeneity} for a similar method to deal with endogeneity in high-dimensional regression. To get a crude sense of what $\mathsf{J}(\cdot)$ imposes, note that $\mathsf{J}(\cdot)$ discourages selecting variables that have a strong correlation with the fitted residuals in some environments and hence encourages the selected variables to be exogenous, uncorrelated with the fitted residuals for all the environments in a sense that
\begin{align*}
    \forall e\in \mathcal{E}, j~\text{with}~\bar{\beta}_j\neq 0 ~~~~~~~~ \mathbb{E} [(y^{(e)} - \bar{\bbeta}^\top \bx^{(e)}) x_j^{(e)}] = 0.
\end{align*} 
This can be seen when $\gamma \to \infty$ and has a similar spirit of imposing $\mathbb{E}\big[y^{(e)} |\bx_{\supp(\bar{\bbeta})}^{(e)}\big]=\bar{\bbeta}^\top \bx^{(e)}$ for all the $e\in \mathcal{E}$. As for the empirical counterpart, given the $n \cdot |\mathcal{E}|$ observations from $|\mathcal{E}|$ environments, the EILLS estimator minimizes $\hat{\mathsf{Q}}_\gamma(\bbeta)$
which substitutes all the expectations in \eqref{eq:intro-eills-obj-popu} by their corresponding empirical means. One can also use the EILLS estimator in high-dimensional case ($p>n$) by adding an $\ell_0$ penalty with some pre-defined hyper-parameter $\lambda>0$ when needed.

We further develop theories in \cref{sec:theory} characterizing when our proposed EILLS estimator can consistently estimate $\bbeta^*$ and $S^*$. We show that our EILLS estimator can identify $\bbeta^*$ with some large enough $\gamma$ in the sense that $\bbeta^*$ is the unique minimizer of the population-level objective \eqref{eq:intro-eills-obj-popu} if
\begin{align}
\label{eq:intro-ident}
    \forall S\subseteq \{1,\ldots, p\} ~\text{with}~ \sum_{e\in\mathcal{E}} \mathbb{E}[\varepsilon^{(e)} \bx^{(e)}_S] \neq 0 ~~~ \implies ~~~\exists e, e'\in \mathcal{E}, ~~\bbeta^{(e,S)} \neq \bbeta^{(e',S)}
\end{align} where $\bbeta^{(e,S)} = \argmin_{\supp(\bbeta) \subseteq S} \mathbb{E}[|y^{(e)} - \bbeta^\top \bx^{(e)}|^2]$ is the best linear predictor constrained on $S$ for environment $e\in \mathcal{E}$. Such a general identification condition is minimal if only linear information is used. This demonstrates that the EILLS objective can automatically circumvent the problem caused by endogeneity if incorporating any {\it pooled linear spurious variable}, defined as the variable $x_j$ satisfying $\sum_{e\in\mathcal{E}} \mathbb{E}[\varepsilon^{(e)} x^{(e)}_j] \neq 0$, will lead to shifts in the least squares solutions among different environments, and it is sample-efficient in terms of $|\mathcal{E}|$ required. 

Under the general identification condition \eqref{eq:intro-ident}, a series of non-asymptotic results are presented for the EILLS estimator when $\gamma$ is greater than some critical threshold $\gamma_*$, illustrating the sample efficiency of our proposed EILLS estimator. In the low-dimensional regime, the vanilla EILLS estimator can obtain the optimal rate for linear regression. Moreover, the EILLS objective also embodies a variable selection property that can eliminate all the linear spurious variables while keeping all the true important variables provided $n$ is large enough. In the high-dimensional regime where $p > n$, the $\ell_0$ regularized EILLS estimator can achieve the variable selection consistency $\{j: \hat{\beta}_{j} \neq 0\} = S^*$ with high probability in a similar manner to the standard $\ell_0$ regularized least squares. Non-asymptotic upper bounds on $\gamma_*$ for some concrete models are calculated, illustrating the applicability of our general result. 

This paper proposes a provably sample-efficient estimation method for the general model \eqref{eq:intro-dgp}. To the best of our knowledge, it is the {\it first} estimator with non-asymptotic guarantees in terms of $|\mathcal{E}|$ and $n$ for the general multi-environment linear regression model \eqref{eq:intro-dgp}, or a slightly restricted version of it, e.g., \eqref{eq:intro-icp-cond}. As a comparison, previous provably sample-efficient estimation methods, like anchor regression \citep{rothenhausler2021anchor} and Causal Dantzig \citep{rothenhausler2019causal}, have predominantly been confined to the linear SCM with additive intervention case, that is, $((\bx^{(e)})^\top, y^{(e)})^\top \gets \bB ((\bx^{(e)})^\top, y^{(e)})^\top + g(\ba^{(e)}, \bvarepsilon^{(e)})$ with some invariant matrix $\bB\in \mathbb{R}^{(p+1)\times (p+1)}$. These imposed structures not only restrict the scalable applicability of these methods in practice but also hinder their potential for extension to nonlinear models. While our method primarily addresses the linear model, it demonstrates promise for extension to nonlinear models; see a discussion in \cref{sec:ex-nonlinear}. Our approach also shares a similar spirit to the ICP method but necessitates a weaker identification condition to recover the true parameter $\bbeta^*$. This implies that our method is more sample-efficient than ICP regarding $|\mathcal{E}|$. Furthermore, we conduct a numerical analysis to benchmark our proposed EILLS estimator against various other invariance methods, demonstrating its superior performance in \cref{sec:simulation}.

\section{Setup and Background}
\label{sec:setup}

\subsection{Multi-Environment Linear Regression}
\label{subsec:problem-formulation}

Suppose we are interested in uncovering the ``scientific truth'' between the response variable $y$ and $\bx_{S^*}$, a sub-vector of $p$-dimensional covariate vector $\bx \in \mathbb{R}^p$. As $S^*$ is unknown, we collect many more variables, but the collected covariates are easily correlated with residuals of $y$ on $\bx_{S^*}$ \citep{fan2014endogeneity, fan2014challenges}. This gives rise to the following more realistic assumption: The distribution of $\bx$ and $y$, $\mu$, satisfies
\begin{align}
\label{eq:true-law-set}
    \mu \in \mathcal{U}_{\bbeta^*, \sigma^2} = \Bigg\{\begin{split}\mu&: ~\mathbb{E}_\mu[y|\bx_{S^*} ] = (\bbeta^*_{S^*})^\top \bx_{S^*}, \text{Var}_{\mu}[y|\bx_{S^*}] \le \sigma^2, ~~ \mu\text{-a.s.} ~ \bx, \\
    &~~~~~\forall j \in [p], ~ \mathbb{E}_{\mu}[x_j^2] \le \sigma^2
    \end{split}
    \Bigg\}.
\end{align} Here, $S^* \subseteq [p]$ denotes the set of important variables that contribute to explain the ``truth'',  $\bbeta^*$ is the parameter of interest whose support set $\supp(\bbeta^*)$ is $S^*$, and $\sigma^2$ is any given positive number.

With only one environment, it is impossible to identify $S^*$.
Consider the data collected from multiple environments: in each environment $e$, the data $(\bx^{(e)}, y^{(e)})$ follows from some law $\mu^{(e)}$ in $\mathcal{U}_{\bbeta^*, \sigma^2}$. Denote the set of environments by $\mathcal{E}$. For each environment $e\in \mathcal{E}$, we observe $n^{(e)}$ i.i.d. samples $(\bx^{(e)}_1,y^{(e)}_1),\ldots, (\bx^{(e)}_{n^{(e)}}, y^{(e)}_{n^{(e)}})  \sim \mu^{(e)}$ with $\mu^{(e)} \in \mathcal{U}_{\bbeta^*, \sigma^2}$. Let $\mathcal{D}^{(e)} = \{(\bx_i^{(e)}, y_i^{(e)})\}_{i=1}^{n^{(e)}}$ be the data collected from environment $e$. Our goal is to use the data $\{\mathcal{D}^{(e)}\}_{e\in \mathcal{E}}$ collected from multiple environments to find a linear predictor $f(\bx) = \bbeta^\top \bx$ that  minimizes the out-of-sample $L_2$ risk that is robust to all potential out-of-sample distributions in $\mathcal{U}_{\bbeta^*, \sigma^2}$. That is, find the minimizer of
\begin{align}
\label{eq:oos-mse}
    \mathsf{R}_{\mathtt{oos}}(\bbeta;\mathcal{U}_{\bbeta^*, \sigma^2}) = \sup_{\mu \in \mathcal{U}_{\bbeta^*, \sigma^2}} \mathbb{E}_{(\bx,y)\sim \mu} \left[ |y - \bbeta^\top \bx|^2 \right].
\end{align}
The following proposition asserts that $\bbeta^*$ is the unique minimizer of \eqref{eq:oos-mse}, whose proof is given in \aosversion{Appendix~\blue{B.1}}{\cref{prop:oos-definition-property-proof}}.

\begin{proposition}[Properties of $\mathsf{R}_{\mathtt{oos}}$] \label{prop:oos-definition-property} We have $\mathsf{R}_{\mathtt{oos}}(\bbeta^*;\mathcal{U}_{\bbeta^*, \sigma^2})=\sigma^2$, and for any $\bbeta \in \mathbb{R}^p$,
\begin{align}
\label{eq:optimal-beta-s-star}
\begin{split}
\sigma^2 \|\bbeta - \bbeta^*\|_2^2 \le \mathsf{R}_{\mathtt{oos}}(\bbeta;&\mathcal{U}_{\bbeta^*,\sigma^2}) - \mathsf{R}_{\mathtt{oos}}(\bbeta^*;\mathcal{U}_{\bbeta^*, \sigma^2}) \\
&\le \sigma^2 p^2 \|\bbeta-\bbeta^*\|_2^2 + 2 \sigma^2 p \|\bbeta_{[p]\setminus S^*}\|_2.
\end{split}
\end{align}
\end{proposition}
As a by-product, \cref{prop:oos-definition-property} also implies that $\|\bbeta - \bbeta^*\|_2$ can be seen as a surrogate metric for the out-of-sample $L_2$ risk \eqref{eq:oos-mse}. \cref{prop:oos-definition-property} motivates us to design methods to estimate the variable set $S^*$  and the true parameter $\bbeta^*$  using data from multiple environments $\mathcal{E}$ which share the same structure $\mathcal{U}_{\bbeta^*,\sigma^2}$. Particularly, define the CE-invariant set as follows. We will take advantage of the fact that $S^*$ is CE-invariant across $\mathcal{E}$.

\begin{definition}[CE-invariant Set] A variable set $S \subseteq [p]$ is \textbf{conditional expectation invariant} (CE-invariant) across environments $\mathcal{E}$ if there exists some $\bbeta$ with support set $S$ such that
\begin{align}
\label{eq:ce-invariant}
    \forall e\in \mathcal{E}, ~~~~~~ \mathbb{E} [y^{(e)}|\bx_{S}^{(e)}] = \bbeta_S^\top \bx_{S}^{(e)}.
\end{align}
\end{definition} 

\noindent \textbf{Necessity of heterogeneous environments. } The following proposition argues that introducing multiple environments, i.e., $|\mathcal{E}| \ge 2$, with potentially different distributions is necessary to identify $\bbeta^*$. The proof is given in \aosversion{Appendix~\blue{B.2}}{\cref{prop:necessary-multiple-envs-proof}}.

\begin{proposition} \label{prop:necessary-multiple-envs} Given fixed $\bbeta^*$, for any $\bbeta \in \mathbb{R}^p$ such that $\supp(\bbeta) \setminus S^* \neq \emptyset$, there exists some large enough $\sigma^2>0$ and $\mu \in \mathcal{U}_{\bbeta^*, \sigma^2}$ such that
\begin{align*}
    \mathbb{E}_{\mu}[y|\bx] = \bbeta^\top \bx.
\end{align*}
\end{proposition}

Note that $\mathbb{E} [|\mathbb{E}[y|\bx] - y|^2] \le \mathbb{E} [|f(\bx) - y|^2]$ for any measurable function $f$. Hence \cref{prop:necessary-multiple-envs} implies the population $L_2$ risk minimizer $f^{\mu} = \argmin_f \mathbb{E}_\mu [|f(\bx) - y|^2]$ for a specific $\mu$ is not necessarily equal to $f^*(\bx) = (\bbeta^*)^\top \bx$. Instead, the bias $\mathbb{E}[|(f^{\mu} - f^{*})(\bx)|^2]$ between the population $L_2$ risk minimizer and the true regression function $f^*$ can be arbitrarily large. It illustrates that running a linear regression on the data in one environment may not be able to estimate $\bbeta^*$ well. Instead, running a linear regression on data in one sole environment may introduce some \emph{spurious variables} even in a population aspect. These variables are spurious since incorporating them in the prediction model can increase the prediction performance in one environment. However, the associations between these variables and $y$ are unstable and can lead to much worse prediction performances in other environments.

To estimate $\bbeta^*$ well, we should use data from \emph{heterogeneous} environments and exploit the invariant structure \eqref{eq:ce-invariant}. This is the main idea of this paper: we will take advantage of the underlying heterogeneity and invariance to infer the important variable set $S^*$ and the true parameter $\bbeta^*$. Furthermore, we will show later in our theoretical analysis that only two environments, i.e., $|\mathcal{E}|=2$, are enough to estimate $\bbeta^*$ consistently.

\subsection{Notations}

The following notations will be used throughout this paper. Let $\mathbb{R}^+$ and $\mathbb{N}^+$ be the set of positive real numbers/integers, respectively. Let $|S|$ denote the cardinality of a set $S$. Define $2^S = \{A: A\subseteq S\}$. Denote by $[m]=\{1,\ldots, m\}$. Let $a\lor b = \max\{a,b\}$ and $a\land b = \min\{a,b\}$. We use $a(n) \lesssim b(n)$ or $a(n) = O(b(n))$ to represent that there exists some universal constant $C>0$ such that $a(n) \le C\cdot b(n)$ for all the $n \in \mathbb{N}^+$, and use $a(n) \gtrsim b(n)$ or $a(n) = \Omega(b(n))$ if $a(n) \ge c \cdot b(n)$ with some universal constant $c>0$ for any $n\in \mathbb{N}^+$. Denote $a(n) \asymp b(n)$ if $a(n) \lesssim b(n)$ and $a(n)\gtrsim b(n)$. We use the notations $a(n) \ll b(n)$, or $b(n) \gg a(n)$, or $a(n) = o(b(n))$ if $\limsup_{n\to\infty} (a(n) / b(n)) = 0$.

We use bold lower case letter $\bx = (x_1,\ldots, x_d)^\top$ to represent a $d$-dimension vector, let $\|\bx\|_q = (\sum_{i=1}^d |x_i|^q)^{1/q}$ be it's $\ell_q$ norm. We use $\supp(\bx) = \{j\in [d]: x_j \neq 0\}$ to denote the support set of the vector $\bx$. For any $S=\{j_1,\ldots, j_{|S|}\} \subseteq[d]$ with $j_1<j_2<\cdots<j_{|S|}$, we denote $[\bx]_{S} = (x_{j_1},\ldots, x_{j_{|S|}})^\top$ be the $|S|$-dimension sub-vector of $\bx$ and abbreviate it as $\bx_S$ when there is no ambiguity. We use bold upper case letter $\bA = [A_{i,j}]_{i\in [n], j\in [m]}$ to denote a matrix. Denote $\bA_{S,T} = [A_{i,j}]_{i\in S, j\in T}$ be the sub-matrix of $\bA$, and abbreviate $\bA_{S,S}$ as $\bA_S$. We let $\|\bA\|_2 = \max_{\bv\in \mathbb{R}^m, \|\bv\|_2=1} \|\bA \bv\|_2$.

For each $e\in \mathcal{E}$, suppose we have $n^{(e)}$ observations $\{(\bx_i^{(e)}, y_i^{(e)})\}_{i=1}^{n^{(e)}} \subseteq \mathbb{R}^p \times \mathbb{R}$ drawn i.i.d. from some distribution $\mu^{(e)}$. Let $\mathbb{E}[f(\bx^{(e)}, y^{(e)})] = \int f(\bx, y) \mu^{(e)}(d\bx, dy)$ and $\hat{\mathbb{E}}[f(\bx^{(e)},y^{(e)})] = \frac{1}{n^{(e)}} \sum_{i=1}^{n^{(e)}} f(\bx^{(e)}_i, y^{(e)}_i)$ for a measurable function $f$. Define the empirical $L_2$ risk $\hat{\mathsf{R}}^{(e)}$ and population $L_2$ risk $\mathsf{R}^{(e)}$  as
\begin{align}
\label{eq:def-l2-risk}
    \hat{\mathsf{R}}^{(e)}(\bbeta) = \hat{\mathbb{E}}\left[|y^{(e)} - \bbeta^\top \bx^{(e)}|^2\right] ~~~~ \text{and} ~~~~ \mathsf{R}^{(e)}(\bbeta) = \mathbb{E} \left[|y^{(e)} - \bbeta^\top \bx^{(e)}|^2\right].
\end{align} 
We also define $\varepsilon^{(e)} = y^{(e)} - \mathbb{E}[y^{(e)}|\bx_{S^*}^{(e)}] = y^{(e)} - (\bbeta^*)^\top \bx^{(e)}$ and $\varepsilon^{(e)}_i = y^{(e)}_i - (\bbeta^*)^\top \bx^{(e)}_i$. Let $\hat{\bSigma}^{(e)}$ and $\bSigma^{(e)}$ denote the empirical covariance matrix and population covariance matrix for environment $e\in \mathcal{E}$, respectively, that is, $\hat{\bSigma}^{(e)} = \hat{\mathbb{E}} \left[\bx^{(e)} (\bx^{(e)})^\top \right]$ and $\bSigma^{(e)} = \mathbb{E} \left[\bx^{(e)} (\bx^{(e)})^\top\right]$. When $\bSigma^{(e)}$ is positive definite, for any $S\subseteq [p]$ we can define the $\bbeta^{(e,S)}$, the population-level best linear predictor constrained on $S$ for environment $e\in \mathcal{E}$ as
\begin{align}
\label{eq:local-minimizer}
    \bbeta^{(e,S)} = \argmin_{\bbeta \in \mathbb{R}^p, \supp(\bbeta) \subseteq S} \mathsf{R}^{(e)}(\bbeta).
\end{align} It is worth noticing that though $\bbeta^{(e,S)} \in \mathbb{R}^p$, the support set of $\bbeta^{(e,S)}$ is a subset of $S$.

\subsection{An Example: Structural Causal Model with Different Interventions}\label{subsec:causal}

We provide a generic statistical model of interest in \cref{subsec:problem-formulation}. 
Here, we present an instance of the multiple environments setting -- structural causal models with different interventions. We will show in this example where such heterogeneity of environments comes from and provide a specific meaning of $S^*$ and $\bbeta^*$. We first introduce the concept of the Structural Causal Model \citep{glymour2016causal}, also called the Structural Equation Model \citep{bollen1989structural}. 

A structural causal model (SCM) on $p+1$ variables $\bz = (z_1,\ldots, z_{p+1}) \in \mathbb{R}^{p+1}$ consists of $p+1$ structural assignments $\{f_j\}_{j=1}^{p+1}$,
\begin{align}
\label{eq:def-scm}
    z_j \gets f_j(\bz_{\mathtt{pa}(j)}, u_j) ~~~~~~~~ j = 1,\ldots, p+1
\end{align} where $\mathtt{pa}(j) \subseteq [p+1]$ is the set of parents, or \emph{direct cause}, of the variable $z_j$, and $u_1, \ldots, u_{p+1}$ are independent zero mean \emph{exogenous variables}. We call an SCM a \emph{linear SCM} if all the functions $f_j$ are linear, in which the above assignments \eqref{eq:def-scm} can be written in a matrix form as $\bz \gets \bB \bz + \bu$ for some structured matrix $\bB \in \mathbb{R}^{(p+1)\times (p+1)}$. 

The above SCM naturally induces a directed graph $G = (V, E)$, where $V = \{1,\ldots, p+1\}$ is the set of nodes, $E$ is the edge set such that $(i,j) \in E$ if and only if $i \in \mathtt{pa}(j)$. We say there is a directed path from $i$ to $j$ if there exists $(v_1,\cdots, v_k)$ with $k\ge 2$ such that $v_1=i$, $v_k=j$ and $(v_l, v_{l+1}) \in E$ for any $l\in [k-1]$. We call a directed graph a \emph{directed acyclic graph (DAG)} if there does not exist a direct path from $j$ to $j$ for any $j\in[p+1]$.  Throughout this paper, the induced graphs of the SCMs we consider are all DAGs.

\begin{figure}[t]
    \centering

    \begin{tikzpicture}[scale=0.7,state/.style={circle, draw, minimum size=0.7cm, scale=0.8}]
    \draw[black, rounded corners] (-1, -4.5) rectangle (4, 1);

    \node[state] at (0,0) (x1) {$x_1$};
    \node[state] at (1.5,-1.5) (y) {$y$};
    \node[state] at (3,0) (x2) {$x_2$};
        \node[state] at (0,-3)  (x3) {$x_3$};
    \node[state] at (3,-3) (x4) {$x_4$};
    
    \draw[->] (x1) -- node{\small $1$} (x2);
    \draw[->] (x2) -- node{\small $2$} (y);
    \draw[->] (y) -- node{\small $1$} (x3);
    \draw[->] (y) -- node{\small $1$} (x4);
    \draw[->] (x4) -- node{\small $.5$} (x3);
    \draw node at (1.5, -4) {\small environment e=1};
    
    \draw[black, fill=mylightblue!10, rounded corners] (5, -4.5) rectangle (10, 1);

    \node[state] at (6,0) (e2x1) {$x_1$};
    \node[state] at (7.5,-1.5) (e2y) {$y$};
    \node[state] at (9,0) (e2x2) {$x_2$};
        \node[state] at (6,-3)  (e2x3) {$x_3$};
    \node[state,fill=gray!50] at (9,-3) (e2x4) {$x_4$};
    
    \draw[->] (e2x1) -- node{\small $1$} (e2x2);
    \draw[->] (e2x2) -- node{\small $2$} (e2y);
    \draw[->] (e2y) -- node{\small $1$} (e2x3);
    \draw[->] (e2y) -- node{\small \myred{$-1$}} (e2x4);
    \draw[->] (e2x4) -- node{\small $.5$} (e2x3);
    \draw node at (7.5, -4) {\small environment e=2};
    
    \draw[black, fill=myblue!10, rounded corners] (11, -4.5) rectangle (16, 1);

    \node[state] at (12,0) (e3x1) {$x_1$};
    \node[state] at (13.5,-1.5) (e3y) {$y$};
    \node[state,fill=gray!50] at (15,0) (e3x2) {$x_2$};
        \node[state,fill=gray!50] at (12,-3)  (e3x3) {$x_3$};
    \node[state] at (15,-3) (e3x4) {$x_4$};
    
    \draw[->] (e3x1) -- node{\small \myred{$-.5$}} (e3x2);
    \draw[->] (e3x2) -- node{\small $2$} (e3y);
    \draw[->] (e3y) -- node{\small $1$} (e3x3);
    \draw[->] (e3y) -- node{\small $1$} (e3x4);
    \draw[->] (e3x4) -- node{\small \myred{$1$}} (e3x3);
    \draw node at (13.5, -4) {\small environment e=3};
\end{tikzpicture}
    \caption{Linear SCMs with different interventions when $p=4$ and $|\mathcal{E}|=3$. Here $z_5=y$, $S^*=\{2\}$, and we omit the dependence on the exogenous variables $u_1,\ldots, u_5$ for a clear presentation. The arrow from node $i$ to node $j$ is marked by $B_{j,i}^{(e)}$ if $B_{j,i}^{(e)} \neq 0$. We can treat $e=1$ as the observational environment and $e=2,3$ as interventional environments. One intervention is performed to the mechanisms of variable $x_4$ (gray shadowed) in environment $e=2$, and simultaneously interventions on variable $x_2$, $x_3$ (gray shadowed) are applied in environment $e=3$. }
    \label{fig:linear-scm-example}
\end{figure}
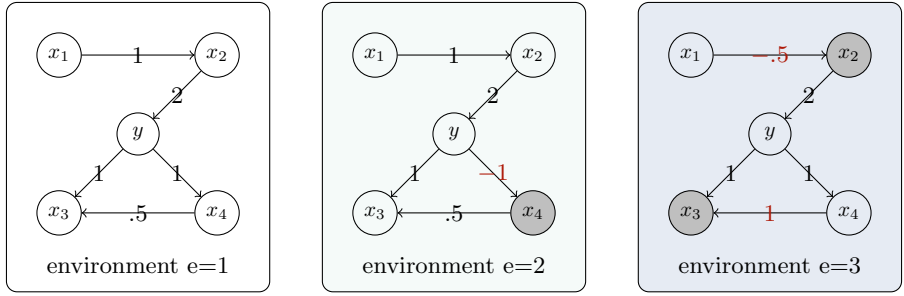

Consider the following $|\mathcal{E}|$ SCMs on $p+1$ variables $\bz = (x_1,\ldots, x_p, y) \in \mathbb{R}^{p+1}$, that for any $e\in \mathcal{E}$, the following assignments holds
\begin{align}
\label{eq:general-scm-example}
\begin{split}
x^{(e)}_j := z_j^{(e)} &\gets {f^{(e)}_j}(\bz^{(e)}_{{\mathtt{pa}(j)}}, {u_{j}^{(e)}}), ~~~~~~~~~~~~~~~~~~~~ j = 1,\ldots, p\\
y^{(e)} := z_{p+1}^{(e)} &\gets ({\bbeta^*})^\top_{{\mathtt{pa}(p+1)}} \bx^{(e)}_{{\mathtt{pa}(p+1)}} + {u_{p+1}^{(e)}},
\end{split}
\end{align} where all the $\bu^{(e)} = (u_1^{(e)}, \ldots, u_{p+1}^{(e)})$ are independent across different environments and also have independent, zero mean entries. Here the direct cause relationship $\mathtt{pa}: [p+1] \to 2^{[p+1]}$ and the function $f_{p+1}(\bx^{(e)}_{\mathtt{pa}(p+1)}, u_{p+1}^{(e)})=({\bbeta^*})^\top_{{\mathtt{pa}(p+1)}} \bx^{(e)}_{{\mathtt{pa}(p+1)}} + {u_{p+1}^{(e)}}$ are both invariant across different $e\in \mathcal{E}$. The structural assignments $f^{(e)}_j$ for variables $x_j^{(e)}$ and the distribution of exogenous variables $u_j^{(e)}$ may vary among different environments. We use the superscript $(e)$ to emphasize this heterogeneity. We use $\bz$ in the assignments \eqref{eq:general-scm-example} to emphasize that $y$ can be the direct cause of some variables $x_j$. We can see that the heterogeneity may come from performing do-interventions on variables other than $y$ (which will result in different $f_j^{(e)}$ or distribution $u_j^{(e)}$). \cref{fig:linear-scm-example} provides an example of the model \eqref{eq:general-scm-example}.  Here we consider the case where there are no hidden confounders and the system is not anti-causal, since the model misspecification in either direction cannot be addressed without imposing further structural assumptions. 

It is easy to verify that the above model \eqref{eq:general-scm-example} is an instance of the problem formulation discussed in  \cref{subsec:problem-formulation} with the important variable set $S^*=\mathtt{pa}(p+1)$ and the true parameter $\bbeta^*$. Moreover, the set $S^*$ and the parameter $\bbeta^*$ have precise meanings in the context of SCM. Specifically, $S^*$ is the set of \emph{direct causes} of the response variable. From a traditional viewpoint of causal inference, for each $j\in S^*$, the corresponding coefficient $\beta_j^*$ measures how the outcome $y$ will change if we only apply intervention on the variable $x_j$.  

\section{Methodology}
\label{sec:method}

\subsection{Focused Linear Invariance Regularizer}
\label{subsec:focused-regularizer}

As discussed in \cref{prop:necessary-multiple-envs} and \cref{subsec:problem-formulation}, it is necessary to leverage the conditional expectation invariance structure across heterogeneous environments. The straightforward idea is to impose a conditional expectation invariance to the solution $\bar{\bbeta}$ we find. That is, we wish our solution $\bar{\bbeta}$ to satisfy
\begin{align*}
    \forall e\in \mathcal{E}, ~~~~~~ \mathbb{E}\left[y^{(e)} \big |\bx^{(e)}_{\supp(\bar{\bbeta})}\right] \equiv \bar{\bbeta}^\top \bx^{(e)}
\end{align*} from {a population perspective}. To implement the idea efficiently, we consider the variational principle that, for any $e\in \mathcal{E}$ and given $\bbeta \in \mathbb{R}$ with support set $S$, 
\begin{align}
     \mathbb{E}[y^{(e)}|\bx_{S}^{(e)}] = \bbeta^\top_S \bx_S^{(e)} ~~&\Leftrightarrow~~ \mathbb{E} \left[ \left(y^{(e)} - \bbeta^\top_S \bx^{(e)}_S\right) g(\bx_S^{(e)})\right] = 0  ~~~~~~ \forall g ~\text{with}~ \mathbb{E} [ |g(\bx^{(e)}_S)|^2 ] < \infty \nonumber \\
     &\overset{(a)}{\Rightarrow} ~~ \mathbb{E} \left[ \left(y^{(e)} - \bbeta^\top_S \bx^{(e)}_S\right) g(\bx_S^{(e)})\right] = 0  ~~~~~~\forall g~\mathrm{linear} \nonumber\\
     &\Leftrightarrow ~~ \mathbb{E} \left[ \left(y^{(e)} - \bbeta^\top_S \bx^{(e)}_S\right) \bx_j^{(e)}\right] = 0 ~~~~~~~~~~~ \forall j \in S \nonumber \\
     &\Leftrightarrow ~~ \nabla_j \mathsf{R}^{(e)}(\bbeta) = 0 ~~~~~~~~~~~~~~~~~~~~~~~~~~~~~~~ \forall j \in S \label{eq:fr-final-condition}
\end{align} where the statements connected by ``$\Leftrightarrow$'' are equivalent and $\mathsf{R}^{(e)}(\bbeta)$ is defined by \eqref{eq:def-l2-risk}, and we make a relaxation in $(a)$ such that the last statement $\nabla_S \mathsf{R}^{(e)}(\bbeta) = 0$ is a necessary but not sufficient condition for $\mathbb{E}[y^{(e)}|\bx_S^{(e)}] = \bbeta^\top_S\bx_S^{(e)}$. 

Based on the above derivations, we propose to minimize a \emph{focused linear invariance regularizer}, whose population-level form $\mathsf{J}(\bbeta; \bomega)$ can be written as
\begin{align}
\label{eq:invariant-regularizer-population}
\mathsf{J}(\bbeta;\bomega) = \sum_{j=1}^p \mathds{1}\{\beta_j \neq 0\} \sum_{e\in \mathcal{E}} \frac{\omega^{(e)}}{4} |\nabla_j \mathsf{R}^{(e)}(\bbeta)|^2,
\end{align} where $\bomega = (\omega^{(e)})_{e\in \mathcal{E}} \in \mathbb{R}^{|\mathcal{E}|}$ are pre-determined weights associated with environments $\mathcal{E}$ satisfying $\sum_{e\in \mathcal{E}} \omega^{(e)}=1$ and $\omega^{(e)} > 0$ for any $e\in \mathcal{E}$. Some typical choices of $\bomega$ can be $\omega^{(e)} = 1/|\mathcal{E}|$ or $\omega^{(e)} = n^{(e)} / (\sum_{e'\in \mathcal{E}} n^{(e')})$. The word ``\emph{focused}'' emphasizes that we only penalize the gradient in the direction of non-zero coordinates due to the statement \eqref{eq:fr-final-condition}.  See \cite{fan2014endogeneity} for a similar idea for dealing with endogeneity.

We first analyze what it implies when $\mathsf{J}(\bbeta; \bomega)$ attains its global minima. Since $S^*$ is CE-invariant across environments $\mathcal{E}$ according to our assumption, we have $\mathsf{J}(\bbeta^*;\omega)=0$ and this implies the global minima of the loss function $\mathsf{J}(\bbeta;\omega)$ is $0$. Moreover, for a fixed $\tilde{\bbeta}$ with support set $\tilde{S}$, it is easy to see that
\begin{align*}
    \mathsf{J}(\tilde{\bbeta};\bomega) = 0 ~~&\Leftrightarrow~~ \forall e\in \mathcal{E}, ~\nabla_{\tilde{S}} \mathsf{R}^{(e)}(\tilde{\bbeta}) = 0 \\
    ~~&\Leftrightarrow ~~\forall e\in \mathcal{E}, ~\tilde{\bbeta} \in \argmin_{\supp(\bbeta) \subseteq \tilde{S}} \mathsf{R}^{(e)}(\bbeta).
\end{align*}
Therefore, though $\mathsf{J}(\tilde{\bbeta};\omega) = 0$ (attaining its global minima) does not imply that the selected variables set $\tilde{S}$ is \emph{CE-invariant} across environments $\mathcal{E}$, we argue that $\mathsf{J}(\tilde{\bbeta};\bomega) = 0$ indeed indicates that the select variable set $\tilde{S}$ is \emph{LLS-invariant} across environments $\mathcal{E}$, which can be defined as follows.
\begin{definition}[LLS-invariant Set] We say a variable set $S \subseteq [p]$ is \textbf{linear least squares invariant} (LLS-invariant) across environments $\mathcal{E}$ if there exists some $\bbeta$ with support set $S$ such that
\begin{align}
\label{eq:lls-invariant}
    \forall e\in \mathcal{E}, ~~~~~~ \bbeta \in \argmin_{\supp(\bbeta) \subseteq S} \mathbb{E} \left[ |y^{(e)} - \bbeta^\top \bx^{(e)}|^2 \right].
\end{align}
\end{definition}

LLS-invariance is weaker than CE-invariance because the CE-invariance of $S^\star$ implies the LLS-invariance of $S^\star$ while the converse is false.  In this paper, we only leverage linear information to facilitate LLS-invariance on the solution we find. That is what the word ``\emph{linear}'' in our regularizer's name emphasizes. According to the above discussion of the relaxation in $(a)$, we can incorporate more hand-crafted nonlinear features other than linear feature $x_j$ in the \emph{focused invariance regularizer} to somewhat strengthen the degree of invariance. For example, consider the following enhanced regularizer that includes another marginal nonlinear feature $h(\cdot)$ in,
\begin{align}
\label{eq:lls-invariant-nonlinear-feat}
    \mathsf{J}_{h}(\bbeta;\bomega) = \sum_{j=1}^p \mathds{1}\{\beta_j \neq 0\} \sum_{e\in \mathcal{E}} \omega^{(e)} \left\{ \left| \mathbb{E}[ \hat{\varepsilon}^{(e)}_{\bbeta} \bx_j^{(e)}] \right|^2 + \left| \mathbb{E}[ \hat{\varepsilon}^{(e)}_{\bbeta} h(\bx_j^{(e)}) ] \right|^2\right\},
\end{align} where $\hat{\varepsilon}^{(e)}_{\bbeta} = y^{(e)} - \bbeta^\top \bx^{(e)}$. Specifically, we can take $h(u) = u^2$ or $h(u) = \cos(u)$.

Though LLS-invariance is weaker than CE-invariance, these two types of invariance are equivalent under some specific models. For example, suppose $\mu^{(e)}_{\bx, y}$ are all joint Gaussian distribution for any $e\in \mathcal{E}$. In that case, it is easy to verify that if a set $S$ is LLS-invariant, then $S$ is also CE-invariant using the fact that uncorrelatedness implies independence for joint Gaussian random variables. In this case, the enhanced regularizer \eqref{eq:lls-invariant-nonlinear-feat} would have no advantages over the most simple linear one \eqref{eq:invariant-regularizer-population}.

\subsection{Our Approach: EILLS}

Given data $\{\mathcal{D}^{(e)}\}_{e\in \mathcal{E}} = \{\{(\bx_i^{(e)},y_i^{(e)})\}_{i=1}^{n^{(e)}}\}_{e\in \mathcal{E}}$ from heterogeneous environments, the empirical-level \emph{focused linear invariance regularizer} with weights $\bomega$ can be written as
\begin{align}
\label{eq:invariant-regularizer-empirical}
    \hat{\mathsf{J}}(\bbeta;\bomega) = \sum_{j=1}^p \mathds{1}\{\beta_j \neq 0\} \sum_{e\in \mathcal{E}} \omega^{(e)} \left| \hat{\mathbb{E}}[x_j^{(e)} (y^{(e)} - \bbeta^\top \bx^{(e)})]\right|^2.
\end{align} 
Recall that $\hat{\mathsf{R}}^{(e)}(\bbeta)$ defined in \eqref{eq:def-l2-risk} is the empirical $L_2$ loss in environment in $e\in \mathcal{E}$. We also define the following \emph{pooled empirical $L_2$ loss} over all the environments: for $\bbeta \in \mathbb{R}^p$,
\begin{align}
\label{eq:pooled-l2-risk-empirical}
    \hat{\mathsf{R}}(\bbeta;\bomega) = \sum_{e\in \mathcal{E}} \omega^{(e)} \hat{\mathsf{R}}^{(e)}(\bbeta) = \sum_{e\in \mathcal{E}}  \frac{\omega^{(e)}}{n^{(e)}} \sum_{i=1}^{n^{(e)}} \{ y^{(e)}_i - \bbeta^\top \bx^{(e)}_i \}^2.
\end{align}
The \emph{environment invariant linear least squares (EILLS)} estimator $\hat{\bbeta}_{\mathsf{Q}}$ is defined by minimizing the following objective function:
\begin{align}
\label{eq:eills-objective-q}
    \hat{\mathsf{Q}}(\bbeta; \gamma, \bomega) = \hat{\mathsf{R}}(\bbeta;\bomega) + \gamma \hat{\mathsf{J}}(\bbeta;\bomega),
\end{align} which is a linear combination of the pooled empirical $L_2$ loss $\hat{\mathsf{R}}(\bbeta;\bomega)$ and focused linear invariance regularizer $\hat{\mathsf{J}}(\bbeta;\bomega)$ with weights $(1,\gamma)$ for some given hyper-parameter $\gamma \in \mathbb{R}^+$. We also define the population analogs of the EILLS objective function as follows:
\begin{align}
\label{eq:eills-objective-q-population}
    \mathsf{Q}(\bbeta;\gamma,\bomega) = \mathsf{R}(\bbeta; \bomega) + \gamma \mathsf{J}(\bbeta; \bomega) ~~~~ \text{with} ~~~~ \mathsf{R}(\bbeta; \bomega) = \sum_{e\in \mathcal{E}} \omega^{(e)} \mathbb{E} \left[|y^{(e)} - \bx^\top \bx^{(e)}|^2 \right].
\end{align}

The focused regularizer $\mathsf{J}(\bbeta; \bomega)$ can screen out all linear spurious variables when $\gamma$ is large enough.  This can lead to a low-dimensional problem and be sufficient for many applications. However, some linear exogenous variables that do not contribute to explaining $y$ can still survive after the screening. They can be eliminated further by introducing an $\ell_0$ penalty.
This leads us to considering the $\ell_0$ regularized EILLS estimator $\hat{\bbeta}_{\mathsf{L}}$ that minimizes the following objective:
\begin{align}
\label{eq:eills-objective-l}
    \hat{\mathsf{L}}(\bbeta; \lambda, \gamma, \bomega) = \hat{\mathsf{Q}}(\bbeta; \gamma, \bomega) + \lambda \|\bbeta\|_0 =  \hat{\mathsf{R}}(\bbeta;\bomega) +  \gamma \hat{\mathsf{J}}(\bbeta;\bomega)  +  \lambda \|\bbeta\|_0
\end{align} with given hyper-parameter $\lambda$.  This helps reduce variables that are uncorrelated to residuals and $y$.

At the methodology level, our method has a close connection to the FGMM estimator in \cite{fan2014endogeneity} and the invariant risk minimization \citep{arjovsky2019invariant} framework; see detailed connections and differences in \aosversion{Appendix~{\color{blue}D}}{\cref{sec:discussion}}. 

\medskip

\noindent \textbf{Practical Concerns} The current estimator has some combinatorial nature in optimization of \eqref{eq:invariant-regularizer-empirical} and involves an additional hyper-parameter $\gamma$. For the first concern, we show that one can use the Gumbel \citep{jang2016categorical} trick introduced by a follow-up work \cite{gu2024causality} to make variants of gradient descent continue to work in practice; see how to adopt Gumbel approximation to transform \eqref{eq:eills-objective-q} to a continuous optimization program and some simulation studies when $p=70$ in \aosversion{Appendix \blue{D.6}}{\cref{sec:gumbel}}. We also provide some guidance on how to tune the hyper-parameter $\gamma$ in practice; see \aosversion{Appendix \blue{D.7}}{\cref{sec:gamma-tune}}.

\section{Theory}
\label{sec:theory}

 To simplify the presentation, we consider the case of balanced data with equal weights, that is, $n^{(e)} \equiv n$ and $\omega^{(e)} \equiv 1/|\mathcal{E}|$, and defer the results of varying $(n^{(e)}, \omega^{(e)})$ to \aosversion{Appendix~\blue{A}}{\cref{sec:theory:general}}. 
Define the pooled covariance matrix $\bSigma = \frac{1}{|\mathcal{E}|} \sum_{e\in \mathcal{E}} \bSigma^{(e)}$. We first impose some regularity conditions for theoretical analysis.

\begin{condition}
\label{cond0-model-main}
    For each $e\in \mathcal{E}$,  $(\bx_1^{(e)}, y_1^{(e)}), \ldots, (\bx_{n}^{(e)}, y_{n}^{(e)})$ are i.i.d. copies of $(\bx^{(e)}, y^{(e)}) \sim \mu^{(e)}$, where $\mu^{(e)}$ belongs to $\mathcal{U}_{\bbeta^*, \sigma^2}$ for some $\sigma^2$. The data from different environments are also independent.  We set $\omega^{(e)} \equiv 1/|\mathcal{E}|$. 
\end{condition}

\begin{condition}
\label{cond1-well-condition-main}
    There exists some universal constants $\kappa_L \in (0,1]$ and $\kappa_U \in [1,\infty)$ such that
    \begin{align}
        \forall e\in \mathcal{E}, ~~~~~~~~~~ \kappa_L \bI_p \preceq \bSigma^{(e)} \preceq \kappa_U \bI_p.
    \end{align}
\end{condition}

\begin{condition}
\label{cond2-subgaussain-x-main}
    There exists some universal constant $\sigma_x \in [1,\infty)$ such that
    \begin{align}
        \forall e\in \mathcal{E}, \bv \in \mathbb{R}^p, ~~~~~~~~~~  \mathbb{E} \left[\exp\left\{\bv^\top {\bSigma}^{-1/2}\bx^{(e)}\right\}\right]  \le \exp\left(\frac{\sigma_x^2}{2} \cdot \|\bv\|_2^2\right).
    \end{align}
\end{condition}

\begin{condition}
\label{cond3-sub-gaussian-eps-main}
    There exists some universal constant $\sigma_\varepsilon \in \mathbb{R}^+$ such that,
    \begin{align}
        \forall e\in \mathcal{E}, \lambda \in \mathbb{R}, ~~~~~~~~~~ \mathbb{E} [e^{\lambda \varepsilon^{(e)}}] \le e^{\frac{1}{2} \lambda^2 \sigma_\varepsilon^2}.
    \end{align}
\end{condition}

\cref{cond0-model-main} is the basic setup of our multi-environment linear regression described in \cref{subsec:problem-formulation}. \cref{cond1-well-condition-main}--\ref{cond3-sub-gaussian-eps-main} are standard in high-dimensional linear regression analysis. We focus on the centered covariate case that $\mathbb{E}[\bx^{(e)}] = \bm{0}$, and it can be easily generalized to the non-centered counterpart.  The lower bound on the smallest eigenvalue, $\kappa_L$, is to establish non-asymptotic error bounds on $\ell_2$ norm $\|\hat{\bbeta}-\bbeta^*\|_2$. To simplify expressions in our results, we set $\kappa_U \land \sigma_x \ge 1$ and $\kappa_L \le 1$, thus avoiding repeated use of $(\kappa_U \lor 1)$, $(\sigma_x \lor 1)$, and $(\kappa_L^{-1} \lor 1)$. We adopt simple sub-Gaussian conditions to convey our main message. It is possible to relax the above conditions and derive a similar result. For example, one can replace the sub-Gaussian condition of the noise $\varepsilon^{(e)}$ with the sub-Weibull condition \citep{vladimirova2020sub}. One can also substitute the joint sub-Gaussian condition of the covariate with the marginal sub-Weibull condition and show that the EILLS objective \eqref{eq:eills-objective-q} with folded concave penalty function introduced by \cite{fan2001variable} has certain oracle property.

\subsection{Pooled Linear Spurious Variables and the Bias of Pooled Least Squares}
\label{subsec:linear-spurious-variable}

We first define \emph{pooled linear spurious variables} and demonstrate that the vanilla pooled least squares method is an inconsistent estimator in the presence of pooled linear spurious variables.

\begin{definition}[Pooled Linear Spurious Variables] We let $G$ be the index set of all pooled linear spurious variables in environments $\mathcal{E}$ concerning the uniform weights $\omega^{(e)}\equiv 1/|\mathcal{E}|$, that is,
        $G = \{j \in [p]:  \sum_{e\in \mathcal{E}} \mathbb{E}[x_j^{(e)}\varepsilon^{(e)}] \neq 0 \}$. We say $x_j$ is a pooled linear spurious variable if $j\in G$.
\end{definition}

The following proposition characterizes the properties of pooled least squares.

\begin{proposition}[Properties of Pooled Least Squares]
\label{prop:pooled-least-squares-main}
 Assume Conditions \ref{cond0-model-main}--\ref{cond3-sub-gaussian-eps-main} hold. Then, there exists some $\bar{\bbeta}^{\mathsf{R}} \in \mathbb{R}^p$ satisfying $\frac{1}{\kappa_U} \|\frac{1}{|\mathcal{E}|}\sum_{e\in \mathcal{E}} \mathbb{E} [\varepsilon^{(e)} \bx^{(e)}] \|_2 \le \|\bar{\bbeta}^{\mathsf{R}} - \bbeta^*\|_2 \le \frac{1}{\kappa_L} \|\frac{1}{|\mathcal{E}|}\sum_{e\in \mathcal{E}} \mathbb{E} [\varepsilon^{(e)} \bx^{(e)}] \|_2$ 
    such that, for any $\bbeta \in \mathbb{R}^p$,
    \begin{align}
    \label{eq:pooled-least-square-strong-convexity-main}
        \mathsf{R}(\bbeta) - \mathsf{R}(\bar{\bbeta}^{\mathsf{R}}) = \|\bSigma^{1/2}(\bbeta - \bar{\bbeta}^{\mathsf{R}})\|_2^2.
    \end{align} Moreover, there exist universal constants $c_1$ and $c_2$ such that if $n\cdot |\mathcal{E}| \ge c_1 \sigma_x^4(p + t)$, then the pooled least      squares estimator $\hat{\bbeta}_{\mathsf{R}}$ minimizing \eqref{eq:pooled-l2-risk-empirical} satisfies
    \begin{align*}
        \|{\bSigma}^{1/2}(\hat{\bbeta}_{\mathsf{R}} - \bar{\bbeta}^{\mathsf{R}})\|_2 \le c_2 \sigma_x \left(\sigma_\varepsilon + \sigma_x \|{\bSigma}^{1/2}(\bar{\bbeta}^{\mathsf{R}} - \bbeta^*)\|_2\right) \sqrt{\frac{p + t}{{ n\cdot |\mathcal{E}|}}}
    \end{align*} with probability at least $1-2e^{-t}$.
\end{proposition}

From \cref{prop:pooled-least-squares-main}, we observe that $\delta_{\mathtt{b}}=\|\bar{\bbeta}^{\mathsf{R}} - \bbeta^*\|_2$, which can be interpreted as the bias of the pooled least squares, satisfies $\delta_{\mathtt{b}} \asymp \left\||\mathcal{E}|^{-1} \sum_{e\in \mathcal{E}}  \mathbb{E} [\varepsilon^{(e)} \bx^{(e)}] \right\|_2$. \cref{prop:pooled-least-squares-main} thus indicates that, for large enough $n$,
\begin{align*}
    \|\hat{\bbeta}_{\mathsf{R}} - \bbeta^*\|_2 \asymp \left\|\frac{1}{|\mathcal{E}|}\sum_{e\in \mathcal{E}} \mathbb{E}[\varepsilon^{(e)} \bx^{(e)}]\right\|_2.
\end{align*} 
Therefore, the pooled least squares can be a fair estimate of $\bbeta^*$ only when the biases from each environment happen to cancel: $\sum_{e\in \mathcal{E}} \mathbb{E} [\varepsilon^{(e)} x_j^{(e)}] = o(|\mathcal{E}|)$ for any $j\in \mathcal{E}$. Moreover, the strong convexity with respect to $\bar{\bbeta}^{\mathsf{R}}$ \eqref{eq:pooled-least-square-strong-convexity-main} suggests that this issue persists for penalized least squares.

\subsection{Local Strong Convexity for Population Loss}
\label{subsec:strong-convexity}

We first examine the population EILLS loss \eqref{eq:eills-objective-q-population} when $G\neq \emptyset$. Specifically, we show that with a large enough $\gamma$, the population EILLS loss satisfies certain strong convexity with respect to $\bbeta^*$ under the following identification condition.

\begin{condition}[Identification]
\label{cond-ident-main}
    For any $S\subseteq [p]$ satisfying $S \cap G \neq \emptyset$, there exists some $e, e'\in \mathcal{E}$ such that $\bbeta^{(e, S)} \neq \bbeta^{(e',S)}$, where $\bbeta^{(e,S)}$ is defined in \eqref{eq:local-minimizer}.
\end{condition}

\begin{remark}[Near Minimal Identification Condition]
It is worth noticing that the above identification condition is minimal if only linear information is considered. To see this, consider the case where $(\bx^{(e)}, y^{(e)})$ are all multivariate normal distributions, a violation of \cref{cond-ident-main} implies that there exist some $\tilde{\bbeta}$ with support set $\tilde{S}$ containing variables outside $S^\star$ such that
\begin{align*}
    \forall e\in \mathcal{E}, ~~~~~ \mathbb{E}[y^{(e)}|\bx_{\tilde{S}}^{(e)}] \equiv (\tilde{\bbeta}_{\tilde{S}})^\top \bx_{\tilde{S}}^{(e)}.
\end{align*} 
In this case, it is impossible to determine which one among $S^*$ and $\tilde{S}$ is the true important variable set because they are all CE-invariant across $\mathcal{E}$ and $\tilde{S}\setminus S^*\neq \emptyset$. 
\end{remark}

\begin{theorem}[Strong Convexity with respect to $\bbeta^*$] 
\label{thm:global-strong-convexity-main}
Assume Conditions \ref{cond0-model-main}--\ref{cond1-well-condition-main} and \ref{cond-ident-main} hold. Then $\bbeta^*$ is the unique minimizer of $\mathsf{Q}(\bbeta;\gamma, \bomega)$ for large enough $\gamma$: for any $\epsilon \in (0, 1)$ and any $\gamma\ge \epsilon^{-1} \gamma^*$ with
\begin{align}
\label{eq:thm1-gamma-star-main}
    \gamma^* = (\kappa_L)^{-3} \sup_{S: S\cap G \neq \emptyset} \left(\mathsf{b}_S/\bar{\mathsf{d}}_S\right),
\end{align} where $\mathsf{b}_S = \|\frac{1}{|\mathcal{E}|}\sum_{e\in \mathcal{E}} \mathbb{E} [\varepsilon^{(e)} \bx_S^{(e)}] \|^2_2$ and $\bar{\mathsf{d}}_S = \sum_{e\in \mathcal{E}} \frac{1}{|\mathcal{E}|} \|\bbeta^{(e,S)} - \bar{\bbeta}^{(S)}\|_2^2$
 with $\bar{\bbeta}^{(S)}=\frac{1}{|\mathcal{E}|} \sum_{e'\in \mathcal{E}} \bbeta^{(e',S)}$, we have
\begin{align}
\label{eq:thm1-jsc-condition-main}
    \mathsf{Q}(\bbeta;\gamma, \bomega) - \mathsf{Q}(\bbeta^*;\gamma, \bomega) \ge (1-\epsilon) \|\bSigma^{1/2}(\bbeta - \bbeta^*)\|_2^2 + \kappa_L^2 (\gamma - \epsilon^{-1}\gamma^*) \bar{\mathsf{d}}_{\supp(\bbeta)}.
\end{align}
\end{theorem}

Note $\bar{\mathsf{d}}_S>0$ for any $S\cap G \neq \emptyset$ under \cref{cond-ident-main}. \cref{thm:global-strong-convexity-main} provides a generic condition under which $\bbeta^*$ is the global optimal of the population-level EILLS objective \eqref{eq:eills-objective-q-population} with a non-asymptotic critical threshold for $\gamma$. Moreover, a detailed characterization of global optimality \eqref{eq:thm1-jsc-condition-main}, which is slightly stronger than strong convexity, is also established.  The strong convexity with respect to the target function $\bbeta^*$ lays the foundation to derive faster convergence rate to $\bbeta^*$ \citep{van1996weak, wainwright2019high}.

\begin{remark}[Interpretation of the Quantities $\mathsf{b}_S$, $\bar{\mathsf{d}}_S$]
Observe that when $S\supseteq S^*$, the bias of the least squares solution in one environment $e\in \mathcal{E}$ is $\bDelta^{(e)} = \bbeta^{(e,S)} - \bbeta^*$. Here we refer to $\mathsf{b}_S$ as bias mean because $\mathsf{b}_S$ is the bias of running least squares on $X_S$ using all the data when $S\supseteq S^*$, that is, $\|\hat{\bbeta}_{\mathsf{R}, S} - \bbeta^*\|_2^2 \asymp \mathsf{b}_S$ by  \cref{prop:pooled-least-squares-main}. At the same time, we have $\bar{\mathsf{d}}_S = \frac{1}{|\mathcal{E}|} \sum_{e\in \mathcal{E}} \|\bDelta^{(e)} - (|\mathcal{E}|^{-1}\sum_{e'\in \mathcal{E}} \bDelta^{(e')})\|_2^2$ provided $S\supseteq S^*$. Thus, the quantity $\bar{\mathsf{d}}_S$ can be interpreted as the variance of bias since it measures the variations of the biases $\bDelta^{(e)}$ among different environments.
\end{remark}

\begin{remark}[Interpretation of the Critical Threshold $\gamma^*$] 
\cref{thm:global-strong-convexity-main} implies that the global optimality of $\bbeta^*$ is guaranteed when $\gamma > \gamma^*$. $\gamma^*$ is the ratio of bias mean to bias variance and will affect both the forthcoming estimation error and variable selection property. As an intuitive example, for the previous cow/camel thought experiment, a reasonable value for $\gamma^*$ is expected when the fractions of cows on grass and camels on sand are both 90\% in $e=0$ and 60\% in $e=1$. Similarly, a moderate $\gamma^*$ is anticipated when fractions are 51\% and 53\%. However, with fractions 90\% and 89\%, $\gamma^*$ is significantly larger, necessitating more data for both accurate estimation and variable selection. 
\end{remark}

\begin{remark}[Interpretation of Small $\gamma$]
In \cref{thm:global-strong-convexity-main}, we show that the invariant parameter $\beta^\star$ will uniquely minimize the EILLS objective when $\gamma$ is larger than some threshold $\gamma^\star$. Yet, it is unclear what is the regularization effect of $\gamma \mathsf{J}$ when $\gamma$ is small. We provide some intuitions on the impact of the regularizer when $\gamma$ is small in \aosversion{Appendix~\blue{D.8}}{\cref{sec:small-gamma}}. 
\end{remark}

We use the toy example below to demonstrate (1) when \cref{cond-ident-main} holds, and (2) how $\gamma^*$ scales in a concrete model, and leave more examples and discussions in \aosversion{Appendix~\blue{A.4}}{\cref{sec:moreex}}.

\begin{example}
\label{ex-h-de}
    Consider the following two-environment linear SCMs with $p=2$:
    \begin{align*}
        x^{(e)}_1 &\gets \sqrt{0.5} \cdot \varepsilon_1 \\
        y^{(e)} &\gets 1\cdot x^{(e)}_1 + \sqrt{0.5} \cdot \varepsilon_0 \\
        x^{(e)}_2 &\gets  {s^{(e)}} \cdot y^{(e)} + \varepsilon_2.  
    \end{align*} Here $\varepsilon_0, \varepsilon_1, \varepsilon_2$ are independent, standard normally distributed exogenous variables. The linear model is $y^{(e)} = (\bbeta^*)^\top \bx^{(e)} + \varepsilon^{(e)}$ with $\bbeta^*=(1,0)$ and $\varepsilon^{(e)} = \sqrt{0.5} \varepsilon_0$. Note that $\mathbb{E}[\varepsilon^{(e)} x_2^{(e)}] = 0.5 s^{(e)}$. When $\omega^{(e)}\equiv 1/2$, the variable $x_2$ will be a pooled linear spurious variable if $s^{(1)}+s^{(2)} \neq 0$. 
    
{\normalfont We focus on the case where the pooled least squares is not consistent, i.e., $s^{(1)} + s^{(2)} \neq 0$. In the well-conditioned regime where $|s^{(1)}| + |s^{(2)}| = O(1)$, by some calculations, we have
    \begin{align*}
        \sqrt{\frac{\mathsf{b}_{\{2\}}}{\bar{\mathsf{d}}_{\{2\}}}} \asymp \frac{|s^{(1)}+s^{(2)}|}{|s^{(2)} - s^{(1)}|} \cdot \frac{1}{|1 - s^{(1)} s^{(2)}|}~~~~~~\text{and}~~~~~~\sqrt{\frac{\mathsf{b}_{\{1,2\}}}{\bar{\mathsf{d}}_{\{1,2\}}}} \asymp \frac{|s^{(1)} + s^{(2)}|}{|s^{(2)} - s^{(1)}|}.
    \end{align*}

First, the \cref{cond-ident-main} holds if $s^{(1)}\neq s^{(2)}$ and $s^{(1)} s^{(2)} \neq 1$. The first inequality is easy to understand, without which the underlying distributions in the two environments are identical. The second inequality is strange at first glance. However, we have $\mathbb{E}[y^{(e)}|x_2^{(e)}] \equiv \frac{s}{s^2+1} x_2^{(e)}$ when $s^{(1)} = s = 1/s^{(2)}$. In this case, it is impossible to identify which of $\{1\}$ and $\{2\}$ are the true important variable set because all the sets are CE-invariant across $\mathcal{E}$. This also demonstrates the necessity of taking the supremum over all sets $S$ with $S\cap G \neq \emptyset$ in \eqref{eq:thm1-gamma-star-main}.

When $s^{(1)} s^{(2)}$ is away from $1$, the critical threshold $\gamma^*$ satisfies $(\gamma^*)^{1/2}\asymp |s^{(1)}+s^{(2)}|/|s^{(2)}-s^{(1)}|$, where the numerator quantifies the strength of spuriousness, and the denominator quantifies the strength of heterogeneity. Hence, a constant-level $\gamma$ can be adopted when the strength of heterogeneity is of the same order as the strength of spuriousness.
}\qed
\end{example}

\subsection{Statistical Analysis of the EILLS Estimator in the Low-dimensional Regime}
\label{subsec:theory-low-dim}

Our statistical analysis of the EILLS estimator focuses on the regime where it is possible to identify $\bbeta^*$, i.e., \cref{cond-ident-main} holds, and when our choice of $\gamma$ satisfies $\gamma \ge 3\gamma^* \lor 1$, where $\gamma^*$ is the quantity defined in  \cref{thm:global-strong-convexity-main}. We let $\gamma \ge 1$ to simplify the presentation. We are now ready to provide a statistical analysis of the EILLS estimator $\hat{\bbeta}_{\mathsf{Q}}$ minimizing \eqref{eq:eills-objective-q}.  The first result is about the sure screening with false positive control.
\begin{theorem}[Non-asymptotic Variable Selection Property]
\label{thm:rate-low-dim-main} Define
\begin{align}
\label{eq:thm2-signal-main}
    \mathsf{s}_+ = \min_{j\in S^*} |\beta_j^*|^2 \qquad \text{and} \qquad \mathsf{s}_- = \min_{S\subseteq [p], S\cap G \neq \emptyset} \bar{\mathsf{d}}_S
\end{align} 
Suppose Conditions \ref{cond0-model-main}--\ref{cond-ident-main} hold, and we choose $\gamma \ge 3\gamma^*\lor 1$ where $\gamma^*$ is defined in  \cref{thm:global-strong-convexity-main}. There exists some universal constants $c_1$--$c_2$ that only depends on $(\kappa_U, \sigma_x, \sigma_\varepsilon)$ such that for any $t>0$, if $n \ge c_1(\gamma/\kappa_L)(p+ \log(|\mathcal{E}|) + t)\{\mathsf{s}_+^{-0.5}+\mathsf{s}_+^{-1}+(\gamma\kappa_L\mathsf{s}_-)^{-0.5}\}$, and $n \cdot |\mathcal{E}| \ge c_2(\gamma/\kappa_L)^2(p+t)\{\mathsf{s}_+^{-1} + (\gamma\kappa_L\mathsf{s}_-)^{-1} + 1\}$, then the EILLS estimator $\hat{\bbeta}_{\mathsf{Q}}$ minimizing \eqref{eq:eills-objective-q} satisfies
\begin{align}
\label{eq:low-dim-variable-selection-main}
    \mathbb{P}\left[S^* \subseteq \supp(\hat{\bbeta}_{\mathsf{Q}}) \subseteq G^c \right] \ge 1-7e^{-t}.
\end{align}
\end{theorem}

Equation \eqref{eq:low-dim-variable-selection-main} reveals that all endogenous variables are screened out when $\gamma$ is sufficiently large. When the choice of $\gamma$ and the curvature $\kappa_L$ are both of constant order, \cref{thm:rate-low-dim-main} implies that with high probability, the EILLS estimator $\hat{\bbeta}_{\mathsf{Q}}$ can eliminate all the pooled linear spurious variables while keeping all the important variables, i.e., \eqref{eq:low-dim-variable-selection-main} holds, if 
\begin{align}
\label{eq:low-dim-n-m-condition-main}
    \frac{n}{p + \log(|\mathcal{E}|)} \gg \mathsf{s}_+^{-1} + \mathsf{s}_{-}^{-1/2} \left(1\lor \frac{\mathsf{s}_{-}^{-1/2}}{|\mathcal{E}|}\right).
\end{align}  Here, the quantities $\mathsf{s}_+$ and $\mathsf{s}_-$ defined in \eqref{eq:thm2-signal-main} can be interpreted as the signal of true important variables and the signal of heterogeneity, respectively. One can expect more data to differentiate whether it is a signal or noise when one of $\mathsf{s}_-$ and $\mathsf{s}_+$ is small.

When the variable selection property \eqref{eq:low-dim-variable-selection-main} is satisfied, $\hat{S}$ does not contain any pooled linear spurious variables while maintaining all the important variables. In this case, one can provide a good estimate of $\bbeta^*$ by running another least squares constrained on $\hat{S} = \supp(\hat{\bbeta}_{\mathsf{Q}})$ without regularization. In this case, there is no bias anymore. It then follows from \cref{prop:pooled-least-squares-main} that, with high probability, we have the $\ell_2$-error $\{|G^c|/(n \cdot |\mathcal{E}|)\}^{1/2}$ given $\hat{S}\subseteq G^c$. 

When some weak spurious variables exist such that the corresponding absolute value of $\bar{\mathsf{d}}_S$ is small, the EILLS estimator requires a large $n$ to eliminate all these spurious variables. To see this, suppose there exists a weak spurious variable $x_j$ such that $|\mathbb{E}[x_j^{(e)}\varepsilon^{(e)}]| \le \epsilon$ for all the $e\in \mathcal{E}$. Consider the set $\tilde{S} = S^* \cup \{j\}$, it is easy to verify that $\mathsf{b}_{\tilde{S}} \le \varepsilon^2$ and $\bar{\mathsf{d}}_{\tilde{S}} \le \varepsilon^2$. In this case, we require $n\gg n_{*,\mathtt{sel}} \asymp p\epsilon^{-1}( \epsilon^{-1}/\mathcal{|E|} + 1)$ to eliminate variable $x_j$ by \eqref{eq:low-dim-n-m-condition-main}. The next theorem claims that a small $\ell_2$ estimation error $\|\hat{\bbeta}_{\mathsf{Q}} - \bbeta^*\|_2$ can be obtained in the regime where $p \ll n \ll n_{*,\mathtt{sel}}$ regardless of whether EILLS selects the correct variable.

\begin{theorem}[Non-asymptotic $\ell_2$ Error Bound] 
\label{thm:rate-faster-low-dim-main}
Assume Conditions \ref{cond0-model-main}--\ref{cond-ident-main} hold, and we choose $\gamma \ge 3\gamma^*\lor 1$ where $\gamma^*$ is defined in \cref{thm:global-strong-convexity-main}. There exists some universal constants $c_1$--$c_4$ that only depend on $(\kappa_U, \sigma_x)$ such that for any $t>0$, if $n \ge p + \log(2|\mathcal{E}|) + t$ and $n \cdot |\mathcal{E}| \ge c_1 (\gamma/\kappa_L)(p + t)$, then $\hat{\bbeta}_{\mathsf{Q}}$ minimizing \eqref{eq:eills-objective-q} satisfies
\begin{align}
\label{eq:low-dim-rate2}
    \frac{\|\hat{\bbeta}_{\mathsf{Q}} - \bbeta^*\|_2}{\sigma_\varepsilon (\gamma/\kappa_L)}  &\le c_2 \left(\sqrt{\frac{p+t}{n \cdot |\mathcal{E}|}} + \frac{p+\log(|\mathcal{E}|)+t}{n}\right) + c_3 \frac{\sqrt{|S^*|}}{n} \cdot \frac{\log(|\mathcal{E}||S^*|)+t}{\min_{j\in S^*}|\beta^*_j|}
\end{align} with probability at least $1-7e^{-t}$. Moreover, when the additional conditions in \cref{thm:rate-low-dim-main} hold, then
\begin{align}
\label{eq:low-dim-rate1}
    \frac{\|\hat{\bbeta}_{\mathsf{Q}} - \bbeta^*\|_2}{\sigma_\varepsilon (\gamma/\kappa_L)}  &\le c_2 \left(\sqrt{\frac{|{G}^c|+t}{n \cdot |\mathcal{E}|}} + \frac{|{G}^c|+\log(|\mathcal{E}|)+t}{n}\right)
\end{align} occurs with probability at least $1-14e^{-t}$.
\end{theorem}

In the well-conditioned regime where $\min_{j\in S^*} |\beta_j^*| \gtrsim p^{-1/2}$, $\kappa_L\gtrsim 1$, and $\gamma^*\asymp 1$, one can adopt a constant-level hyper-parameter $\gamma$ such that
\begin{align*}
    \|\hat{\bbeta}_{\mathsf{Q}} - \bbeta^*\|_2 \lesssim \sqrt{\frac{p}{n\cdot |\mathcal{E}|}} + \frac{p}{n}
\end{align*} with high probability provided $n \gtrsim p$. When the first term dominates ($|\mathcal{E}|$ is not too big), the EILLS estimator achieves an optimal linear regression rate. This implies that the EILLS estimator can estimate $\bbeta^*$ well in $\ell_2$ error when there are some weak spurious variables and the number of data points is not large enough to eliminate all the variables in $G$. 

One technical novelty behind Theorems \ref{thm:rate-low-dim-main}--\ref{thm:rate-faster-low-dim-main} and \cref{thm:high-dim-main} later is to apply the localization argument in the analysis of the invariance regularizer to get a faster rate. To see this, we have $\|\hat{\bbeta}_{\mathsf{Q}} - \bbeta^*\|_2^2 \asymp \delta_{n,\mathsf{Q}}^2 \asymp (n\cdot |\mathcal{E}|)^{-1} + n^{-2}$ when $p=O(1)$. The first term is faster than $(n\cdot |\mathcal{E}|)^{-1/2}$ which directly applies uniform concentration, and the second term is faster than $n^{-1} \asymp \sup_{\|\bbeta\|_2 \le 1} |\mathsf{J}(\bbeta) - \mathbb{E}[\hat{\mathsf{J}}(\bbeta)]|$, the (uniform) bias of the invariance regularizer. We use \aosversion{Lemma~\blue{C.4}}{\cref{lemma:instance-dependent-bound-j-lowdim}}, a novel one-side instance-dependent error bound for $\mathsf{J}(\bbeta) - \hat{\mathsf{J}}(\bbeta) + \hat{\mathsf{J}}(\bbeta^*) - \mathsf{J}(\bbeta^*)$ to obtain such a faster rate; see \aosversion{Appendix~\blue{C.3}}{\cref{sec:proof-thm2}}. 

\subsection{Variable Selection Consistency in the High-dimensional Regime}
\label{subsec:theory-high-dim}

 In the high-dimensional regime, we further define $s^* = |S^*|$, $\beta_{\min} = \min_{j\in S^*} |\beta_j^*|$. We need a condition asserting that the sample size $n$ should be large enough for the given hyper-parameter $\gamma$.

\begin{condition}
\label{cond6-n-req-main}
    Suppose that $\log(|\mathcal{E}|)\le C\log p$ and that 
    
\noindent (1) $n \ge c_1 (\gamma/\kappa_L) \left\{(s^* +  \beta_{\min}^{-2})\log p + (\kappa_L\beta_{\min})^{-1} \sqrt{(s^* + \log p) s^*\log p}\right\}$
    
\noindent (2) $n \cdot |\mathcal{E}| \ge c_2 (\gamma/\kappa_L)^2 (s^* \log p ) \{1 + 1/ (\kappa_L \beta_{\min}^2 ) \}$, 
    
\noindent Here $c_1$--$c_2$ are positive universal constants that depend only on $(C,\sigma_x,\kappa_U, \sigma_\varepsilon)$.
\end{condition} 

\begin{theorem}[Variable Selection Consistency in High Dimensions]  
\label{thm:high-dim-main}  Assume Conditions \ref{cond0-model-main}--\ref{cond6-n-req-main} hold, and we choose $\gamma \ge 3\gamma^* \lor 1$ where $\gamma^*$ is defined in \cref{thm:global-strong-convexity-main}. Suppose further that the choice of $\lambda$ satisfies
\begin{align*}
    c_1 \left\{ (\gamma/\kappa_L)^2 \frac{s^*\log p}{n \cdot |\mathcal{E}|} + \epsilon(n) \right\} \le \lambda \le c_2 \kappa_L \beta_{\min}^2,
\end{align*} where $\epsilon(n) = (\gamma/\kappa_L)^2 s^\star (\log p)(s^* + \log p) / n^2 + (\gamma/\kappa_L) \log p \sqrt{n^{-3}(s^* + \log p)}$, and  $c_1, c_2$ are some universal positive constants only depends on $(C,\kappa_U, \sigma_x,\sigma_\varepsilon)$. Then the $\ell_0$ regularized EILLS estimator $\hat{\bbeta}_{\mathsf{L}}$ minimizing \eqref{eq:eills-objective-l} satisfies $\mathbb{P}[\supp(\hat{\bbeta}_{\mathsf{L}}) = S^*] \ge 1-p^{-10}$.
\end{theorem}

We can see that $\epsilon(n)$ is negligible when $|\mathcal{E}| \asymp 1$ and $s^*+\log p = o(n)$ because 
\begin{align*}
    \epsilon(n) \Big/ \left\{ (\gamma/\kappa_L)^2\frac{s^* \log p}{n \cdot |\mathcal{E}|}\right\} \lesssim |\mathcal{E}| \left\{ \frac{s^*+\log p}{n} + 
    \left(\sqrt{\frac{\log p}{n}}\right) \right\} = o(1).
\end{align*} Therefore, in the regime where $\gamma, \kappa_L, |\mathcal{E}|$ are fixed while $n, p, s^*$ may grow, the variable selection consistency can be achieved when 
\begin{align*}
\frac{s^* \log p}{n}\ll \lambda \ll \beta_{\min}^2.
\end{align*}  Recall that for standard $\ell_0$ regularized least squares with sample size $n$, the variable selection consistency can be obtained when $n \gg (s^* + \beta_{\min}^{-2})\log p$ and $\lambda$ satisfies $n^{-1} \log p \ll \lambda \ll \beta_{\min}^2$ \citep{zhang2012general}. Compared to the standard $\ell_0$ regularized least squares, the EILLS estimator needs $n \gg s^* \beta_{\min}^{-2} \log p$ in the fixed number of environments setting.

Another question is how much can the $\ell_0$ regularized EILLS estimator benefit from growing $|\mathcal{E}|$? When $\gamma \asymp\kappa_L \asymp s^*\asymp 1$, \cref{thm:high-dim-main} implies that though achieving variable selection consistency still needs $(1+\beta_{\min}^{-2})\log p = o(n)$ due to \cref{cond6-n-req-main} (1), we can choose a wider range of $\lambda$. To be specific, in this case, the variable selection consistency can be achieved when $\lambda$ satisfies
\begin{align*}
    \left\{\frac{1}{|\mathcal{E}|} + \left(\frac{\log p}{n}\right)^{1/2} \right\} \frac{\log p}{n} \ll \lambda \ll \beta_{\min}^2.
\end{align*} Hence we can choose any $(n\cdot |\mathcal{E}|)^{-1} \log p \ll \lambda \ll \beta_{\min}^2$ in the regime $(\log p) |\mathcal{E}|^2 = o(n)$. This is the same with running $\ell_0$ regularized least squares on total $n\cdot |\mathcal{E}|$ data when $\mathbb{E}[\varepsilon|\bx]\equiv 0$.

With the variable selection consistency given in \cref{thm:high-dim-main}, we can then attain the optimal $\ell_2$-rate by running the least squares on $\supp(\hat{\bbeta}_{\mathsf{L}})$, as if assisted by an oracle.

\section{An Illustration by Structural Causal Model}
\label{sec:simulation}

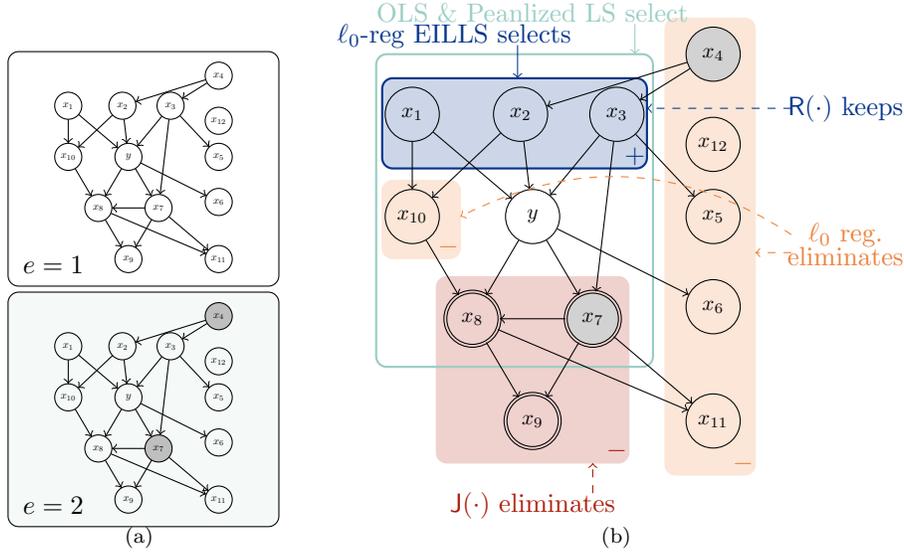
\begin{figure}[!t]
\centering
\begin{tabular}{cc}
\subfigure[]{
\begin{tikzpicture}[scale=0.4, state/.style={circle, draw, minimum size=0.9cm, scale=0.4}]

\draw[black, rounded corners] (-2, -6) rectangle (7, 1.8);
\draw node at (-0.5, -5.3) {$e=1$};
\node[state] at (0, 0) (x1) {$x_1$};
\node[state] at (2, -1.7) (y) {$y$};
\node[state] at (1.8, 0) (x2) {$x_2$};
\node[state] at (3.4, 0) (x3) {$x_3$};
\node[state] at (5, 1) (x4) {$x_4$}; 
\node[state] at (5, -1.7) (x5) {$x_5$};
\node[state] at (5, -3.2) (x6) {$x_6$};
\node[state] at (3, -3.4) (x7) {$x_7$};
\node[state] at (1, -3.4) (x8) {$x_8$};
\node[state] at (0, -1.7) (x10) {$x_{10}$};
\node[state] at (2, -5.1) (x9) {$x_{9}$};
\node[state] at (5, -5.1) (x11) {$x_{11}$};
\node[state] at (5, -0.5) (x12) {$x_{12}$};

\draw[->] (x1) -- (y);
\draw[->] (x2) -- (y);
\draw[->] (x3) -- (y);
\draw[->] (x1) -- (x10);
\draw[->] (x2) -- (x10);
\draw[->] (x3) -- (x5);
\draw[->] (x4) -- (x2);
\draw[->] (x4) -- (x3);
\draw[->] (y) -- (x6);
\draw[->] (y) -- (x7);
\draw[->] (y) -- (x8);
\draw[->] (x7) -- (x8);
\draw[->] (x10) -- (x8);
\draw[->] (x7) -- (x9);
\draw[->] (x8) -- (x9);
\draw[->] (x7) -- (x11);
\draw[->] (x8) -- (x11);
\draw[->] (x3) -- (x7);

\draw[black, rounded corners, fill=mylightblue!10] (-2, -14) rectangle (7, -6.2);
\draw node at (-0.5, -13.3) {$e=2$};
\node[state] at (0, -8) (e2x1) {$x_1$};
\node[state] at (2, -9.7) (e2y) {$y$};
\node[state] at (1.8, -8) (e2x2) {$x_2$};
\node[state] at (3.4, -8) (e2x3) {$x_3$};
\node[state, fill=gray!50] at (5, -7) (e2x4) {$x_4$}; 
\node[state] at (5, -9.7) (e2x5) {$x_5$};
\node[state] at (5, -11.2) (e2x6) {$x_6$};
\node[state, fill=gray!50] at (3, -11.4) (e2x7) {$x_7$};
\node[state] at (1, -11.4) (e2x8) {$x_8$};
\node[state] at (0, -9.7) (e2x10) {$x_{10}$};
\node[state] at (2, -13.1) (e2x9) {$x_{9}$};
\node[state] at (5, -13.1) (e2x11) {$x_{11}$};
\node[state] at (5, -8.5) (e2x12) {$x_{12}$};

\draw[->] (e2x1) -- (e2y);
\draw[->] (e2x2) -- (e2y);
\draw[->] (e2x3) -- (e2y);
\draw[->] (e2x1) -- (e2x10);
\draw[->] (e2x2) -- (e2x10);
\draw[->] (e2x3) -- (e2x5);
\draw[->] (e2x4) -- (e2x2);
\draw[->] (e2x4) -- (e2x3);
\draw[->] (e2y) -- (e2x6);
\draw[->] (e2y) -- (e2x7);
\draw[->] (e2y) -- (e2x8);
\draw[->] (e2x7) -- (e2x8);
\draw[->] (e2x10) -- (e2x8);
\draw[->] (e2x7) -- (e2x9);
\draw[->] (e2x8) -- (e2x9);
\draw[->] (e2x7) -- (e2x11);
\draw[->] (e2x8) -- (e2x11);
\draw[->] (e2x3) -- (e2x7);

\end{tikzpicture}   
} & 
\subfigure[]{
\begin{tikzpicture}[scale=0.8, state/.style={circle, draw, minimum size=0.9cm, scale=0.8}]

\draw[myorange!20, fill=myorange!20, rounded corners] (4.2, -6) rectangle (5.7, 1.6);
\draw[myorange!20, fill=myorange!20, rounded corners] (-0.5, -2.4) rectangle (0.8, -1.1);
\draw[myred!20, fill=myred!20, rounded corners] (0.4, -5.8) rectangle (3.6, -2.7);
\draw[myblue, thick, fill=myblue!20, rounded corners] (-0.5, -0.9) rectangle (3.9, 0.6);
\draw[mylightblue, thick, rounded corners] (-0.6, -4.2) rectangle (4, 1);

\node[state] at (0, 0) (x1) {$x_1$};
\node[state] at (2, -1.7) (y) {$y$};
\node[state] at (1.8, 0) (x2) {$x_2$};
\node[state] at (3.4, 0) (x3) {$x_3$};
\node[state, fill=gray!35] at (5, 1) (x4) {$x_4$}; 
\node[state] at (5, -1.7) (x5) {$x_5$};
\node[state] at (5, -3.2) (x6) {$x_6$};
\node[state, double=myred!20, double distance=0.7pt, fill=gray!35] at (3, -3.4) (x7) {$x_7$};
\node[state, double=myred!20, double distance=0.7pt] at (1, -3.4) (x8) {$x_8$};
\node[state] at (0, -1.7) (x10) {$x_{10}$};
\node[state, double=myred!20, double distance=0.7pt] at (2, -5.1) (x9) {$x_{9}$};
\node[state] at (5, -5.1) (x11) {$x_{11}$};
\node[state] at (5, -0.5) (x12) {$x_{12}$};

\draw[myblue] node at (3.7, -0.7) {$+$};
\draw[myorange] node at (5.5, -5.8) {$-$};
\draw[myorange] node at (0.6, -2.2) {$-$};
\draw[myred] node at (3.4, -5.6) {$-$};

\draw[->] (x1) -- (y);
\draw[->] (x2) -- (y);
\draw[->] (x3) -- (y);
\draw[->] (x1) -- (x10);
\draw[->] (x2) -- (x10);
\draw[->] (x3) -- (x5);
\draw[->] (x4) -- (x2);
\draw[->] (x4) -- (x3);
\draw[->] (y) -- (x6);
\draw[->] (y) -- (x7);
\draw[->] (y) -- (x8);
\draw[->] (x7) -- (x8);
\draw[->] (x10) -- (x8);
\draw[->] (x7) -- (x9);
\draw[->] (x8) -- (x9);
\draw[->] (x7) -- (x11);
\draw[->] (x8) -- (x11);
\draw[->] (x3) -- (x7);

\draw[mylightblue] (2, 1.7) node{OLS \& Peanlized LS select};
\draw[->, mylightblue] (3.7, 1.5) -- (3.7, 1);

\draw[myblue] (0.7, 1.3) node{$\ell_0$-reg EILLS selects};
\draw[->, myblue] (1.75, 1.15) -- (1.75, 0.6);

\draw[myred] (2, -6.5) node{$\mathsf{J}(\cdot)$ eliminates};
\draw[->, dashed, myred] (3, -6.3) to (3, -5.8);

\draw[myorange] (7.2, -2) node{$\ell_0$ reg.};
\draw[myorange] (7.2, -2.35) node{eliminates};
\draw[->, dashed, myorange] (6.3,-2.3) to (5.7, -2.3);
\draw[->, dashed, myorange] (6.3,-2) to[bend right] (0.8, -1.7);

\draw[myblue] (7.2, 0.1) node{$\mathsf{R}(\cdot)$ keeps};
\draw[->, dashed, myblue] (6.3, 0.1) to (3.9, 0.1);

\end{tikzpicture}

}
\end{tabular}

\caption{(a) An illustration of the two-environment model, the SCMs in the two environments share the same induced graph, which is also plotted in (b). The arrow from node $x$ to node $z$ indicates that $x$ is the direct cause of $z$. (b) An illustration of how EILLS works under the two-environment model. The double-circled nodes represent the pooled linear spurious variables.}
\label{fig:scm-variable-selection}
\end{figure}

In this section, we use an example of a structural causal model (with potential interventions on covariates $\bx$) to illustrate how different components of our proposed EILLS objective \eqref{eq:eills-objective-q} and $\ell_0$ regularized EILLS objective \eqref{eq:eills-objective-l} contribute to either seizing the important variables $S^*$ or eliminating pooled linear spurious variables and unrelated variables. Simulations further support intuitive claims\footnote{Scripts to reproduce the simulation result in this section can be found in the supplemental material, see also \href{https://github.com/wmyw96/EILLS}{https://github.com/wmyw96/EILLS}.}.

\noindent \textbf{Model.} Consider a two-environment model ($|\mathcal{E}|=2$). The mechanism of $(\bx, y)$ in each environment is characterized by an SCM. As shown in \cref{fig:scm-variable-selection} (a), the two SCMs share the same direct cause relationship and most of the structural assignments, except that the mechanisms of $x_4$ and $x_7$ are different (gray shadowed); see the detailed structural assignments in \aosversion{Appendix~\blue{D.4}}{\cref{subsec:implementation}}. In this case, the set of important variables $S^*$ is $\{1,2,3\}$, corresponding to the set of direct causes of $y$. The true parameter is $\bbeta^*=(3,2,-0.5,0,\ldots, 0)^\top$. The set of pooled linear spurious variables $G$ is the subset of $y$'s offspring. In this example, $G=\{7,8,9\}$ (denote as double circled node in \cref{fig:scm-variable-selection} (b)).

\noindent \textbf{How EILLS rescues.} First, although vanilla pooled least squares, which minimizing \eqref{eq:pooled-l2-risk-empirical}, can asymptotically eliminate some uncorrelated variables (for example, $\hat{\beta}_{\mathsf{R},12} = o_{\mathbb{P}}(1)$), it will asymptotically select some of the pooled linear spurious variables in $G$ together with other variables related to these variables (for example, some of their ancestors and offsprings) according to \eqref{eq:intro-lse-sol}. For example, it may asymptotically select variables $S_{\mathsf{R}} = \{1, 2, 3, 7, 8, 10\}$ as shown in the \fcolorbox{mylightblue}{white}{\mylightblue{light blue rectangle}}. That is, $|\hat{\beta}_{\mathsf{R},j} - \beta_j| = o_{\mathbb{P}}(1)$ with some $\beta_j\neq 0$ will hold for any $j \in S_{\mathsf{R}}$.

As for our EILLS estimator that minimizes \eqref{eq:eills-objective-q}, if the inclusion of pooled linear spurious variables leads to a shift of best linear predictor (\cref{cond-ident-main}), then, with a large enough $\gamma$, our regularizer \myred{$\hat{\mathsf{J}}$} will eliminate all pooled linear spurious variables in \colorbox{myred!20}{red shadows} for large enough sample size. We add a red ``\myred{$-$}'' in the red shadow to emphasize the regularizer \myred{$\hat{\mathsf{J}}$}'s role in eliminating those pooled linear spurious variables. 
Moreover, including the pooled $L_2$ risk \myblue{$\hat{\mathsf{R}}_{\mathcal{E}}$} will prevent our EILLS objective from collapsing to conservative solutions. Thus it will select all the important variables in \colorbox{myblue!20}{blue shadows}, in which we also use a blue ``\myblue{$+$}'' to underline the pooled $L_2$ risk \myblue{$\hat{\mathsf{R}}_{\mathcal{E}}$}'s impact on keeping all the important variables. 
Finally, the objective \eqref{eq:eills-objective-q} can only guarantee that $\hat{\beta}_{\mathsf{Q},j} = o_{\mathbb{P}}(1)$ for unrelated variables $j$ in \colorbox{myorange!20}{orange shadows}, the inclusion of \myorange{$\ell_0$ regularization} completes the last step towards the target of variable selection consistency: with a properly chosen hyper-parameter $\lambda$, we can eliminate (\myorange{$-$}) these unrelated variables through the \myorange{$\ell_0$ penalty} and only keep the important variables in \fcolorbox{myblue}{white}{\myblue{blue rectangle}}.

\begin{figure}
\centering
\begin{tabular}{ccc}
\subfigure[]{
\includegraphics[scale=0.36]{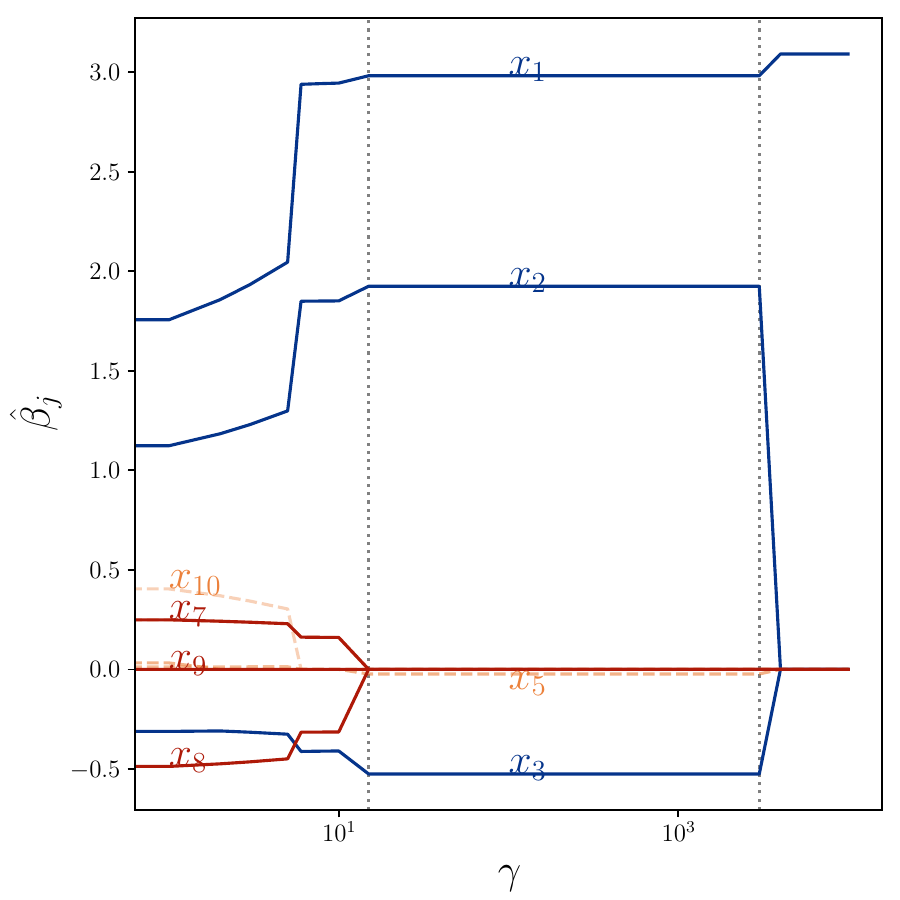}
}&
\subfigure[]{
\includegraphics[scale=0.36]{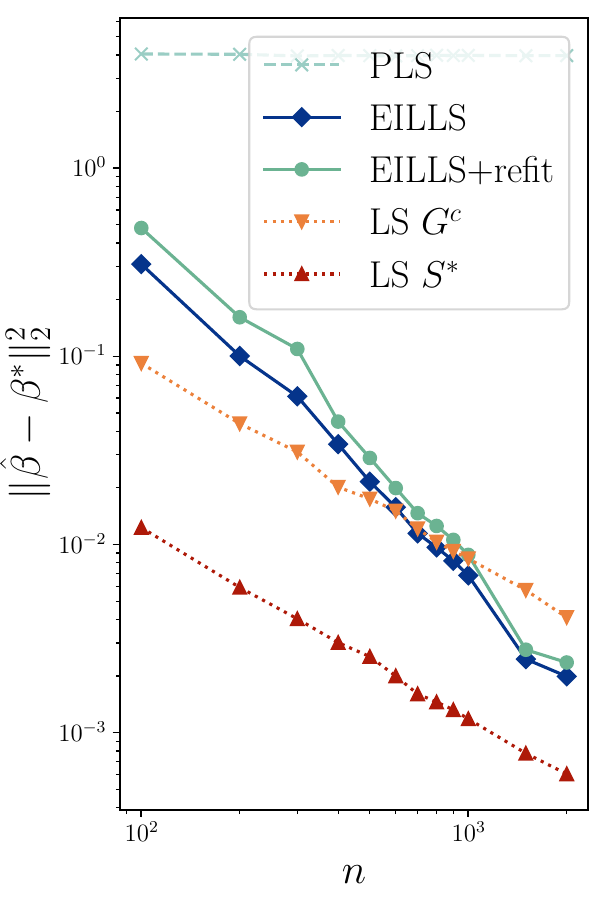}
} &
\subfigure[]{
\includegraphics[scale=0.36]{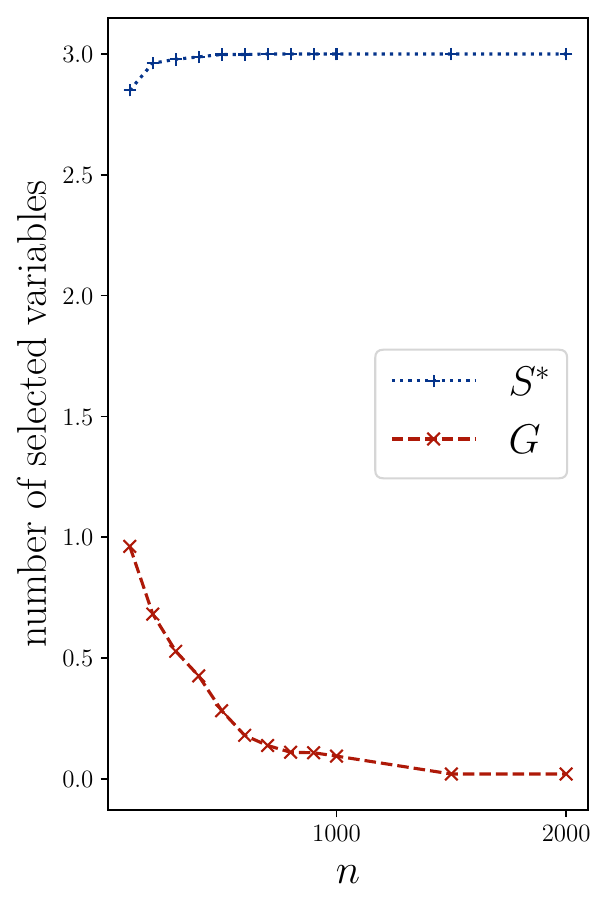}
} 
\end{tabular}
\caption{The simulation results for the model in \cref{fig:scm-variable-selection} (a). (a) depicts how the estimated coefficients for the EILLS estimator vary across hyper-parameter $\gamma$ in one trial when $n=300$: we use \myblue{blue} and \myred{red} solid lines to represent the corresponding coefficients for variables in \myblue{$S^*$} and \myred{$G$}, respectively; and use \myorange{orange} dashed lines to represent the coefficients for other variables. The two gray vertical lines are $\gamma=15$ and $\gamma=3\times 10^3$, respectively. (b) depicts how the average $\ell_2$ errors (based on $500$ replications, shown in log scale) for different estimators (marked with different shapes) change when $n$ grows: `LS $S$' is the estimator that runs least squares on $\bx_S$ using all the data and `PLS' is referred to `LS $[p]$'. 
(c) depicts the average number of selected variables in $S^*$ (\myblue{$+$}) and $G$ (\myred{$\times$}) for the EILLS estimator over $500$ replications.
}
\label{fig:numerical}
\end{figure}

\noindent \textbf{Experimental Justification.} We support the above intuition via simulations, in which the balanced data setting ($n^{(e)} \equiv n$, $\omega^{(e)}\equiv 1/2$) is adopted. Detailed implementation and experimental configurations are presented in \aosversion{Appendix~\blue{D.4}}{\cref{subsec:implementation}}.

Let us first see how EILLS works in practice by visualizing how the estimates of coefficients change over $\gamma \in [0, 10^4]$ in one trial. As shown in \cref{fig:numerical} (a), the pooled least squares running least squares on all the data ($\gamma=0$) will lead to a biased solution, which significantly selects variables in $S_{\mathsf{R}}$. Meanwhile, the EILLS estimator with proper regularization parameter $\gamma \in [15, 3\times 10^3]$ selects variables $\{1,2,3,5\}$: it screens out all the variables in $G$ and keeps all the important variables in $S^*$. \cref{fig:numerical} (a) also demonstrates the necessity of incorporating $L_2$ risk: with a very large $\gamma$ ($\gamma \ge 4\times 10^3$) that the $L_2$ risk is relatively negligible compared with the invariance regularizer, the best (empirical) linear predictor on $\{1,2,3\}$ is not preferred than that on $\{1\}$ by the invariance regularizer.

\cref{fig:numerical} (b) and \cref{fig:numerical} (c) further support the above claims using the average performances over $500$ replications, in which we use fixed $\gamma=20$ for EILLS. As presented in \cref{fig:numerical} (b), the average square of $\ell_2$ estimation error $\|\hat{\bbeta} - \bbeta^*\|_2^2$ for pooled least squares estimator (\mylightblue{$\times$} PLS) does not decrease (remains to be $\approx 4$) as $n$ increases, indicating that it converges to a biased solution. At the same time, the average $\|\hat{\bbeta} - \bbeta^*\|_2^2$ for the EILLS estimator (\myblue{$\blacklozenge$}) decays as $n$ grows (is $\approx 0.1$ when $n=200$) and lies in between that for least squares on $\bx_{G^c}$ (\myorange{$\blacktriangledown$}) and least squares on $\bx_{S^*}$ (\myred{$\blacktriangle$}) when $n\ge 700$. This is expected to happen since the EILLS estimator can not screen out all the uncorrelated variables.  We also report the $\ell_2$ estimation error for the refitted EILLS estimator, which runs OLS on the variable set $\hat{S}$ EILLS selects. It is interesting to see that its performance is slightly worse than the vanilla EILLS estimator; we provide a quantitative explanation in \aosversion{Appendix~\blue{D.5}}{\cref{sec:diagnosis}}.

The variable selection property for EILLS is further demonstrated in \cref{fig:numerical} (c), where the average number of selected variables in $S^*$ (\myblue{$+$}) and $G$ (\myred{$\times$}) is plotted across different $n$. The ``\myblue{$+$}'' curve keeps increasing and is almost $3$ while the ``\myred{$\times$}'' curve decays and approaches $0$, implying that the EILLS will select almost all the variables in $S^*$ and screen out all the variables in $G$ when $n$ grows.

\noindent \textbf{Comparison with Other Invariance Approaches.} In addition, we compare our EILLS approach with other invariance approaches, including invariance causal prediction (ICP) \citep{peters2016causal}, anchor regression (Anchor) \citep{rothenhausler2021anchor}, and invariant risk minimization (IRM) \citep{arjovsky2019invariant}, using the data generating process above. The pooled least squares (PLS) method is also included for comparative purposes. We use EILLS with hyper-parameter $\gamma=20$. For other invariance approaches, the invariance hyper-parameters are chosen in an oracle manner -- we enumerate all the possible hyper-parameters and pick the one that minimizes the $\ell_2$ prediction error $\|\bar{\bSigma}^{1/2}(\hat{\bbeta} - \bbeta^*)\|_2^2$. The implementation details can be found in the Supplemental Material.

The results are shown in \cref{fig:comp} (a). It is apparent that among all the estimators evaluated, the EILLS estimator stands out as the only one capable of consistent estimation. We also provide graphical visualizations of the solutions found by different methods for $n=100$ and $n=1000$ in \cref{fig:comp} (b) and (c), respectively. In these plots, each point $(x, y)$ with marker $m$ represents a solution $\hat{\bbeta}$ that the method $m$ produces. Here, $x$ denotes the relative $\ell_2$ norm restricted to the true important variables set $S^*$, calculated as $\|\hat{\bbeta}_{S^*}\|_2/\|{\bbeta}^*_{S^*}\|_2$; and $y$ represents the relative $\ell_2$ norm restricted to the pooled linear spurious variables set $G$, expressed as $\|\hat{\bbeta}_{G}\|_2/\|\bar{\bbeta}_{G}\|_2$ with $\bar{\bbeta} = \bar{\bbeta}^{([12])}$. 

As depicted in \cref{fig:comp} (b) and (c), the ICP method demonstrates a very conservative nature, failing to select variables in $S^*$ even when $n=1000$. This reveals its lack of guarantees in power even when the sample size is large enough. For other optimization-based methods, the solutions obtained by anchor regression are similar to those found by the pooled least squares. Although the IRM method showcases a slight divergence from PLS, with a tendency to push solutions toward the direction of ``invariance'', this effect remains marginal. In contrast, our EILLS method not only converges to the true parameter $\bbeta^*$ when $n$ is large but also demonstrates commendable performance when the sample size is moderately large ($n=100$).

\begin{figure}
\centering
\begin{tabular}{ccc}
\subfigure[]{
\includegraphics[scale=0.33]{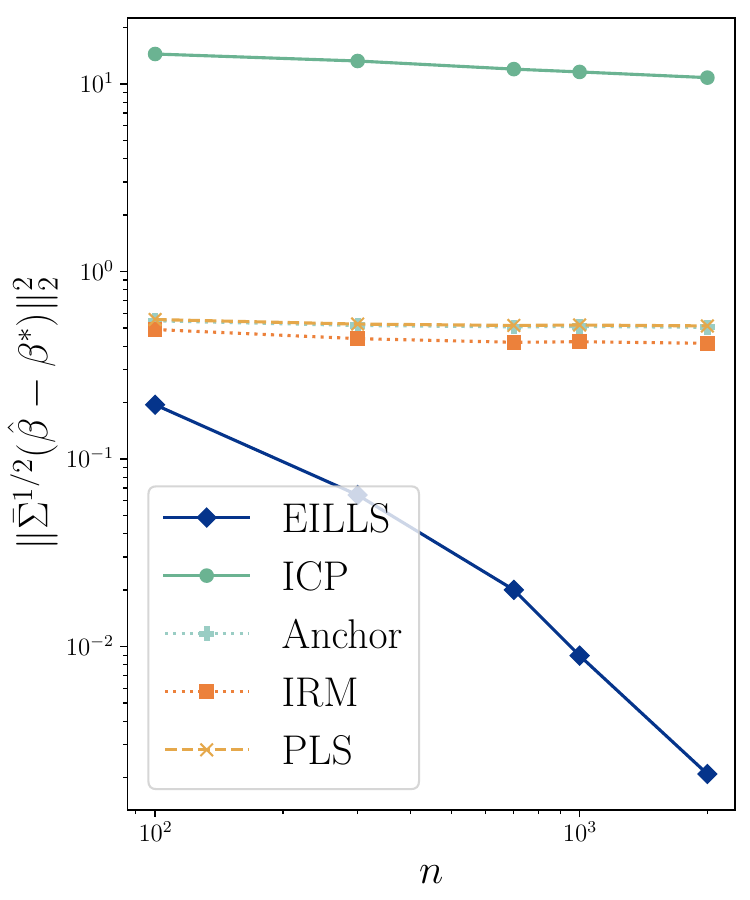}
}&
\subfigure[]{
\includegraphics[scale=0.33]{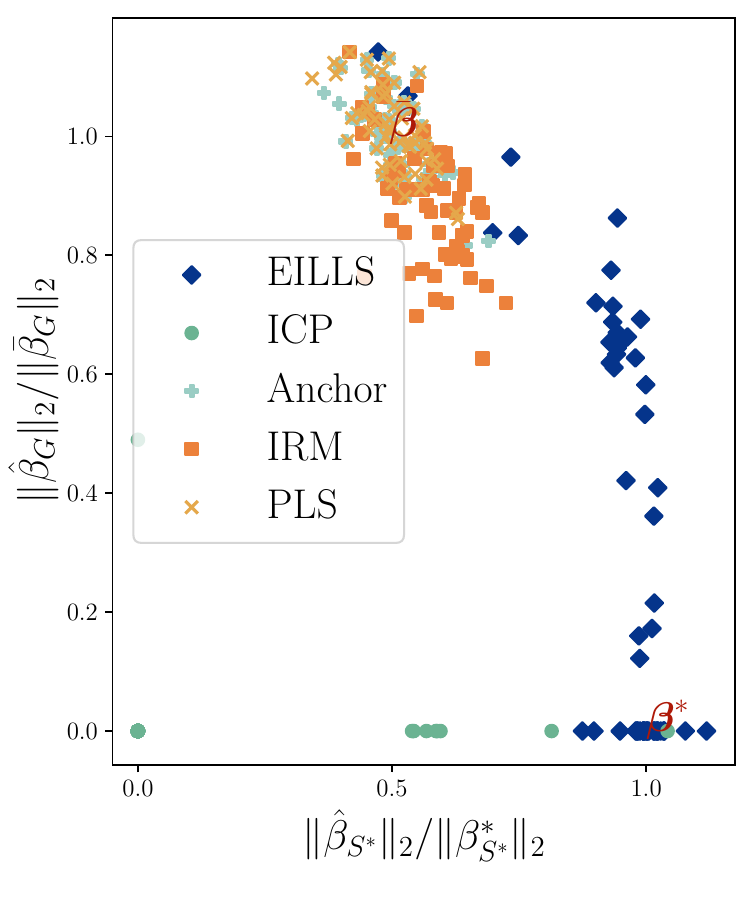}
} &
\subfigure[]{
\includegraphics[scale=0.33]{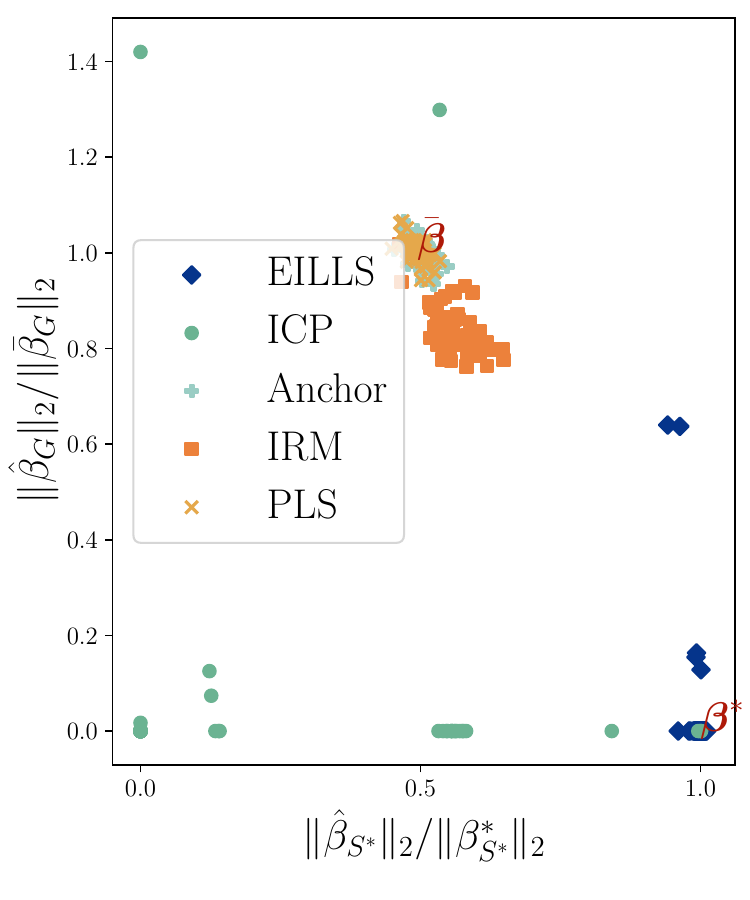}
} 
\end{tabular}
\caption{The simulation results for different methods using the data generated from the model in \cref{fig:scm-variable-selection} (a). (a) depicts how the average $\ell_2$ prediction errors $\|\bar{\bSigma}^{1/2}(\hat{\bbeta} - \bbeta^*)\|_2^2$ (based on $300$ replications) for different invariance methods (marked with different shapes and colors) changes when $n$ grows. (b) and (c) visualizes the solutions of different methods in 60 replications when $n=100$ and $n=1000$, respectively. The true parameter $\bbeta^*$ and the population pooled least squares solution $\bar{\bbeta}$ are also included using red for reference. 
}
\label{fig:comp}
\end{figure}

\section{Discussion}

In this paper, we consider the multi-environment linear regression model. We propose the {\it environment invariant linear least squares}, an optimization-based method applicable under the generic multi-environment linear regression model without additionally imposed structures. We provide a thorough statistical analysis of the proposed method. Specifically, it is possible to identify the true parameter under a near-minimal population-level condition \cref{cond-ident-main}. Under such a condition, the EILLS estimator can obtain an optimal linear regression rate in the low-dimensional regime, and the $\ell_0$ regularized EILLS estimator can achieve variable selection consistency in the high-dimensional regime. 

One key theoretical takeaway from this paper in the invariance field is that a statistically efficient estimation of $\bbeta^*$ is viable under a general, near-minimal identification condition related to the heterogeneity of the environments. This paper proposes an estimator with a non-convex objective function for the linear model to realize statistically efficient estimation, which is the first in the literature. The understanding of the invariance problem that this paper presents, together with this paper's limitations on the linear model and the computationally inefficient algorithm, opens up several interesting and promising future directions.

\subsection{Extension to Nonlinear Models} 
\label{sec:ex-nonlinear}

As illustrated above, one fundamental limitation of previously studied methods is their reliance on specific structures, such as linear SCM with additive intervention. This nature not only restricts them into ``structural methods'', limiting their scalable uses, but also hinders them from generalizing beyond linear settings, even for generalized linear models. On the contrary, our method can be naturally extended to generalized linear models, and one can use a similar idea to develop approaches in a generic nonparametric setup. Here, we only present a direct extension to the generalized linear model using a similar objective function and leave the extension to the fully nonparametric setup as future studies.

Let $\bx \in \mathbb{R}^p$ be the covariate vector and $y\in \mathbb{R}$ be the response variable of interest, and we collect data from $|\mathcal{E}|$ environments. Following the setup in \cref{sec:setup}, we assume that there exists some unknown true important variable set $S^*$ and true parameter $\bbeta^*$ with support set $S^*$ such that
\begin{align*}
 \forall e\in \mathcal{E} \qquad \mathbb{E}[y^{(e)}|\bx_{S^*}^{(e)}] = \varphi((\bbeta^*_{S^*})^\top \bx_{S^*}^{(e)})
\end{align*} with some known invertible link function $\varphi(\cdot)$ such as the logistic regression or log-linear model. The population-level \emph{generalized EILLS} objective analogy to \eqref{eq:intro-eills-obj-popu} in this setup can be written as
\begin{align*}
    \sum_{e\in \mathcal{E}} \mathbb{E}[\ell(\bbeta^\top \bx^{(e)}, y^{(e)})] + \gamma \sum_{j=1}^p \mathds{1}\{\beta_j\neq 0\} \left|\mathbb{E}[\{y^{(e)}-\varphi(\bbeta^\top \bx^{(e)})\} x_j^{(e)}]\right|^2
\end{align*} where $\ell(y, v) = \psi(v) - vy$ with $\varphi(t) = \psi'(t)$, and $\gamma$ is some hyper-parameter to be determined. Examples include (1) $\psi(t) = \frac{1}{2}t^2, \varphi(t) = t$ for linear regression; (2) $\psi(t) = \log(1+e^t)$, $\varphi(t) = \frac{e^t}{1+e^{t}}$ for logistic regression. The high-level viewpoints for the two parts in the loss with general $\varphi$ are similar to linear regression with $\varphi(t)=t$. One is expected to derive theoretical results analogous to \cref{thm:global-strong-convexity-main}--\ref{thm:high-dim-main}; we leave this for future studies.

\subsection{The Computational Complexity Concern}

Currently, we use an algorithm whose computational complexity scales exponentially with $p$ to search for the global minima of the EILLS objective function. It is natural to ask if we can design a provable algorithm that is both computationally and statistically efficient under the same environment heterogeneity condition \cref{cond-ident-main}. Even if there are some fundamental limits to this problem and it is impossible to develop both computationally and statistically efficient algorithms, it is still interesting to study if we can develop heuristic search algorithms, like sure screening \citep{fan2008sure}, forward-and-backward search \citep{zhang2011adaptive} and mixed-integer programming \citep{bertsimas2016best} for variable selection in linear regression, such that it can offer a good solution within an affordable time limit.

}
\section*{Acknowledgement}
The authors would like to thank three anonymous referees, an Associate Editor, and the Editor for their constructive comments that improved the quality of this paper. Y. Gu thank Yiran Jia for pointing out a typo in the proof of \cref{thm:rate-faster-low-dim}. J. Fan's research is supported by the ONR grants N00014-19-1-2120 and N00014-22-1-2340 and NSF grants DMS-2052926, DMS-2053832 and DMS-2210833. C. Fang's research is supported by National Key R\&D Program of China (2022ZD0160301) and NSF China (No. 62376008).

\bibliographystyle{apalike2}
\bibliography{arxiv_main.bbl}

\newpage
\appendix

\section*{Supplemental Material}

The supplemental materials collect the complete theoretical analysis (\cref{sec:theory:general}), all the population-level proofs (\cref{sec:proof-population}), finite sample proofs (\cref{sec:proof-nonasymptotic}), and omitted discussions in the main text (\cref{sec:discussion}).

\section{General Theoretical Analysis}

\label{sec:theory:general}

\subsection{General Notations and Conditions}

Define the pooled covariance matrix $\bar{\bSigma} = \sum_{e\in \mathcal{E}} \omega^{(e)} \bSigma^{(e)}$. We first state the standard assumptions used in linear regression, which are analogous to \cref{cond0-model-main}--\ref{cond3-sub-gaussian-eps-main}. \cref{cond1-well-condition} and \cref{cond3-sub-gaussian-eps} are just copies of \cref{cond1-well-condition-main} and \cref{cond3-sub-gaussian-eps-main}, respectively. \cref{cond0-model} allows for varying $(n^{(e)}, \omega^{(e)})$. \cref{cond2-subgaussain-x} replaces the $\bSigma$ in \cref{cond2-subgaussain-x-main} by $\bar{\bSigma}$.

\begin{condition}
\label{cond0-model}
	For each $e\in \mathcal{E}$, $(\bx_1^{(e)}, y_1^{(e)}), \ldots, (\bx_{n^{(e)}}^{(e)}, y_{n^{(e)}}^{(e)})$ are i.i.d. copies of $(\bx^{(e)}, y^{(e)}) \sim \mu^{(e)}$, where $\mu^{(e)}$ belongs to $\mathcal{U}_{\bbeta^*, \sigma^2}$ for some $\sigma^2$. The data from different environments are also independent. We have $\omega^{(e)} > 0$ for any $e\in \mathcal{E}$. 
\end{condition}

\begin{condition}
\label{cond1-well-condition}
	There exists some universal constants $\kappa_L \in (0,1]$ and $\kappa_U \in [1,\infty)$ such that
	\begin{align}
		\forall e\in \mathcal{E}, ~~~~~~~~~~ \kappa_L \bI_p \preceq \bSigma^{(e)} \preceq \kappa_U \bI_p.
	\end{align}
\end{condition}

\begin{condition}
\label{cond2-subgaussain-x}
	There exists some universal constant $\sigma_x \in [1,\infty)$ such that
	\begin{align}
		\forall e\in \mathcal{E}, \bv \in \mathbb{R}^p, ~~~~~~~~~~ \mathbb{E} \left[\exp\left\{\bv^\top \bar{\bSigma}^{-1/2}\bx^{(e)}\right\}\right] \le \exp\left(\frac{\sigma_x^2}{2} \cdot \|\bv\|_2^2\right).
	\end{align}
\end{condition}

\begin{condition}
\label{cond3-sub-gaussian-eps}
	There exists some universal constant $\sigma_\varepsilon \in \mathbb{R}^+$ such that,
	\begin{align}
		\forall e\in \mathcal{E}, \lambda \in \mathbb{R}, ~~~~~~~~~~ \mathbb{E} [e^{\lambda \varepsilon^{(e)}}] \le e^{\frac{1}{2} \lambda^2 \sigma_\varepsilon^2}.
	\end{align}
\end{condition}

We also define several quantities regarding sample size:
\begin{align}
    \label{eq:def-n-low-dim}
    n_* =\min_{e\in \mathcal{E}} \frac{n^{(e)}} {\omega^{(e)}} \qquad \bar{n} = \Bigg(\sum_{e\in \mathcal{E}}\frac{\omega^{(e)}}{n^{(e)}}\Bigg)^{-1} \qquad n_{\min} = \min_{e\in \mathcal{E}} n^{(e)} \qquad n_\dagger = \Bigg( \sum_{e\in \mathcal{E}} \frac{\omega^{(e)}}{(n^{(e)})^{3/2}} \Bigg)^{-1}.
\end{align}
It is easy to see that $n_* \ge \bar{n} \ge n_{\min}$. In the case of balanced data with equal weights, i.e., $n^{(e)} \equiv n$ and $\omega^{(e)} \equiv 1/|\mathcal{E}|$, we have $n_* = n\cdot |\mathcal{E}|$, $\bar{n}=n_{\min} = n$ and $n_\dagger=n^{3/2}$.

We also define the {\it pooled linear spurious variable} concerning the environment weights $\bomega$.

\begin{definition}[Pooled Linear Spurious Variables] We let $G_{\bomega}$ be the index set of all pooled linear spurious variables in environments $\mathcal{E}$ concerning the weights $\bomega$, that is,
		$G_{\bomega} = \{j \in [p]:  \sum_{e\in \mathcal{E}} \omega^{(e)} \mathbb{E}[x_j^{(e)}\varepsilon^{(e)}] \neq 0 \}$. We say a variable $x_j$ is a pooled linear spurious variable if $j\in G_{\bomega}$.
\end{definition}

We need the following condition related to the heterogeneity of environments in the presence of pooled linear spurious variables to recover $\bbeta^*$.

\begin{condition}[Identification]
\label{cond-ident}
    For any $S\subseteq [p]$ satisfying $S \cap G_{\bomega} \neq \emptyset$, there exists some $e, e'\in \mathcal{E}$ such that $\bbeta^{(e, S)} \neq \bbeta^{(e',S)}$, where $\bbeta^{(e,S)}$ is defined in \eqref{eq:local-minimizer}.
\end{condition} 

\subsection{Strong Convexity with respect to the True Parameter}

We are now ready to state the main population-level result, a generalized version of \cref{thm:global-strong-convexity-main} with varying environments coefficients $\bomega$.

\begin{theorem}[{Strong Convexity with respect to} $\bbeta^*$] 
\label{thm:global-strong-convexity}
Assume Conditions \ref{cond0-model}--\ref{cond1-well-condition} and \ref{cond-ident} hold. Then $\bbeta^*$ is the unique minimizer of $\mathsf{Q}(\bbeta;\gamma,\bomega)$ for large enough $\gamma$: for any $\epsilon \in (0, 1)$ and any $\gamma\ge \epsilon^{-1} \gamma^*$ with
\begin{align}
\label{eq:thm1-gamma-star}
	\gamma^* = (\kappa_L)^{-3} \sup_{S: S\cap G_{\bomega} \neq \emptyset} \left(\mathsf{b}_S/\bar{\mathsf{d}}_S\right),
\end{align} where
\begin{align}
\label{eq:def-bias-bias-difference}
 \mathsf{b}_S = \Bigg\|\sum_{e\in \mathcal{E}} \omega^{(e)} \mathbb{E} [\varepsilon^{(e)} \bx_S^{(e)}]\Bigg\|^2_2 \qquad \text{and} \qquad \bar{\mathsf{d}}_S = \sum_{e\in \mathcal{E}} \omega^{(e)} \big\|\bbeta^{(e,S)} - \bar{\bbeta}^{(S)}\big\|_2^2  
\end{align} 
 with $\bar{\bbeta}^{(S)}=\sum_{e'\in \mathcal{E}} \omega^{(e')} \bbeta^{(e',S)}$, we have
\begin{align}
\label{eq:thm1-jsc-condition}
	\mathsf{Q}(\bbeta;\gamma, \bomega) - \mathsf{Q}(\bbeta^*;\gamma, \bomega) \ge (1-\epsilon) \|\bar{\bSigma}^{1/2}(\bbeta - \bbeta^*)\|_2^2 + \kappa_L^2 (\gamma - \epsilon^{-1}\gamma^*) \bar{\mathsf{d}}_{\supp(\bbeta)}.
\end{align}
\end{theorem}

It is easy to see that \cref{thm:global-strong-convexity-main} is a direct corollary of the above \cref{thm:global-strong-convexity} with $\omega^{(e)} \equiv 1 / |\mathcal{E}|$.

\subsection{Non-asymptotic Analysis}

The first proposition characterizes the convergence of the pooled least squares.

\begin{proposition}[Properties of Pooled Least Squares]
\label{prop:pooled-least-squares}
	Assume Conditions \ref{cond0-model}--\ref{cond3-sub-gaussian-eps} hold. Then, there exists some $\bar{\bbeta}^{\mathsf{R}} \in \mathbb{R}^p$ satisfying $\frac{1}{\kappa_U} \left\|\sum_{e\in \mathcal{E}} \omega^{(e)} \mathbb{E} [\varepsilon^{(e)} \bx^{(e)}] \right\|_2 \le \|\bar{\bbeta}^{\mathsf{R}} - \bbeta^*\|_2 \le \frac{1}{\kappa_L} \left\|\sum_{e\in \mathcal{E}} \omega^{(e)} \mathbb{E} [\varepsilon^{(e)} \bx^{(e)}] \right\|_2$
 	such that, for any $\bbeta \in \mathbb{R}^p$,
	\begin{align}
	\label{eq:pooled-least-square-strong-convexity}
		\mathsf{R}(\bbeta) - \mathsf{R}(\bar{\bbeta}^{\mathsf{R}}) = \|\bar{\bSigma}^{1/2}(\bbeta - \bar{\bbeta}^{\mathsf{R}})\|_2^2.
	\end{align} Moreover, there exist universal constants $c_1$ and $c_2$ such that if $n^* \ge c_1 \sigma_x^4(p + t)$, then the pooled least      squares estimator $\hat{\bbeta}_{\mathsf{R}}$ minimizing \eqref{eq:pooled-l2-risk-empirical} satisfies
	\begin{align}
		\|\bar{\bSigma}^{1/2}(\hat{\bbeta}_{\mathsf{R}} - \bar{\bbeta}^{\mathsf{R}})\|_2 \le c_2 \sigma_x \left(\sigma_\varepsilon + \sigma_x \|\bar{\bSigma}^{1/2}(\bar{\bbeta}^{\mathsf{R}} - \bbeta^*)\|_2\right) \sqrt{\frac{p + t}{n_*}}
	\end{align} with probability $1-2e^{-t}$.
\end{proposition}

The following two theorems are generalized versions of \cref{thm:rate-low-dim-main} and \cref{thm:rate-faster-low-dim-main} with varying $n^{(e)}$ and $\omega^{(e)}$. 

\begin{theorem}[Non-asymptotic Variable Selection Property]
\label{thm:rate-low-dim} Define
\begin{align}
\label{eq:thm2-signal}
    \mathsf{s}_+ = \min_{j\in S^*} |\beta_j^*|^2 \qquad \text{and} \qquad \mathsf{s}_- = \min_{S\subseteq [p], S\cap G_{\bomega} \neq \emptyset} \bar{\mathsf{d}}_S
\end{align} 
Suppose Conditions \ref{cond0-model}--\ref{cond-ident} hold, and we choose $\gamma \ge 3\gamma^*\lor 1$ where $\gamma^*$ is defined in \cref{thm:global-strong-convexity}. There exists some universal constants $c_1$--$c_2$ that only depends on $(\kappa_U, \sigma_x, \sigma_\varepsilon)$ such that for any $t>0$, if $n_{\min} \ge p + \log(2|\mathcal{E}|) + t$, $\bar{n} \ge c_1(\gamma/\kappa_L)(p+ \log(|\mathcal{E}|) + t)\{\mathsf{s}_+^{-0.5}+\mathsf{s}_+^{-1}+(\gamma\kappa_L\mathsf{s}_-)^{-0.5}\}$, and $n_* \ge c_2(\gamma/\kappa_L)^2(p+t)\{\mathsf{s}_+^{-1} + (\gamma\kappa_L\mathsf{s}_-)^{-1} + 1\}$, then the EILLS estimator $\hat{\bbeta}_{\mathsf{Q}}$ minimizing \eqref{eq:eills-objective-q} satisfies
\begin{align}
\label{eq:low-dim-variable-selection}
    \mathbb{P}\left[S^* \subseteq \supp(\hat{\bbeta}_{\mathsf{Q}}) \subseteq (G_{\bomega})^c \right] \ge 1-7e^{-t}.
\end{align}
\end{theorem}

Combining \cref{thm:rate-low-dim} and \cref{prop:pooled-least-squares}, we can conclude that when $(n_{\min}, \bar{n}, n_*)$ are large enough, one can refit the solution found by the EILLS estimator to obtain a faster rate. To be specific, we can run another pooled least squares on $\bx_{\hat{S}}$ with $\hat{S}=\supp(\hat{\bbeta}_{\mathsf{Q}})$ and obtain the $\ell_2$-error $\{|(G_{\bomega})^c|/n_*\}^{1/2}$ given that $S^*\subseteq \hat{S}\subseteq (G_{\bomega})^c$. 

\begin{theorem}[Non-asymptotic $\ell_2$ Error Bound] 
\label{thm:rate-faster-low-dim}
Assume Conditions \ref{cond0-model}--\ref{cond-ident} hold, and we choose $\gamma \ge 3\gamma^*\lor 1$ where $\gamma^*$ is defined in \cref{thm:global-strong-convexity}. There exists some universal constants $c_1$--$c_4$  that only depend on $(\kappa_U, \sigma_x)$ such that for any $t>0$, if $n_{\min} \ge p + \log(2|\mathcal{E}|) + t$ and $n_* \ge c_1 (\gamma/\kappa_L)(p + t)$, then $\hat{\bbeta}_{\mathsf{Q}}$ minimizing \eqref{eq:eills-objective-q} satisfies
\begin{align}
\label{eq:low-dim-rate2}
    \frac{\|\hat{\bbeta}_{\mathsf{Q}} - \bbeta^*\|_2}{\sigma_\varepsilon (\gamma/\kappa_L)}  &\le c_2 \left(\sqrt{\frac{p+t}{n_*}} + \frac{p+\log(|\mathcal{E}|)+t}{\bar{n}}\right) + c_3 \frac{\sqrt{|S^*|}}{\bar{n}} \cdot \frac{\log(|\mathcal{E}||S^*|)+t}{\min_{j\in S^*}|\beta^*_j|}
\end{align} with probability at least $1-7e^{-t}$. Moreover, when the additional conditions in \cref{thm:rate-low-dim} hold, then
\begin{align}
\label{eq:low-dim-rate1}
    \frac{\|\hat{\bbeta}_{\mathsf{Q}} - \bbeta^*\|_2}{\sigma_\varepsilon (\gamma/\kappa_L)}  &\le c_2 \left(\sqrt{\frac{p_0+t}{n_*}} + \frac{p_0+\log(|\mathcal{E}|)+t}{\bar{n}}\right) \qquad \text{with}~~p_0 = |({G}_{\bomega})^c|
\end{align} occurs with probability at least $1-14e^{-t}$. The dependency of $(c_1,c_2,c_3)$ on $(\kappa_U,\sigma_x)$ can be found in~\eqref{eq:low-dim-faster-rate-c1} and~{\eqref{eq:low-dim-faster-rate-c2-4}} in \cref{sec:proof-thm3}. 
\end{theorem} 

In the high-dimensional regime, recall that $s^* = |S^*|$, $\beta_{\min} = \min_{j\in S^*} |\beta_j^*|$. We need the following condition regarding the sample size defined above in \eqref{eq:def-n-low-dim}.

\begin{condition}
\label{cond6-n-req}
    Suppose that $\log(|\mathcal{E}|)\le C\log p$ and that 
    
\noindent (1) $n_{\min} \ge c_1 (s^* + \log p)$,
    
\noindent (2) $n_* \ge c_2 (\gamma/\kappa_L)^2 \Big(s^* \log p \Big) \Big\{1 + 1\big/ \big(\kappa_L \beta_{\min}^2 \big)\Big\}$, 
    
\noindent (3) $\bar{n} \ge c_3 (\gamma/\kappa_L) (\log p) \left\{s^* + 1/ (\beta_{\min}^2)\right\}$ and $\bar{n} \ge c_3 (\gamma/\kappa_L) \sqrt{(s^*\log p)(s^* + \log p)} \big/ \big(\kappa_L \beta_{\min} \big)$, 
    
\noindent (4) $n_\dagger \ge c_4(\gamma/\kappa_L) \big(s^* \log p \big)\sqrt{s^* + \log p}\Big\{1+1\big/ \big(\sqrt{\kappa_L} \beta_{\min}\big)\Big\}$.
    
\noindent Here $c_1$--$c_4$ are universal positive constants that depend only on $(C,\sigma_x,\kappa_U, \sigma_\varepsilon)$. See details in~{\eqref{eq:req-n-high-dim}}.
\end{condition}

The above \cref{cond6-n-req} matches \cref{cond6-n-req-main} in the case of balanced data with equal weights. The following theorem claims that EILLS can attain variable selection consistency with proper choice of the model selection hyper-parameter $\lambda$.

\begin{theorem}[Variable Selection Consistency in High Dimensions]	
\label{thm:high-dim}  Assume Conditions \ref{cond0-model}--\ref{cond6-n-req} hold, and we choose $\gamma \ge 3\gamma^* \lor 1$ where $\gamma^*$ is defined in \cref{thm:global-strong-convexity}. Suppose further that the choice of $\lambda$ satisfies
\begin{align}
    c_1 \left\{ (\gamma/\kappa_L)^2 \frac{s^*\log p}{n_*} + \epsilon(\bar{n}, n_\dagger)\right\} \le \lambda \le c_2 \kappa_L \beta_{\min}^2,
\end{align} where $\epsilon(\bar{n}, n_\dagger) = (\gamma/\kappa_L)^2 (\log p) (s^* + \log p) \{\bar{n}^{-2}s^* + n_\dagger^{-2} (s^*)^2\log p + (n_\dagger \sqrt{s^* + \log p})^{-1}(\gamma/\kappa_L)^{-1} \}$, and  $c_1, c_2$ are some universal positive constants only depends on $(C,\kappa_U, \sigma_x,\sigma_\varepsilon)$. Then the $\ell_0$ regularized EILLS estimator $\hat{\bbeta}_{\mathsf{L}}$ solving \eqref{eq:eills-objective-l} satisfies 
\begin{align*}
    \mathbb{P}\left[\supp(\hat{\bbeta}_{\mathsf{L}}) = S^*\right] \ge 1-p^{-10}.
\end{align*}
\end{theorem}

\subsection{Examples of \cref{thm:global-strong-convexity}}

\label{sec:moreex}

\subsubsection{Linear SCM with Heterogeneous Variance}

We start with a toy example similar to \cref{ex-h-de}.

\begin{example}[Toy Example, Heterogeneous Variances]
\label{ex-h-var-toy}
	 Consider the following multi-environment linear model with $\mathcal{E}=\{1,2\}$ and $p=2$,
	 \begin{align*}
	 	x^{(e)}_1 &\gets \sqrt{v_1} \cdot \varepsilon_1 \\
	 	y^{(e)} &\gets 1\cdot x^{(e)}_1 + \myred{\sqrt{v_0^{(e)}}} \cdot \varepsilon_0 \\
	 	x^{(e)}_2 &\gets h \cdot x^{(e)}_1 + s \cdot y^{(e)} + \sqrt{v_2} \cdot \varepsilon_2	
	 \end{align*}
	for some parameters $s, h \in \mathbb{R}$ and $v_1, v_2 \in \mathbb{R}^+$ that is invariant across different environments, and one parameter $v_0^{(e)}$ that will vary for different environments. Here $\varepsilon_0, \varepsilon_1, \varepsilon_2$ are independent, standard normal distributed exogenous variables. We let $\omega^{(e)}\equiv 1/2$.

    {\normalfont In the example, $S^*=\{1\}$, $\bbeta^* = (1,0)^\top$, and $\varepsilon^{(e)} = (v_0^{(e)})^{1/2} \cdot \varepsilon_0$. Hence $G_{\bomega} = \{2\}$ whenever $s\neq 0$ since $\mathbb{E}[\varepsilon^{(e)} x_2^{(e)}] = s v_0^{(e)}$. 
	In a well-conditioned regime where $|h|+|s| \lesssim 1$ and $v_1 \asymp v_2 \asymp v_0^{(e)} \asymp 1$, we have
	\begin{align*}
		\sqrt{\frac{\mathsf{b}_{\{2\}}}{\bar{\mathsf{d}}_{\{2\}}}} 
		&\asymp \frac{1}{|v_0^{(1)} - v_0^{(2)}|} \frac{1}{|h(s+h)v_1 + v_2|} ~~~~~~ \text{and} ~~~~~~ \sqrt{\frac{\mathsf{b}_{\{1,2\}}}{\bar{\mathsf{d}}_{\{1,2\}}}} \asymp \frac{1}{|v_0^{(1)} - v_0^{(2)}|}.
	\end{align*}
	Therefore, \cref{cond-ident} holds if 
	\begin{align}
	\label{eq:toy-example-identification}
		v_0^{(1)} \neq v_0^{(2)} ~~~~~~ \text{and} ~~~~~~ h (h+s) v_1 + v_2 \neq 0 
	\end{align} and the quantity $\gamma^*$ is of constant order if
	\begin{align}
	\label{eq:toy-example-constant-order-gamma-star}
		|v_0^{(1)} - v_0^{(2)}|	= \Omega(1) ~~~~~~ \text{and} ~~~~~~ \left|s - \left(-\frac{v_2}{v_1 h} - h\right)\right| = \Omega(1).
	\end{align} The first condition, $v^{(2)}_0$ should not be close to $v_0^{(1)}$, is easy to understand. This is because when $v_0^{(2)}$ is very close to $v_0^{(1)}$, the distribution of the two environments might be very similar, which will make it hard to distinguish. The second condition, the direct effect of $y$ on $x_2$ must be apart from $(-v_2 / (v_1 h) - h)$, is strange at first glance. Let us see what it implies when $s$ coincides with $(-v_2 / (v_1 h) - h)$. If $s=(-v_2 / (v_1 h) - h)$, then the heterogeneity of the environments (from $v_0^{(1)}$ to $v_0^{(2)}$) will not affect the solution $\bbeta^{(e,\{1,2\})}$ -- it will be fixed as $\bbeta^{(e,\{1,2\})} \equiv (0, s^{-1})^\top$. In other words, though the distributions of $(\bx^{(1)},y^{(1)})$ and $(\bx^{(2)},y^{(2)})$ differs a lot, the two identity equations hold
		\begin{align*}
			\mathbb{E} [y^{(e)}|x_1^{(e)}] \equiv 1\cdot x_1^{(e)} ~~~~~~ \text{and} ~~~~~~ \mathbb{E}[y^{(e)}|x_2^{(e)}] \equiv s^{-1}\cdot x_2^{(e)}.
		\end{align*} This will be problematic since the two sets $\{1\}$ and $\{2\}$ are all CE-invariant. In this case, one can still show that $\bbeta^*$ is the unique minimizer of \eqref{eq:eills-objective-q-population} if and only if
		$(v_0^{(1)} + v_0^{(2)})/2 < s^{-2}(h^2 v_1 + v_2) = v_1^2 h^2 / (v_1 h^2 + v_2)$, that is, whether the $L_2$ risk of $\bbeta^*$ is smaller than the $L_2$ risk of the spurious solution $(0, s^{-1})^\top$. But it is beyond the scope of our discussion. 
	
	Moreover, the condition \eqref{eq:toy-example-constant-order-gamma-star} is independent of $|s|$ that measures the magnitude of spuriousness. In particular, we do not need to use a very large $\gamma^*$ when $|s|$ is very small. This is because both $\bar{\mathsf{d}}_S$ and $\mathsf{b}_S$ grow linearly with $s^2$ around $s=0$, and the choice of $\gamma^*$ is their ratio.}
	\qed
\end{example}

Given the intuitions of the above \cref{ex-h-var-toy}, we are ready to present a clean condition for a generalization of the model in \cref{ex-h-var} with $p \ge 3$.

\begin{example}[General Linear SCM with intervention on the scale of $\varepsilon^{(e)}$.]
\label{ex-h-var}
	Consider the following two-environment model	 $\mathcal{E}=\{1,2\}$ for $p$-dimensional covariate $\bx$ and response variable $y$,
	\begin{align}
	\label{eq:model-linear-scm-varying-v0}
	\begin{split}
		\bx^{(e)} &= \bB \bx^{(e)} + \balpha y^{(e)} + \bvarepsilon_x,\\
		y^{(e)} &= (\bbeta^*)^\top \bx^{(e)} + \varepsilon_0^{(e)}
	\end{split}
	\end{align} where $\bB\in \mathbb{R}^p$, $\balpha \in \mathbb{R}^p$ and $\bbeta^*$ are the same across the two environments, the exogenous variables $(\bvarepsilon_x, \varepsilon_0^{(e)}) \in \mathbb{R}^{p}\times \mathbb{R}$ are independent, and the covariance matrix of $\bvarepsilon_x$, $\bD = \mathbb{V}[\bvarepsilon_x]$, is also fixed for the two environments. The only heterogeneity comes from the variance of the noise $\varepsilon_0^{(e)}$: $v_0^{(e)}=\mathbb{V}[\epsilon_{0}^{(e)}]$ is different for $e\in \{1,2\}$. We assume the induced graph is acyclic, and use $\bW = (\bI - \bB - \balpha (\bbeta^*)^\top)^{-1}$ to denote the total effect matrix for the covariate $\bx$, that is, $W_{i,j}$ is the total effect of the variable $x_j$ on variable $x_i$.
\end{example}

\begin{proposition} 
\label{prop:ex-h-var}
Consider the two-environment model in \cref{ex-h-var}, let $\omega^{(e)}\equiv 1/2$, suppose further that \begin{align}
\label{eq:ex2-req1}
	\max_{i\in [p]} D_{i,i} \lor \|\bW\|_2 \lor \|\balpha\|_2 \le c_1	
\end{align} for some universal constant $c_1>0$. For any $S$, define
\begin{align*}
\xi(S) = 1 - \balpha^\top \bW_{S,:}^\top \left(\bW_{S,:} \bD \bW_{S,:}^\top \right)^{-1} \bW_{S,:} \bD \bW_{T,:}^\top \bbeta_T^* ~~~~ \text{with} ~~~~T=S^*\setminus S.
\end{align*} If $\xi(S) \neq 0$ for any $S \subseteq [p]$, then \cref{cond-ident} holds, and the property \eqref{eq:thm1-jsc-condition} in \cref{thm:global-strong-convexity} holds with
\begin{align}
\label{eq:ex2-cond}
	\gamma^* = c_2 \left(\min_{i\in [p]} D_{i,i}\right)^{-7} &\times \left\{\frac{(v_0^{(1)} + v_0^{(2)})(1 + v_0^{(1)})(1 + v_0^{(2)})}{|v_0^{(1)} - v_0^{(2)}|} \times \sup_{S \subseteq [p]} \frac{1}{\left|\xi(S)\right|}\right\}^2
\end{align} for some constant $c_2>0$ depends only $c_1$.
\end{proposition}

When $\min_{i\in [p]} D_{i,i} \gtrsim 1$ and $v_0^{(e)} \lesssim 1$, \cref{prop:ex-h-var} suggests the following requirements 
\begin{align*}
	| v_0^{(1)} - v_0^{(2)} | = \Omega(1) ~~~~~~ \text{and} ~~~~~~ \inf_{S\subseteq [p]} |\xi(S)| = \Omega(1)
\end{align*} are needed to get a constant order $\gamma^*$, here $\xi(S)=1$ if $S^* \subseteq S$ or $S\cap G_{\bomega} = \emptyset$. Compared with the conditions in \eqref{eq:toy-example-constant-order-gamma-star}, the first condition is the same, and the condition $\inf_{S\subseteq [p]} |\xi(S)| = \Omega(1)$ generalizes the second condition of \eqref{eq:toy-example-constant-order-gamma-star} to the multivariate case with complicated and arbitrary variable dependencies. Similar to \cref{prop:ex-h-var}, the choice of $\gamma^*$ is independent of $\|\balpha\|_2$ when $\|\balpha\|_2 \lesssim 1$, which means one do not need to use a large $\gamma^*$ when the magnitude of spuriousness $\|\balpha\|_2$ is small, i.e., $\|\balpha\|_2 \approx 0$.

\section{Proofs for Population-level Results}
\label{sec:proof-population}

\subsection{Proof of \cref{prop:oos-definition-property}}
\label{prop:oos-definition-property-proof}

We first prove the equality $\mathsf{R}_{\mathtt{oos}}(\bbeta^*;\mathcal{U}_{\bbeta^*, \sigma^2})=\sigma^2$. The definition of $\mathcal{U}_{\bbeta^*, \sigma^2}$ implies that,
\begin{align*}
    \mathsf{R}_{\mathtt{oos}}(\bbeta^*;\mathcal{U}_{\bbeta^*, \sigma^2}) = \sup_{\mu \in \mathcal{U}_{\bbeta^*, \sigma^2}} \mathbb{E}_{(\bx,y)\sim \mu} \left[ |y - \mathbb{E}[y|\bx_S^*] |^2 \right] = \sup_{\mu \in \mathcal{U}_{\bbeta^*, \sigma^2}} \mathrm{var}_\mu [y|\bx_{S^*}] \le \sigma^2.
\end{align*} Moreover, denote $\nu \sim \mathcal{N}(0, \sigma^2 I_{p})$ as a Gaussian distribution on $\mathbb{R}^{p}$. Let $\mu_*(d\bx, y) = \nu(\bx) \nu_y(d\by|\bx)$, where $\nu_y$ is also a Gaussian distribution with mean $(\bbeta^*)^\top \bx$ and variance $\sigma^2$, it is easy to show that $\mu_* \in \mathcal{U}_{\bbeta^*, \sigma^2}$ and $\mathbb{E}_{\mu_*} [|y - (\bbeta^*)^\top \bx|^2] = \sigma^2$.Combining with the above upper bound, we can conclude that $\mathsf{R}_{\mathtt{oos}}(\bbeta^*;\mathcal{U}_{\bbeta^*, \sigma^2})=\sigma^2$.

We next prove the upper bound part of \eqref{eq:optimal-beta-s-star}. For any $\mu$, we have
\begin{align*}
    \mathbb{E}_{(\bx,y)\sim \mu} \left[ |y - \bbeta^\top \bx|^2 \right] &= \mathbb{E}\left[\left(y - (\bbeta^*)^\top\bx + (\bbeta^*)^\top\bx - \bbeta^\top \bx\right)^2\right] \\
    &= \mathbb{E}\left[\left(y - (\bbeta^*)^\top\bx \right)^2\right] + \mathbb{E} \left[ \left((\bbeta^*)^\top\bx - \bbeta^\top \bx\right)^2 \right] \\ 
    & ~~~~~~~~~~~~~~~~~~ + 2 \mathbb{E}\left[\left(y - (\bbeta^*)^\top\bx \right) \left((\bbeta^*)^\top\bx - \bbeta^\top \bx\right) \right] \\
    &\overset{(a)}{\le} \sigma^2 + \lambda_{\max} \left(\mathbb{E}\left[\bx\bx^\top\right]\right) \|\bbeta - \bbeta^*\|^2_2 \\
    &~~~~~~~~~~~~~~~~~~~ + 2\sigma \sqrt{\lambda_{\max} \left(\mathbb{E}\left[\bx\bx^\top\right]\right) \|(\bbeta - \bbeta^*)_{[p]\setminus S^*}\|^2_2} \\
    &\le \mathsf{R}_{\mathtt{oos}}(\bbeta^*;\mathcal{U}_{\bbeta^*, \sigma^2}) + p^2\sigma^2 \|\bbeta - \bbeta^*\|_2^2 + 2\sigma^2 p  \|(\bbeta - \bbeta^*)_{[p]\setminus S^*}\|_2
\end{align*} where $(a)$ follows from the fact $\mathbb{E}[(y-(\bbeta^*)^\top)\bx|\bx_{S^*}]=0$ and Cauchy-Schwarz inequality.
For the lower bound part, 
Now that $\varepsilon = y - (\bbeta^*)^\top \bx$ is independent of $\bx$ under $\mu_*$  yields, for any $\bbeta \in \mathbb{R}^p$,
\begin{align*}
    \mathsf{R}_{\mathtt{oos}}(\bbeta;\mathcal{U}_{\bbeta^*, \sigma^2}) \ge \mathbb{E}_{(\bx,y)\sim\mu_*} \left[|y-\bbeta^\top \bx|^2\right] &= \mathbb{E}_{(\bx,y)\sim \mu_*} \left[\left(y - (\bbeta^*)^\top \bx +  (\bbeta^*)^\top \bx - \bbeta^\top \bx\right)^2\right] \\
    &= \mathbb{E}_{(\bx,y)\sim \mu_*} \left[\left(y - (\bbeta^*)^\top \bx\right)^2\right] + \sigma^2 \|\bbeta - \bbeta^*\|_2^2 \\
    &= \mathsf{R}_{\mathtt{oos}}(\bbeta^*;\mathcal{U}_{\bbeta^*,\sigma^2}) + \sigma^2 \|\bbeta - \bbeta^*\|_2^2,
\end{align*} which completes the proof. \qed

\subsection{Proof of \cref{prop:necessary-multiple-envs}}
\label{prop:necessary-multiple-envs-proof}
    Denote $s^* = |S^*|$. Without loss of generality, let $S^* = \{1,\cdots, s\}$ and $\beta_{s+1} \neq 0$. Consider the following matrix
    \begin{align*}
        \bar{\bSigma} = \begin{bmatrix}
            \bI_{s^*} &\bB &\bbeta^*_{S^*} \\
            \bB^\top &\bI_{p-s^*} + \bB^\top \bB & \balpha \\
            (\bbeta^*_{S^*})^\top & \balpha^\top & c\\
        \end{bmatrix}
    \end{align*} where $\bI_{q}$ is $q$ by $q$ identity matrix, $\bB \in \mathbb{R}^{s^*\times (p-s^*)}$ satisfies $B_{j,k} = 1\{k=1\} (\beta_j^* - \beta_j)/(\beta_k)$, $\balpha = \bB^\top \bbeta_{S^*} + (\bI_{p-s^*} + \bB^\top \bB) \bbeta_{T}$, $c=1+\|\balpha\|_2^2 + \|\bbeta_{S^*}^*\|_2^2$. It is easy to verify that our constructed $\bar{\bSigma}$ satisfies $\bar{\bSigma} \succeq \bI_{p+1}$. We then argue that the distribution $\mu$ that $(\bx,y) \sim \mathcal{N}(0,\bar{\bSigma})$ satisfies $\mu \in \mathcal{U}_{\bbeta^*,\sigma^2}$ for some large enough $\sigma^2$ and $\mathbb{E}_{\mu}[y|\bx] = \bbeta^\top \bx$. To this end, we only need to verify that $\mathbb{E}[y|\bx_{S^*}] = (\bbeta^*)^\top \bx$ and $\mathbb{E}[y|\bx] = \bbeta^\top \bx$. For the first condition, our construction of $\bar{\bSigma}$ and the conditional distribution of multivariate Gaussian ensures $\mathbb{E}[y|\bx_{S^*}] = (\bI_{s^*})^{-1} (\bbeta_{S^*}^*)^\top \bx_S^* = (\bbeta_{S^*}^*)^\top \bx_S^*$. Similarly, in order to show $\mathbb{E}[y|\bx] = \bbeta^\top \bx$, it suffices to verify that
    \begin{align*}
        \begin{bmatrix}
            \bI_{s^*} &\bB \\
            \bB^\top &\bI_{p-s^*} + \bB^\top \bB
        \end{bmatrix} \bbeta = \begin{bmatrix}
            \bbeta^*_{S^*} \\
            \balpha \\
        \end{bmatrix},
    \end{align*} which can be validated by plugging in our construction of $\bB$ and $\balpha$. This completes the proof. \qed

\subsection{Proof of \cref{thm:global-strong-convexity}}

It follows from the definition of $\mathsf{Q}(\bbeta; \gamma, \bomega)$ and \cref{cond0-model} that
\begin{align*}
	\mathsf{Q}(\bbeta; \gamma,\bomega) - \mathsf{Q}(\bbeta^*;\gamma, \bomega) = \mathsf{R}(\bbeta;\bomega) - \mathsf{R}(\bbeta^*; \bomega) + \gamma \mathsf{J}(\bbeta; \bomega) = \mathsf{T}_1(\bbeta) + \gamma \mathsf{T}_2(\bbeta).
\end{align*}
Denote $\bDelta = \bbeta - \bbeta^*$, $S=\supp(\bbeta)$, $T=S^*\setminus S$, and $\bar{S} = S\cup S^*$. it follows from the model $y^{(e)} = (\bbeta^*)^\top \bx^{(e)} + \varepsilon^{(e)}$ that
\begin{align*}
	\mathsf{T}_1(\bbeta) &= \sum_{e\in \mathcal{E}} \omega^{(e)} \left(\mathsf{R}^{(e)}(\bbeta) - \mathsf{R}^{(e)}(\bbeta^*)\right) \\
	&= \sum_{e\in \mathcal{E}} \omega^{(e)} \left(\mathbb{E} \left[|y^{(e)} - \bbeta^\top \bx^{(e)}|^2 \right]- \mathbb{E} \left[|y^{(e)} - (\bbeta^*)^\top \bx^{(e)}|^2 \right]\right) \\
	&= \sum_{e\in \mathcal{E}} \omega^{(e)} \left(\mathbb{E} \left[|(\bbeta^* - \bbeta)^\top \bx^{(e)} + \varepsilon^{(e)} |^2\right] - \mathbb{E}[|\varepsilon^{(e)}|^2 ]\right) \\
	&= \sum_{e\in \mathcal{E}} \omega^{(e)} \bDelta^\top \bSigma^{(e)} \bDelta - 2\bDelta^\top \mathbb{E}\left[\varepsilon^{(e)} \bx^{(e)} \right] \\
	&= \bDelta^\top \Bigg(\sum_{e\in \mathcal{E}} \omega^{(e)} \bSigma^{(e)} \Bigg) \bDelta - 2\bDelta^\top \Bigg(\sum_{e\in \mathcal{E}} \omega^{(e)} \mathbb{E}[\varepsilon^{(e)} \bx^{(e)}] \Bigg) = \bDelta_{\bar{S}}^\top \bar{\bSigma}_{\bar{S}} \bDelta_{\bar{S}} - 2 \bDelta^\top_{\bar{S}} \bar{\bu}_{\bar{S}}
\end{align*} 
It follows from the fact $\bar{\bu}_T=\bm{0}$ and Cauchy-Schwarz inequality that
\begin{align}
\label{eq:proof-thm1-eq1}
	2\bDelta^\top_{\bar{S}} \bar{\bu}_{\bar{S}} = 2\left(\sqrt{\epsilon} \bDelta_{\bar{S}} \bar{\bSigma}_{\bar{S}}^{1/2} \right) \left(\epsilon^{-1/2} \bar{\bSigma}_{\bar{S}}^{-1/2} \bar{\bu}_{\bar{S}} \right) \le \epsilon \bDelta_{\bar{S}}^\top \bar{\bSigma}_{\bar{S}} \bDelta_{\bar{S}} + \epsilon^{-1} \bar{\bu}_{\bar{S}}^\top \bar{\bSigma}_{\bar{S}}^{-1} \bar{\bu}_{\bar{S}}
\end{align}
Plugging \eqref{eq:proof-thm1-eq1} back yields
\begin{align}
\label{eq:proof-thm1-t1-lb}
	\mathsf{T}_1(\bbeta) \ge (1-\epsilon) \bDelta_{\bar{S}}^\top \bar{\bSigma}_{\bar{S}} \bDelta_{\bar{S}} - \epsilon^{-1} \bar{\bu}_{\bar{S}}^\top \bar{\bSigma}_{\bar{S}}^{-1} \bar{\bu}_{\bar{S}}.
\end{align}
At the same time, we also have
\begin{align*}
	\mathsf{T}_2(\bbeta) &= \sum_{e\in \mathcal{E}} \omega^{(e)} \left\|\mathbb{E} \left[(y^{(e)} - \bbeta^\top \bx^{(e)}) \bx_S^{(e)}\right]\right\|_2^2 \\
	&= \sum_{e\in \mathcal{E}} \omega^{(e)} \left\| \mathbb{E} \left[\varepsilon^{(e)} \bx_{\bar{S}}^{(e)} \right] + (\bbeta^*_S - \bbeta_S)^\top \mathbb{E} \left[\bx_S^{(e)} (\bx_S^{(e)})^\top\right] + (\bbeta_T^*)^\top \mathbb{E} \left[\bx_T^{(e)} (\bx_S^{(e)})^\top\right] \right\|^2 \\
	&= \sum_{e\in \mathcal{E}} \omega^{(e)} \left\| \bar{\bu}_{S} + \bSigma^{(e)}_{S,T} \bbeta_T^* - \bSigma^{(e)}_S \bDelta_S\right\|_2^2 \\
	&= \sum_{e\in \mathcal{E}} \omega^{(e)} \left\| \bSigma_S(\bbeta^{(e, S)} - \bbeta^* - \bDelta_S) \right\|^2.
\end{align*} Putting these pieces together with \cref{cond1-well-condition}, we have
\begin{align}
	\mathsf{T}_1(\bbeta) + \gamma \mathsf{T}_2(\bbeta) & \ge (1-\epsilon)  \|\bar{\bSigma}^{1/2}\bDelta_{\bar{S}}\|_2^2 - \varepsilon^{-1} \kappa_L^{-1} \|\bar{\bu}_{S}\|_2^2 + \gamma \sum_{e\in \mathcal{E}} \omega^{(e)} \kappa_L^2 \| \bbeta^{(e,S)} - \bbeta^* - \bDelta_S\|^2_2 \nonumber \\
	&\ge (1-\epsilon) \|\bar{\bSigma}^{1/2}\bDelta_{\bar{S}}\|_2^2 - \varepsilon^{-1} \kappa_L^{-1} \mathsf{b}_{S} + \inf_{\bv \in \mathbb{R}^{|S|}}\gamma \sum_{e\in \mathcal{E}} \omega^{(e)} \kappa_L^2 \| \bbeta^{(e,S)}_S - \bbeta^*_S - \bv\|^2_2 \label{eq:proof-thm1-mixed-lb}
\end{align}
Define 
\begin{align*}
	L(\bv) = \sum_{e\in \mathcal{E}} \omega^{(e)} \| \bbeta^{(e,S)}_S - \bbeta^*_S - \bv\|^2_2.
\end{align*} It is obvious that $L(\bv)$ is of a quadratic form and is strong convex with curvature $1$. So it has the unique global minimizer $\bv^*=\sum_{e\in \mathcal{E}}\omega^{(e)} \bbeta^{(e,S)} -\bbeta_S^*$ with minimum value
\begin{align}
\label{eq:j-to-bs}
	\inf_{\bv \in \mathbb{R}^{|S|}} L(\bv) = \sum_{e\in \mathcal{E}} \omega^{(e)} \|\bbeta_S^{(e,S)} - \bbeta_S^* - \bv^*\|_2^2 = \bar{\mathsf{d}}_S.
\end{align} Substituting it back into \eqref{eq:proof-thm1-mixed-lb} gives
\begin{align}
\label{eq:thm1-proof-decomp}
	\mathsf{T}_1(\bbeta) + \gamma \mathsf{T}_2(\bbeta) &\ge (1 - \epsilon) \|\bar{\bSigma}^{1/2}\bDelta\|_2^2 - \epsilon^{-1} \kappa_L^{-1} \mathsf{b}_S + \gamma \kappa_L^2 \mathtt{d}_{S} \\
	&= (1-\epsilon) \|\bar{\bSigma}^{1/2}\bDelta\|_2^2 - \epsilon^{-1} \kappa_L^{-1} \mathsf{b}_S + \epsilon^{-1} \kappa_L^2 \bar{\mathsf{d}}_S \gamma^* +  (\gamma - \epsilon^{-1} \gamma^*) \kappa_L^2 \bar{\mathsf{d}}_S \nonumber \\
	&\ge (1-\epsilon) \|\bar{\bSigma}^{1/2}\bDelta\|_2^2 + (\gamma - \epsilon^{-1} \gamma^*) \kappa_L^2 \bar{\mathsf{d}}_S. \nonumber
\end{align} This completes the proof. \qed

\subsection{Proof of \cref{prop:ex-h-var}}

We first calculate serval quantities of interests and determine the original curvature $\kappa_L$.

It follows from the model \eqref{eq:model-linear-scm-varying-v0} that
\begin{align*}
\bx^{(e)} = \bB \bx^{(e)} + 	\balpha (\bbeta^*)^\top \bx^{(e)} + \balpha \varepsilon^{(e)}_0 + \bvarepsilon_x.
\end{align*} Because the induced graph is acyclic, then there exists a permutation $\pi: [p] \to [p]$ such that
\begin{align*}
	\left[\bB + \balpha (\bbeta^*)^\top \right]_{i,j} = 0 ~~~ \text{if} ~~~ \pi(i) < \pi(j),
\end{align*} that is, there exists a permutation matrix $\bP$ such that the matrix 
\begin{align*}
	\tilde{\bV} = \bP [\bB + \balpha (\bbeta^*)^\top] \bP^\top
\end{align*} satisfies $\tilde{\bV}$ is a lower triangular matrix with all zeros on the diagonal. Hence, the inverse of the matrix $\bP \bI \bP^\top - \tilde{\bV}$ exists and is an upper triangular matrix with all ones on the diagonal, thus implying
\begin{align*}
	\bW = (\bI - \bB + \balpha (\bbeta^*)^\top)^{-1} ~~\text{exists with}~~ \nu_{\min}(\bW) \ge 1.
\end{align*} Therefore, we can represent the covariate $\bx^{(e)}$ in terms of all the exogenous variables $\bvarepsilon_x$ and $\varepsilon_0^{(e)}$ as
\begin{align}
\label{eq:proof-ex2-eq-x}
	\bx^{(e)} = \bW (\bvarepsilon_x + \balpha \varepsilon_0^{(e)}).
\end{align} Therefore, denote $\bu = \bW \balpha$, we further have
\begin{align}
\bSigma^{(e)} &= \bW \bD \bW^\top + v_0^{(e)} \bu \bu^\top,\\
\mathbb{E}\left[\varepsilon^{(e)} \bx^{(e)}\right] &= v_0^{(e)} \bu^{(e)}
\end{align}
Given these quantities, we can write down the bias for any $S$ as
\begin{align}
\label{eq:proof-ex2-b}
	\mathsf{b}_S = \left\|\frac{1}{2} \left(\mathbb{E}[\varepsilon^{(e)} \bx_S^{(e)}]\right) (v_0^{(1)} + v_0^{(2)}) \right\|_2^2 = \Bigg(\frac{v_0^{(1)} + v_0^{(2)}}{2}\Bigg)^2 \cdot \|\bu_S\|_2^2.
\end{align}
The calculation of bias-difference term $\bar{\mathsf{d}}_S$ is more involved. Denote $T=S^*\setminus S$, then one has
\begin{align*}
	\bbeta^{(e,S)}_S - \bbeta^*_S = (\bSigma_S^{(e)})^{-1} \mathbb{E}[\varepsilon^{(e)} \bx_S^{(e)}] + (\bSigma_S^{(e)})^{-1} \bSigma_{S,T}^{(e)} \bbeta_T^* = \mathsf{T}_1^{(e)} + \mathsf{T}_2^{(e)}.
\end{align*} Denoting $\bM_S = \bW_{S,:} \bD \bW_{S,:}^\top$, it follows from Sherman-Morrison formula that
\begin{align*}
	\mathsf{T}_1^{(e)} &= (\bM_S + v_0^{(e)} \bu_S \bu_S)^{-1} \bu_S v_0^{(e)}\\
	&= \left\{\bM_S^{-1} - \frac{\bM_S^{-1} \bu_S \bu_S^\top \bM_S^{-1} v_0^{(e)}}{1+\bu_S \bM_S^{-1} \bu_S v_0^{(e)}}\right\} \bu_S v_0^{(e)} \\
	&= \frac{v_0^{(e)}}{1 + \bu_S \bM_S^{-1} \bu_S v_0^{(e)}} \bM_S^{-1} \bu_S.
\end{align*}
Observe that $\bu_T = 0$. This yields
\begin{align*}
	\mathsf{T}_2^{(e)} &= \left(\bM_S + v_0^{(e)} \bu_S \bu_S^\top \right)^{-1} (\bW_{S,:} \bD \bW_{T,:}^\top) \bbeta_T^* \\
	&= \bM_S^{-1} (\bW_{S,:} \bD \bW_{T,:}^\top) \bbeta_T^* - \frac{\bM_S^{-1} \bu_S^{(e)}}{1+v_0^{(e)} \bu_S^\top\bM_S^{-1} \bu_S} v_0^{(e)} \bu_S^\top \bM_S^{-1} (\bW_{S,:} \bD \bW_{T,:}^\top) \bbeta_T^*
\end{align*}
Putting these pieces together, denote $\iota_S = \bu_S \bM_S^{-1} \bu_S$, we have 
\begin{align}
	\bar{\mathsf{d}}_S &= \sum_{e\in \{1,2\}} \frac{1}{2} \left\|\bbeta^{(e,S)}- \frac{\bbeta^{(1,S)} + \bbeta^{(2,S)}}{2}\right\|_2^2 \nonumber\\
				 &= \frac{1}{2} \left\|\bbeta^{(1,S)} - \bbeta^{(2,S)}\right\|_2^2 \nonumber\\
				 &= \frac{1}{2} \|\bM_S^{-1} \bu_S\|_2^2 \left|1 - \bu_S^\top \bM_S^{-1} (\bW_{S,:} \bD \bW_{T,:}^\top) \bbeta_T^*\right|^2 \left|\frac{v_0^{(1)}}{1+\iota_S v_0^{(1)}} - \frac{v_0^{(2)}}{1+\iota_S v_0^{(2)}} \right|^2 \nonumber\\
				 &\ge \frac{1}{2} c_1^{-2} \|\bu_S\|_2^2 \frac{|v_0^{(1)} - v_0^{(2)}|^2 \left|1 - \bu_S^\top \bM_S^{-1} (\bW_{S,:} \bD \bW_{T,:}^\top) \bbeta_T^*\right|^2 }{(1 + \iota_S v_0^{(1)})^2(1 + \iota_S v_0^{(2)})^2} \nonumber\\
				 &\ge \frac{C^{-1}}{2} \|\bu_S\|_2^2 \frac{|v_0^{(1)} - v_0^{(2)}|^2 \left|1 - \bu_S^\top \bM_S^{-1} (\bW_{S,:} \bD \bW_{T,:}^\top) \bbeta_T^*\right|^2 }{(1 + v_0^{(1)})^2(1 + v_0^{(2)})^2} \left\{\min_{i} D_{i,i} \right\}^{4} \label{eq:proof-ex2-lb-d}
\end{align} where the last inequality follows from the fact
\begin{align*}
	\iota_S \le \|\bM_S\|^{-1} \|\bu_S\|^2 \le \{\nu_{\min}(\bW)\}^2 \left\{\min_{i} D_{i,i} \right\}^{-1} \|\bW\|_2^2 \|\balpha\|_2^2 \lesssim \left\{\min_{i} D_{i,i} \right\}^{-1}.
\end{align*} Plugging the above quantities \eqref{eq:proof-ex2-b} and \eqref{eq:proof-ex2-lb-d} into \eqref{eq:thm1-gamma-star} completes the proof.

\section{Proofs for Non-asymptotic Results}
\label{sec:proof-nonasymptotic}

\subsection{Preliminaries}

We first define some concepts that will be used throughout the proof.

\begin{definition}[Sub-Gaussian Random Variable]
A random variable $X$ is a sub-Gaussian random variable with parameter $\sigma \in \mathbb{R}^+$ if 
\begin{align*}
    \forall \lambda \in \mathbb{R}, ~~~~~~ \mathbb{E} \left[\exp(\lambda X)\right] \le \exp\left(\frac{\lambda^2}{2}\sigma^2\right)
\end{align*}
\end{definition}

\begin{definition}[Sub-Exponential Random Variable]
    A random variable $X$ is a sub-Exponential random variable with parameter $(\nu,\alpha) \in \mathbb{R}^+ \times \mathbb{R}^+$ if
    \begin{align*}
        \forall |\lambda|<1/\alpha, ~~~~~~ \mathbb{E}\left[\exp(\lambda X)\right] \le \exp\left(\frac{\lambda^2}{2} \nu^2\right).
    \end{align*}
\end{definition}

It is easy to verify that the product of two sub-Gaussian random variables is a Sub-Exponential random variable, and the dependence of the parameters can be written as follows.

\begin{lemma}[Product of Two Sub-Gaussian Random Variables]
\label{lemma:product-sub-gaussian}
    Suppose $X_1$ and $X_2$ are two zero-mean sub-Gaussian random variables with parameters $\sigma_1$ and $\sigma_2$, respectively. Then $X_1 X_2$ is a sub-exponential random variable with parameter $(c_1 \sigma_1\sigma_2, c_2 \sigma_1\sigma_2)$, where $c_1,c_2>0$ are some universal constants.
\end{lemma}

We also have the following lemma stating the concentration inequality for the sum of independent sub-exponential random variables.
\begin{lemma}[Sum of Independent Sub-exponential Random Variables]
\label{lemma:sum-sub-exp}
Suppose $X_1,\ldots, X_N$ are independent sub-exponential random variables with parameters $\{(\nu_i, \alpha_i)\}_{i=1}^N$, respectively. There exists some universal constant $c_1$ such that the following holds,
\begin{align*}
    \mathbb{P}\left[\left|\sum_{i=1}^N (X_i - \mathbb{E}[X_i]) \right| \ge c_1 \left\{\sqrt{t\times \sum_{i=1}^N \nu_i^2} + t\times \max_{i\in [N]} \alpha_i \right\}\right] \le 2e^{-t}.
\end{align*}
\end{lemma}

We also define the following quantity regarding sample size
\begin{align}
\label{eq:def-n-omega}
n_{\bomega} = \Bigg(\sum_{e\in \mathcal{E}} \frac{(\omega^{(e)})^2}{n^{(e)}}\Bigg)^{-1}.
\end{align} It is easy to see that $n_{\bomega} \ge n_*$.

We also define
\begin{align}
\mathcal{B}(S) = \left\{\bbeta \in \mathbb{R}^p, \supp(\bbeta) = S, \|\bbeta\|_2=1\right\} \qquad \text{and} \qquad \mathcal{B}_s = \bigcup_{S\subseteq [p], |S| \le s} \mathcal{B}(S). \label{eq:ball}
\end{align}

\subsection{Proof of \cref{prop:pooled-least-squares}}

It follows from the definition of $\mathsf{R}(\bbeta)$ and $\mathsf{R}^{(e)}(\bbeta)$ that
\begin{align*}
    \mathsf{R}(\bbeta) &= \sum_{e\in \mathcal{E}} \omega^{(e)} \mathsf{R}^{(e)}(\bbeta) \\
    &= \sum_{e\in \mathcal{E}} \omega^{(e)} \left\{(\bbeta - \bbeta^*)^\top \bSigma^{(e)} (\bbeta - \bbeta^*) - 2(\bbeta - \bbeta^*)^\top \mathbb{E}[\bx^{(e)} \varepsilon^{(e)}] + \mathbb{E} |\varepsilon^{(e)}|^2 \right\}\\
    &= (\bbeta - \bbeta^*)^\top \bar{\bSigma} (\bbeta - \bbeta^*) - 2(\bbeta - \bbeta^*)^\top \sum_{e\in \mathcal{E}} \omega^{(e)} \mathbb{E} [\bx^{(e)} \varepsilon^{(e)}] + \sum_{e\in \mathcal{E}} \omega^{(e)} \mathbb{E} |\varepsilon^{(e)}|^2.
\end{align*}
We can see $\mathsf{R}(\bbeta)$ is of quadratic form of $\bbeta - \bbeta^*$. Since $\lambda_{\min}(\bSigma) > 0$, the minimizer of $\mathsf{R}(\bbeta)$ is unique and is
\begin{align}
\label{eq:beta-bar-r}
    \bar{\bbeta}^{\mathsf{R}} = \bbeta^* + (\bar{\bSigma})^{-1} \sum_{e\in \mathcal{E}} \omega^{(e)} \mathbb{E}[\bx^{(e)} \varepsilon^{(e)}].
\end{align}
Combining the above definition of $\bar{\bbeta}^{\mathsf{R}}$ together with \cref{cond1-well-condition} validates the upper bound and lower bound on $\|\bar{\bbeta}^{\mathsf{R}}-\bbeta^*\|_2$. Moreover, plugging in above \eqref{eq:beta-bar-r} into $\mathsf{R}(\bbeta)$ following by some calculations gives
\begin{align*}
    \mathsf{R}(\bbeta) - \mathsf{R}(\bar{\bbeta}^{\mathsf{R}}) = (\bbeta - \bbeta^*)^\top \bar{\bSigma} (\bbeta - \bbeta^*),
\end{align*} which completes the proof of claim \eqref{eq:pooled-least-square-strong-convexity}.

Moreover, for any $\bbeta \in \mathbb{R}^p$, we have
\begin{align*}
    \|\bar{\bSigma}^{1/2}(\bbeta - \bar{\bbeta}^{\mathsf{R}})\|_2^2 &= \mathsf{R}(\bbeta) - \mathsf{R}(\bar{\bbeta}^{\mathsf{R}}) \\
    &= \mathsf{R}(\bbeta) - \hat{\mathsf{R}}(\bbeta) + \hat{\mathsf{R}}(\bbeta) - \hat{\mathsf{R}}(\bar{\bbeta}^{\mathsf{R}}) + \hat{\mathsf{R}}(\bar{\bbeta}^{\mathsf{R}}) - \mathsf{R}(\bar{\bbeta}^{\mathsf{R}}).
\end{align*}

We argue that
\begin{align}
\label{eq:pooled-least-squares-high-prob-bound1}
\begin{split}
    \mathbb{P}[\mathcal{C}_{1,t}] = \mathbb{P} \Bigg\{ \forall \bbeta \in \mathbb{R}^p, ~~& \mathsf{R}(\bbeta) - \mathsf{R}(\bar{\bbeta}^{\mathsf{R}}) - \left(\hat{\mathsf{R}}(\bbeta) - \hat{\mathsf{R}}(\bar{\bbeta}^{\mathsf{R}}) \right) \\
    &~~~~~~ \le C \sigma_x \left(\sigma_\varepsilon + \sigma_x\|\bar{\bSigma}^{1/2}(\bbeta^* - \bar{\bbeta}^{\mathsf{R}})\|_2 \right) \sqrt{\frac{p+t}{n_*}} \|\bar{\bSigma}^{1/2}(\bbeta - \bar{\bbeta}^{\mathsf{R}})\|_2 \\
    &~~~~~~~~~~~~ + C \sigma_x^2 \sqrt{\frac{p + t}{n_*}} \|\bar{\bSigma}^{1/2}(\bbeta - \bar{\bbeta}^{\mathsf{R}})\|_2^2\Bigg\}\ge 1-2e^{-t}
\end{split}
\end{align} for any $t \in (0, n_* - p]$, which will be validated later. If such a claim holds, then under $\mathcal{C}_{1,t}$ which occurs with probability at least $1-2e^{-t}$, we have that, for any $\bbeta \in \mathbb{R}^p$
\begin{align*}
        \|\bar{\bSigma}^{1/2}(\bbeta - \bar{\bbeta}^{\mathsf{R}})\|_2^2 & \le \hat{\mathsf{R}}(\bbeta) - \hat{\mathsf{R}}(\bar{\bbeta}^{\mathsf{R}}) + C \sigma_x^2 \sqrt{\frac{p + t}{n_*}} \|\bar{\bSigma}^{1/2}(\bbeta - \bar{\bbeta}^{\mathsf{R}})\|_2^2 \\
        &~~~~~~ + C \sigma_x \left(\sigma_\varepsilon + \sigma_x\|\bar{\bSigma}^{1/2}(\bbeta^* - \bar{\bbeta}^{\mathsf{R}})\|_2 \right) \sqrt{\frac{p+t}{n_*}} \|\bar{\bSigma}^{1/2}(\bbeta - \bar{\bbeta}^{\mathsf{R}})\|_2.
\end{align*} Suppose further that $n^* \ge (p + t) 4C^2\sigma_x^4$, we can apply the above inequality with $\hat{\bbeta}_\mathsf{R}$ which satisfies $\hat{\mathsf{R}}(\hat{\bbeta}_\mathsf{R}) - \hat{\mathsf{R}}(\bar{\bbeta}^{\mathsf{R}})\le 0$ and obtain
\begin{align*}
\|\bar{\bSigma}^{1/2}(\hat{\bbeta}_\mathsf{R} - \bar{\bbeta}^{\mathsf{R}})\|_2^2 \le 2C \sigma_x \left(\sigma_\varepsilon + \sigma_x\|\bar{\bSigma}^{1/2}(\bbeta^* - \bar{\bbeta}^{\mathsf{R}})\|_2 \right) \sqrt{\frac{p+t}{n_*}} \|\bar{\bSigma}^{1/2}(\hat{\bbeta}_\mathsf{R} - \bar{\bbeta}^{\mathsf{R}})\|_2.
\end{align*} This completes the proof.

\noindent \emph{Proof of the Bound \eqref{eq:pooled-least-squares-high-prob-bound1}.} It follows from the definition of $\hat{\mathsf{R}}$ and the above identity that \begin{align*}
    &\mathsf{R}(\bbeta) - \mathsf{R}(\bar{\bbeta}^{\mathsf{R}}) - \left(\hat{\mathsf{R}}(\bbeta) - \hat{\mathsf{R}}(\bar{\bbeta}^{\mathsf{R}}) \right) \\
    &~~~~~~ = (\bbeta - \bar{\bbeta}^{\mathsf{R}})^\top \bar{\bSigma} (\bbeta - \bar{\bbeta}^{\mathsf{R}}) - (\bbeta - \bar{\bbeta}^{\mathsf{R}})^\top \Bigg(\sum_{e\in \mathcal{E}} \omega^{(e)}\hat{\bSigma}^{(e)}\Bigg) (\bbeta - \bar{\bbeta}^{\mathsf{R}}) \\
    &~~~~~~~~~~~~ - 2(\bbeta - \bar{\bbeta}^{\mathsf{R}})^\top \sum_{e\in \mathcal{E}} \omega^{(e)} \hat{\mathbb{E}} \left[\bx^{(e)} \left((\bar{\bbeta}^{\mathsf{R}})^\top \bx^{(e)} - y^{(e)}\right)\right].
\end{align*}
When $p + t \le n_*$, define the following event
\begin{align*}
    \mathcal{K}_{1,t} &= \left\{ \left\|\bI - \bar{\bSigma}^{-1/2} \sum_{e\in \mathcal{E}} \omega^{(e)}\hat{\bSigma}^{(e)} \bar{\bSigma}^{-1/2}\right\|_2 \le C_1 \sigma_x^2 \sqrt{\frac{p+t}{n_*}}\right\} \\
    \mathcal{K}_{2,t} &= \Bigg\{ \left\|\sum_{e\in \mathcal{E}} \omega^{(e)} \bar{\bSigma}^{-1/2} \hat{\mathbb{E}} \left[\bx^{(e)} \left((\bar{\bbeta}^{\mathsf{R}})^\top \bx^{(e)} - y^{(e)}\right)\right]\right\| \\    
    &~~~~~~~~~~~~~~~~~~~~~~~~~~\le C_2 \sigma_x \left(\sigma_\varepsilon + \sigma_x\|\bar{\bSigma}^{1/2}(\bbeta^* - \bar{\bbeta}^{\mathsf{R}})\|_2 \right) \sqrt{\frac{p+t}{n_*}}\Bigg\}
\end{align*} for some universal constant $C_1$--$C_2$ to be determined. It suffices to prove that $\mathbb{P}[\mathcal{K}_{1,t}] \land \mathbb{P}[\mathcal{K}_{2,t}] \ge 1-e^{-t}$ for any $t>0$. If the two claims are verified, then we can conclude that under the event $\mathcal{K}_{1,t} \cap \mathcal{K}_{2,t}$ that occurs with probability at least $1-2e^{-t}$, the following holds
\begin{align*}
&\mathsf{R}(\bbeta) - \mathsf{R}(\bar{\bbeta}^{\mathsf{R}}) - \left(\hat{\mathsf{R}}(\bbeta) - \hat{\mathsf{R}}(\bar{\bbeta}^{\mathsf{R}}) \right) \\
    &~~~~~~ = \left\|\bar{\bSigma}^{1/2}(\bbeta - \bar{\bbeta}^{\mathsf{R}})\right\|_2 \left\|\bI - \bar{\bSigma}^{-1/2} \sum_{e\in \mathcal{E}} \omega^{(e)}\hat{\bSigma}^{(e)} \bar{\bSigma}^{-1/2} \right\|_2 \left\|\bar{\bSigma}^{1/2}(\bbeta - \bar{\bbeta}^{\mathsf{R}})\right\|_2 \\
    &~~~~~~~~~~~~ + 2\left\|\bar{\bSigma}^{1/2}(\bbeta - \bar{\bbeta}^{\mathsf{R}})\right\|_2 \left\|\sum_{e\in \mathcal{E}} \omega^{(e)} \bar{\bSigma}^{-1/2} \hat{\mathbb{E}} \left[\bx^{(e)} \left((\bar{\bbeta}^{\mathsf{R}})^\top \bx^{(e)} - y^{(e)}\right)\right]\right\|_2 \\
    &~~~~~~ \le C' \sigma_x^2 \sqrt{\frac{p+t}{n_*}} \|\bar{\bSigma}^{1/2}(\bbeta - \bar{\bbeta}^{\mathsf{R}})\|_2^2  \\
    &~~~~~~~~~~~~~~~ + C'\sigma_x \left(\sigma_\varepsilon + \sigma_x\|\bar{\bSigma}^{1/2}(\bbeta^* - \bar{\bbeta}^{\mathsf{R}})\|_2 \right) \sqrt{\frac{p+t}{n_*}} \|\bar{\bSigma}^{1/2}(\bbeta - \bar{\bbeta}^{\mathsf{R}})\|_2,
\end{align*} which completes the proof of the claim \eqref{eq:pooled-least-squares-high-prob-bound1}.

\noindent {\sc Step 1. High Probability Bound for $\mathcal{K}_{1,t}$. } Let $\bA = \bI - \bar{\bSigma}^{-1/2} \sum_{e\in \mathcal{E}} \omega^{(e)}\hat{\bSigma}^{(e)} \bar{\bSigma}^{-1/2}$, similar to the derivation in \eqref{eq:step3-sup-argument-2}, we can construct $N$ pairs of $p$-dimensional unit vectors $(\bv_1, \bu_1),\ldots, (\bv_N, \bu_N)$ with $N\le 8100^p$ such that
\begin{align*}
    \|\bA\|_2 \le 4\sup_{k\in [N]} \bv_k^\top \bA \bu_k.
\end{align*} For fixed $(\bv, \bu)$, using the identity
\begin{align*}
    \bI = \sum_{e\in \mathcal{E}} \omega^{(e)} \bar{\bSigma}^{-1/2} \bSigma^{(e)} \bar{\bSigma}^{-1/2},
\end{align*} we obtain
\begin{align*}
    \bv^\top \bA \bu &= \sum_{e\in \mathcal{E}} \sum_{i=1}^{n^{(e)}} \frac{\omega^{(e)}}{n^{(e)}} \left\{\bu^\top \bar{\bSigma}^{-1/2} \bSigma^{(e)} \bar{\bSigma}^{-1/2} \bv - \bu^\top \bar{\bSigma}^{-1/2} \bx^{(e)}_i (\bx^{(e)}_i)^\top \bar{\bSigma}^{-1/2} \bv \right\} \\ 
    &= \sum_{e\in \mathcal{E}} \sum_{i=1}^{n^{(e)}} \frac{\omega^{(e)}}{n^{(e)}} \left(\mathbb{E}[X_{e,i}] - X_{e,i}\right)
\end{align*} for $X_{e,i} = \bu^\top \bar{\bSigma}^{-1/2} \bx^{(e)}_i (\bx^{(e)}_i)^\top \bar{\bSigma}^{-1/2} \bv$. It follows from \cref{cond2-subgaussain-x} that $X_{e,i}$ is the product of two zero-mean sub-Gaussian random variables with parameter $\sigma_x$. Then it follows from \cref{lemma:product-sub-gaussian} and \cref{lemma:sum-sub-exp} that the following holds
\begin{align*}
    \mathbb{P}\left[|\bv_k^\top \bA \bu_k| \ge C_3 \sigma_x \left(\sqrt{\frac{u}{n_{\bomega}}} + \frac{u}{n_*}\right)\right] \le 2e^{-u}
\end{align*} for all the $u>0$, where $C_3$ is some universal constant. Applying union bounds further gives
\begin{align*}
    \mathbb{P}\left[\sup_{k\in [N]} |\bv_k^\top \bA \bu_k| \ge C_3 \sigma_x \left(\sqrt{\frac{u}{n_{\bomega}}} + \frac{u}{n_*}\right)\right] \le 2Ne^{-u}.
\end{align*} So it concludes the proof by setting $u=t+\log(2N) = t+C_4 p$ with some $C_4 > 1$ and observing that
\begin{align*}
    \sqrt{\frac{t + C_4 p}{n_{\bomega}}} + \frac{t + C_4 p}{n_*} \le 2C_4 \sqrt{\frac{t + p}{n_*}}
\end{align*} provided $p + t\le n_*$.

\noindent {\sc Step 2. High Probability Bound for $\mathcal{K}_{2,t}$. } 

Let $\bb = \sum_{e\in \mathcal{E}} \omega^{(e)} \bar{\bSigma}^{-1/2} \hat{\mathbb{E}} \left[\bx^{(e)} \left((\bar{\bbeta}^{\mathsf{R}})^\top \bx^{(e)} - y^{(e)}\right)\right]$, the proof is similar to the first part of {\sc Step 5} in the proof of \cref{lemma:instance-dependent-bound-j-lowdim}. To be specific, there exist $N=90^p$ $p$-dimensional unit vectors $\bu_1,\ldots, \bu_N$ such that $\|\bb\|_2 \le 2\sup_{k\in [N]} \bu_k^\top \bb$. The key here is to decompose $\bu^\top \bb$ as the sum of zero-mean independent random variables. Denote
\begin{align*}
    X_{e,i} = \bu^\top \bar{\bSigma}^{-1/2} \bx^{(e)} \left((\bx^{(e)})^\top (\bar{\bbeta}^{\mathsf{R}} - \bbeta^*) - \varepsilon^{(e)}\right).
\end{align*} We have
\begin{align*}
    \bu^\top \bb = \sum_{e\in \mathcal{E}} \sum_{i=1}^{n^{(e)}} \frac{\omega^{(e)}}{n^{(e)}} X_{e,i} &= \sum_{e\in \mathcal{E}} \sum_{i=1}^{n^{(e)}} \frac{\omega^{(e)}}{n^{(e)}} (X_{e,i} - \mathbb{E}[X_{e,i}]) + \sum_{e\in \mathcal{E}} \omega^{(e)} \mathbb{E}[X_{e,1}] \\ &= \sum_{e\in \mathcal{E}} \sum_{i=1}^{n^{(e)}} \frac{\omega^{(e)}}{n^{(e)}} (X_{e,i} - \mathbb{E}[X_{e,i}])
\end{align*} where the last equality follows from the definition of $\bar{\bbeta}^{\mathsf{R}}$ \eqref{eq:beta-bar-r} that
\begin{align*}
    \sum_{e\in \mathcal{E}} \omega^{(e)} \mathbb{E}[X_{e,1}] &= \bu^\top \bar{\bSigma}^{-1/2} \bar{\bSigma} (\bar{\bbeta}^{\mathsf{R}} - \bbeta^*) - \bar{\bSigma}^{-1/2} \sum_{e\in \mathcal{E}} \omega^{(e)} \mathbb{E} [\bx^{(e)} \varepsilon^{(e)}] \\
    &= \bu^\top \bar{\bSigma}^{1/2} \left(\bar{\bbeta}^{\mathsf{R}} - \bbeta^* - \bar{\bSigma}^{-1} \sum_{e\in \mathcal{E}} \omega^{(e)} \mathbb{E} [\bx^{(e)} \varepsilon^{(e)}] \right) = 0.
\end{align*} Note that $\bu^\top \bar{\bSigma}^{-1/2} \bx^{(e)}$ is a zero-mean sub-Gaussian random variable with parameter $\sigma_x$, and $(\bx^{(e)})^\top (\bar{\bbeta}^{\mathsf{R}} - \bbeta^*) - \varepsilon^{(e)}$ is the sum of two zero-mean sub-Gaussian random variables with parameter $\sigma_\varepsilon$ and $\sigma_x \|\bSigma^{1/2}(\bar{\bbeta}^{\mathsf{R}}-\bbeta^*)\|_2$, respectively. It then follows from \cref{lemma:product-sub-gaussian} and \ref{lemma:sum-sub-exp} that
\begin{align*}
    \mathbb{P}\left[ |\bu_k^\top \bb| \ge C_5\sigma_x \left(\sigma_{\varepsilon} + \sigma_x \|\bSigma^{1/2}(\bar{\bbeta}^{\mathsf{R}}-\bbeta^*)\|_2\right) \left(\sqrt{\frac{u}{n_{\bomega}}} + \frac{u}{n_*}\right)\right] \ge 1-2e^{-u}.
\end{align*} Applying union bounds over all the $k$ and setting $u=t+\log(2N)$ completes the proof. \qed

\subsection{Proof of \cref{thm:rate-low-dim}}
\label{sec:proof-thm2}

\subsubsection{Key Proof Idea}

We will use the following decomposition in the proof of \cref{thm:rate-low-dim}.
\begin{align}
\label{eq:q-decomposition}
\begin{split}
\mathsf{Q}(\bbeta;\gamma,\bomega) - \mathsf{Q}(\bbeta^*;\gamma,\bomega)  =& ~\mathsf{R}(\bbeta) - \mathsf{R}(\bbeta^*) + \gamma \left(\mathsf{J}(\bbeta) - \mathsf{J}(\bbeta^*)\right) \\
	=& ~\left\{\mathsf{R}(\bbeta) - \mathsf{R}(\bbeta^*)\right\} - \left\{ \hat{\mathsf{R}}(\bbeta) - \hat{\mathsf{R}}(\bbeta^*)\right\} \\ 
		&~~~+ \left\{\hat{\mathsf{R}}(\bbeta) + \gamma \hat{\mathsf{J}}(\bbeta) - \hat{\mathsf{R}}(\bbeta^*) - \gamma^*\hat{\mathsf{J}}(\bbeta^*)\right\} \\
		&~~~+ \gamma \left[\left\{\mathsf{J}(\bbeta) - \mathsf{J}(\bbeta^*)\right\} - \left\{ \hat{\mathsf{J}}(\bbeta) - \hat{\mathsf{J}}(\bbeta^*)\right\}\right] \\
	=& ~\mathsf{T}_{\mathsf{R}}(\bbeta, \bbeta^*) + \mathsf{T}_{\mathsf{J}}(\bbeta, \bbeta^*) + \hat{\mathsf{Q}}(\bbeta;\gamma, \bomega) -  \hat{\mathsf{Q}}(\bbeta^*;\gamma, \bomega).
\end{split}
\end{align}

The analysis of the first term $\mathsf{T}_{\mathsf{R}}$ is standard. As stated in the following Lemma, we can provide an instance-dependent two-side bound for it.
\begin{lemma}[Instance-dependent Two-side Bound for ${\mathsf{R}}$]
\label{lemma:instance-dependent-bound-r-lowdim}
Assume \cref{cond0-model}--\ref{cond3-sub-gaussian-eps} hold. Define the event
\begin{align}
\begin{split}
    \mathcal{A}_{1,t} = \Bigg\{ \forall \bbeta \in \mathbb{R}^p, ~~ &~ \left|\mathsf{R}(\bbeta) - \mathsf{R}(\bbeta^*) - \hat{\mathsf{R}}(\bbeta) + \hat{\mathsf{R}}(\bbeta^*)\right| \\
    & ~~~~~~~ \le c_1 \left(\kappa_U \sigma_x^2 \delta_1 \|\bbeta - \bbeta_*\|_2^2 + \kappa_U^{1/2} \sigma_x \sigma_\varepsilon \delta_1 \|\bbeta - \bbeta_*\|_2 \right)\Bigg\}.
\end{split} 
\end{align} with $\delta_1 = \sqrt{\frac{p + t}{n_{\bomega}}} + \frac{p + t}{n_*}$ and some universal constants $c_1$. We have 
\begin{align*}
\mathbb{P}[\mathcal{A}_{1,t}] \ge 1-e^{-t}
\end{align*} for any $t>0$.
\end{lemma}
\begin{proof}[Proof of \cref{lemma:instance-dependent-bound-r-lowdim}]
    It concludes by applying \cref{lemma:instance-dependent-bound-r-highdim} with $s=p$.
\end{proof}

However, the analysis of the second term $\mathsf{T}_{\mathsf{J}}$ is more involved. The following Lemma provides an instance-dependent one-side bound because all the statistical analysis only cares about the upper bound of $\mathsf{T}_{\mathsf{J}}(\bbeta,\bbeta^*)$, or the lower bound of $-\mathsf{T}_{\mathsf{J}}(\bbeta,\bbeta^*)$.

\begin{lemma}[Instance-dependent One-side Bound for ${\mathsf{J}}$]
\label{lemma:instance-dependent-bound-j-lowdim}
Assume \cref{cond0-model}--\ref{cond3-sub-gaussian-eps} hold. Define the event
\begin{align}
\begin{split}
    \mathcal{A}_{2,t} = \Bigg\{ \forall \bbeta \in \mathbb{R}^p, ~~& \frac{1}{c_1} \left(\mathsf{J}(\bbeta) - \mathsf{J}(\bbeta^*) - \hat{\mathsf{J}}(\bbeta) + \hat{\mathsf{J}}(\bbeta^*)\right) \\
    &\le \|\bbeta - \bbeta_*\|_2^2 \times \kappa_U^2\sigma_x^2\sqrt{\frac{p+t}{n_*}} \\
    &~~~~~~ + \|\bbeta - \bbeta_*\|_2 \times \kappa_U^{3/2}\sigma_x^2 \sigma_\varepsilon \left( \sqrt{\frac{p + t}{n_*}} +\sigma_x \frac{p + \log(2|\mathcal{E}|)+t}{\bar{n}}\right)\\
    &~~~~~~ + \sqrt{\sum_{e\in \mathcal{E}} \omega^{(e)} \left\|\mathbb{E} [\bx_{\supp(\bbeta)}^{(e)} \varepsilon^{(e)}]\right\|_2^2} \times \kappa_U^{1/2} \sigma_x \sigma_\varepsilon \sqrt{\frac{p+t}{n_*}}  \\
    &~~~~~~ + |S^*\setminus \supp(\bbeta)| \times \kappa_U \sigma_x^2\sigma_\varepsilon^2  \frac{t+ \log(2|\mathcal{E}||S^*|)}{\bar{n}} + \kappa_U \sigma_x \sigma_\varepsilon^2 \frac{p+t}{n_*}
    \Bigg\}
\end{split}
\end{align} for some universal constant $c_1$. Then we have
\begin{align}
    \mathbb{P} [\mathcal{A}_{2,t}] \ge 1-6e^{-t}.
\end{align} for any $t\in (0, n_{\min} - \log(2|\mathcal{E}|) - p]$.

\end{lemma}
\begin{proof}[Proof of \cref{lemma:instance-dependent-bound-j-lowdim}]
    It concludes by applying \cref{lemma:instance-dependent-bound-j-highdim} with $s=p$ and observing that $n_{\bomega} \ge n_*$.
\end{proof}

We will use the following characterization of the population-level excess risk.

\begin{proposition}
\label{prop:population-lb}
    Under \cref{cond0-model}, \ref{cond1-well-condition} and \ref{cond-ident}. If $\gamma \ge 3\gamma^*$, then
    \begin{align*}
        \mathsf{Q}(\bbeta;\gamma,\bomega) - \mathsf{Q}(\bbeta^*;\gamma,\bomega) &\ge \frac{\kappa_L}{2} \|\bbeta - \bbeta^*\|_2^2 + \frac{\gamma}{3} \mathsf{J}(\bbeta;\bomega) \\
        &\ge \frac{\kappa_L}{2} \|\bbeta - \bbeta^*\|_2^2 + \frac{\gamma}{6} \kappa_L^2 \bar{\mathsf{d}}_{\supp(\bbeta)} + \frac{\gamma}{6} \mathsf{J}(\bbeta;\bomega).
    \end{align*}
\end{proposition}
\begin{proof}[Proof of \cref{prop:population-lb}] The proof is almost identical to the proof of \cref{thm:global-strong-convexity}. From \eqref{eq:j-to-bs} we know $\mathsf{J}(\bbeta;\bomega) \ge \kappa_L^2\bar{\mathsf{d}}_{\supp(\bbeta)}$. Following the notations in the proof of \cref{thm:global-strong-convexity}, we have
\begin{align*}
    \mathsf{Q}(\bbeta;\gamma,\bomega) - \mathsf{Q}(\bbeta^*;\gamma,\bomega) &\ge \mathsf{T}_1(\bbeta) + \gamma \mathsf{T}_2(\bbeta) \\
    &\ge \frac{\kappa_L}{2} \|\bbeta-\bbeta^*\|_2^2 - 2 \kappa_L^{-1} \mathsf{b}_{\supp(\bbeta)} + \frac{2}{3} \gamma \bar{\mathsf{d}}_{\supp(\bbeta)} + \frac{\gamma}{3} \mathsf{J}(\bbeta;\bomega)\\
    &\ge \frac{\kappa_L}{2} \|\bbeta-\bbeta^*\|_2^2 + \frac{\gamma}{3} \mathsf{J}(\bbeta;\bomega)
\end{align*} This completes the proof.
\end{proof}

With the help of the above claims, we are ready to prove \cref{thm:rate-low-dim}.

\subsubsection{Proof of the Variable Selection Property \eqref{eq:low-dim-variable-selection}}

It follows from our decomposition \eqref{eq:q-decomposition} that, for any $\bbeta \in \mathbb{R}^p$,
\begin{align}
\label{eq:low-dim-variable-selection-pf1}
    \hat{\mathsf{Q}}(\bbeta;\gamma, \bomega) - \hat{\mathsf{Q}}(\bbeta^*;\gamma, \bomega)
    \ge - \mathsf{T}_{\mathsf{J}}(\bbeta, \bbeta^*) - \mathsf{T}_{\mathsf{R}}(\bbeta, \bbeta^*) + \mathsf{Q}(\bbeta; \gamma, \bomega) - \mathsf{Q}(\bbeta^*; \gamma, \bomega).
\end{align}

The rest of the proof proceeds conditioned on that $n_{\min} > \log(2|\mathcal{E}|) + p + t$ such that the results of \cref{lemma:instance-dependent-bound-j-lowdim} is applicable. Recall that $\gamma \ge 1 + \kappa_L$. Applying \cref{lemma:instance-dependent-bound-r-lowdim} and \cref{lemma:instance-dependent-bound-j-lowdim}, we have that, under the event $\mathcal{A}_{1,t} \cap \mathcal{A}_{2,t}$, 
\begin{align}
\label{eq:low-dim-tj-tr-bound}
\begin{split}
    \mathsf{T}_{\mathsf{J}}(\bbeta, \bbeta^*) + \mathsf{T}_{\mathsf{R}}(\bbeta, \bbeta^*) &\le C\gamma \Bigg\{\|\bbeta - \bbeta^*\|^2_2 \times \kappa_U^2 \sigma_x^2 \sqrt{\frac{p + t}{n_*}} \\
    &~~~~~~~~~ + \|\bbeta - \bbeta_*\|_2 \times \kappa_U^{3/2}\sigma_x^2 \sigma_\varepsilon \left( \sqrt{\frac{p + t}{n_*}} +\sigma_x \frac{p + \log(2|\mathcal{E}|)+t}{\bar{n}}\right)\\
    &~~~~~~~~~ + \sqrt{\sum_{e\in \mathcal{E}} \omega^{(e)} \left\|\mathbb{E} [\bx_{\supp(\bbeta)}^{(e)} \varepsilon^{(e)}]\right\|_2^2} \times \kappa_U^{1/2} \sigma_x \sigma_\varepsilon \sqrt{\frac{p+t}{n_*}}  \\
    &~~~~~~~~~ + |S^*\setminus \supp(\bbeta)| \times \kappa_U \sigma_x^2\sigma_\varepsilon^2  \frac{t+ \log(2|\mathcal{E}||S^*|)}{\bar{n}} + \kappa_U \sigma_x \sigma_\varepsilon^2 \frac{p+t}{n_*} \Bigg\}\\
    &= C\gamma(\mathsf{I}_1 + \mathsf{I}_2 + \mathsf{I_3} + \mathsf{I_4}).
\end{split}
\end{align}
Applying Young's inequality gives
\begin{align}
\label{eq:low-dim-i2}
\mathsf{I}_2 \le \frac{\kappa_L}{16C\gamma} \|\bbeta - \bbeta^*\|^2_2 + \frac{8C\gamma}{\kappa_L} \kappa_U^3 \sigma_x^4 \sigma_\varepsilon^2 \left\{\frac{p + t}{n_*} + \sigma_x^2 \left(\frac{p + \log(2|\mathcal{E}|) + t}{\bar{n}}\right)^2\right\}.
\end{align} 
At the same time, it follows from \cref{lemma:removecondition7} that
\begin{align*}
    \mathsf{I}_3 &\le \kappa_U^{1/2} \sigma_x \sigma_\varepsilon \sqrt{\frac{p + t}{n_*}} \sqrt{2 \mathsf{J}(\bbeta;\bomega) + 2\kappa_U^2 \|\bbeta - \bbeta^*\|_2^2} \\
    &\le \frac{1}{6C} \mathsf{J}(\bbeta; \bomega) + 6C\kappa_U\sigma_x^2 \sigma_\varepsilon^2 \frac{p + t}{n_*} + \frac{\kappa_L}{16C\gamma} \|\bbeta - \bbeta^*\|_2^2 + \frac{16C\gamma}{\kappa_L}\kappa_U^3 \sigma_x^2 \sigma_\varepsilon^2 \frac{p + t}{n_*}
\end{align*}
Using the fact that
\begin{align*}
|S^*\setminus \supp(\bbeta)| \min_{j\in S^*} |\beta_j^*|^2 &= \sum_{j\in S^*, \beta_j=0} \min_{j\in S^*} |\beta_j^*|^2 \\
&\le \sum_{j\in S^*, \beta_j=0} |\beta_j - \beta^*_j|^2 \le \sum_{j=1}^p |\beta_j - \beta_j^*|^2 = \|\bbeta - \bbeta^*\|_2^2,
\end{align*} we obtain \begin{align}
&|S^*\setminus \supp(\bbeta)| \times \kappa_U \sigma_x^2\sigma_\varepsilon^2  \frac{t+ \log(2|\mathcal{E}||S^*|)}{\bar{n}} \nonumber \\
&~~~~~~~~~ = \left(|S^*\setminus \supp(\bbeta)| \min_{j\in S^*}|\beta_j^*|^2 \right) \times \left\{\frac{\kappa_U \sigma_x^2\sigma_\varepsilon^2}{ \min_{j\in S^*} |\bbeta_j^ *|^2}\frac{t+ \log(2|\mathcal{E}||S^*|)}{\bar{n}}\right\} \nonumber \\
&~~~~~~~~~ \le \|\bbeta - \bbeta_*\|_2^2  \left\{\frac{\kappa_U \sigma_x^2\sigma_\varepsilon^2}{ \min_{j\in S^*} |\bbeta_j^ *|^2}\frac{t+ \log(2|\mathcal{E}||S^*|)}{\bar{n}}\right\}, \label{eq:low-dim-bound-i4-1}
\end{align} 

Putting these pieces together, if 
\begin{align}
\label{eq:low-dim-cond-n-1}
    \bar{n} \ge \frac{16 C\gamma \kappa_U \sigma_x^2\sigma_\varepsilon^2 \{t + \log(2|\mathcal{E}||S^*|)\}}{ \kappa_L \min_{j\in S^*} |\beta_j^*|^2} ~~~~\text{and}~~~~ n_* \ge (p + t) \left(\frac{16C \kappa_U^2 \sigma_x^2}{\kappa_L }\gamma\right)^2,
\end{align} then we find that, under $\mathcal{A}_{1,t} \cap \mathcal{A}_{2,t}$,
\begin{align*}
    \mathsf{T}_{\mathsf{J}}(\bbeta, \bbeta^*) + \mathsf{T}_{\mathsf{R}}(\bbeta, \bbeta^*) &\le \frac{\kappa_L}{4} \|\bbeta - \bbeta^*\|_2^2 + \frac{\gamma}{6} \mathsf{J}(\bbeta;\omega) + C_1 \gamma\Bigg\{\frac{\gamma}{\kappa_L} \kappa_U^3 \sigma_x^4 \sigma_\varepsilon^2  \frac{p+t}{n_*} \\
    &~~~~~~~~~~~~~~~~ + \frac{\gamma}{\kappa_L} \kappa_U^3 \sigma_x^6 \sigma_\varepsilon^2 \left(\frac{p + \log(2|\mathcal{E}|) + t}{\bar{n}}\right)^2 \Bigg\}.
\end{align*} for another universal constant $C_1$.

Plugging the above inequality back into \eqref{eq:low-dim-variable-selection-pf1} and applying \cref{prop:population-lb} provided our choice of $\gamma \ge 3\gamma^*$, we establish that, if $\mathcal{A}_{1,t} \cap \mathcal{A}_{2,t}$ occurs, then for any $\bbeta \in \mathbb{R}^p$,
\begin{align*}
    \hat{\mathsf{Q}}(\bbeta) - \hat{\mathsf{Q}}(\bbeta^*) \ge &~\frac{\kappa_L}{4} \|\bbeta - \bbeta^*\|_2^2 + \frac{\gamma}{6} \kappa_L^2 \bar{\mathsf{d}}_{\supp(\bbeta)} \\
    &~~~~~~~~~~~~ - C_2 \Bigg\{ \frac{\gamma^2}{\kappa_L} \kappa_U^3 \sigma_x^6 \sigma_\varepsilon^2 \left(\frac{p + \log(2|\mathcal{E}|) + t}{\bar{n}}\right)^2 + \frac{\gamma^2}{\kappa_L} \kappa_U^3 \sigma_x^4 \sigma_\varepsilon^2 \frac{p + t}{n_*} \Bigg\}.
\end{align*}

The rest of the proof proceeds conditioned on the event $\mathcal{A}_{1,t} \cap \mathcal{A}_{2,t}$. We first argue that if
\begin{align}
\label{eq:low-dim-cond-n-2}
    \frac{n_*}{p + t} \ge \frac{18C_2 \kappa_U^3\sigma_x^4\sigma_\varepsilon^2}{(\kappa_L^3/\gamma)\mathsf{s}_-} ~~~~\text{and}~~~~ \frac{\bar{n}}{(p + \log(2|\mathcal{E}|) + t)} \ge \frac{\sqrt{18C_2} \kappa_U^{3/2} \sigma_x^3\sigma_\varepsilon}{\sqrt{(\kappa_L^3/\gamma) \mathsf{s}_{-}}},
\end{align} with $\mathsf{s}_{-} = \min_{S\cap G_{\bomega} \neq \emptyset} \bar{\mathsf{d}}_S$, then for any $\bbeta$ with $\supp(\bbeta) \cap G_{\bomega} \neq \emptyset$, the following holds,
\begin{align*}
    \hat{\mathsf{Q}}(\bbeta) - \hat{\mathsf{Q}}(\bbeta^*) \ge \frac{\kappa_L}{4} \|\bbeta - \bbeta^*\|_2^2 + \left(\frac{\gamma}{6} - 2\times \frac{\gamma}{18} \right) \kappa_L^2 \mathsf{s}_- > 0,
\end{align*} which means $\bbeta$ is not the empirical risk minimizer. This implies that $\supp(\hat{\bbeta}_{\mathsf{Q}}) \subseteq G_{\bomega}^c$. 

Meantime, denote $\mathsf{s}_+ = \min_{j\in S^*} |\beta_j^*|^2$. We then argue that if
\begin{align}
\label{eq:low-dim-cond-n-3}
    \frac{n_*}{p + t} \ge \frac{12C_2 \gamma^2 \kappa_U^3\sigma_x^4\sigma_\varepsilon^2}{\kappa_L^2 \mathsf{s}_+} ~~~~\text{and}~~~~ \frac{\bar{n}}{(p + \log(2|\mathcal{E}|) + t)} \ge \frac{\sqrt{12C_2}\gamma \kappa_U^{3/2} \sigma_x^3\sigma_\varepsilon}{  \sqrt{\kappa_L^2 \mathsf{s}_{+}}}
\end{align} then for any $\bbeta$ with $S^* \nsubseteq \supp(\bbeta)$, the following holds
\begin{align*}
    \hat{\mathsf{Q}}(\bbeta) - \hat{\mathsf{Q}}(\bbeta^*) \ge \frac{\kappa_L}{4} \|\bbeta - \bbeta^*\|_2^2 - 2\times \frac{\kappa_L}{12} \min_{j\in S^*} |\beta_j^*|^2 \ge \frac{\kappa_L}{12} \min_{j\in S^*} |\beta_j^*|^2 > 0,
\end{align*} which also means that $\bbeta$ is not the empirical risk minimizer, hence implying the $S^*\subseteq \supp(\hat{\bbeta}_{\mathsf{Q}})$. 

Combining the conditions on $n_*$ and $\bar{n}$ in \eqref{eq:low-dim-cond-n-1}, \eqref{eq:low-dim-cond-n-2} and \eqref{eq:low-dim-cond-n-3} together, we can conclude that if
\begin{align}
\label{eq:low-dim-req-n-bar}
    \frac{\bar{n}}{p+\log(|\mathcal{E}|) + t} \ge \underbrace{ (18C_2 \kappa_U^{3/2} \sigma_x^3 \sigma_\varepsilon) \lor (16C\kappa_U \sigma_x^2\sigma_\varepsilon^2)}_{c_1} \frac{\gamma}{\kappa_L} \left( \sqrt{\frac{1}{\mathsf{s}_+ \land (\gamma\kappa_L \mathsf{s}_{-})}} + \frac{1}{\mathsf{s}_+} \right)
\end{align} and
\begin{align}
\label{eq:low-dim-req-n-star}
    \frac{n_*}{p + t} \ge (\gamma/\kappa_L)^2 \left\{\underbrace{18C_2\kappa_U^3 \sigma_x^4 \sigma_\varepsilon^2}_{c_2} \frac{1}{\mathsf{s} \land (\gamma\kappa_L \mathsf{s}_{-})} + \underbrace{(16C^2)^2\kappa_U^4 \sigma_x^4}_{c_3}\right\}
\end{align}
then under $\mathcal{A}_{1,t}\cap \mathcal{A}_{2,t}$ that occurs with probability at least $1-e^{-7}$ due to \cref{lemma:instance-dependent-bound-j-lowdim}, \cref{lemma:instance-dependent-bound-r-lowdim} and union bound, the following holds
\begin{align}
\label{eq:proof-low-dim-variable-selection}
    S^*\subseteq \supp(\hat{\bbeta}_{\mathsf{Q}}) \subseteq G_{\bomega}^c.
\end{align} This completes the proof. \qed

\subsection{Proof of \cref{thm:rate-faster-low-dim}}
\label{sec:proof-thm3}

\subsubsection{Proof of the Rate \eqref{eq:low-dim-rate2}}
The proof is very similar to that of \cref{thm:rate-low-dim} but will use \cref{lemma:instance-dependent-bound-j-lowdim} and the strong convexity of $\mathsf{Q}$ around $\bbeta^*$ in a different way. The proof proceeds conditioned on $\mathcal{A}_{1,t} \cap \mathcal{A}_{2,t}$. It follows from \cref{lemma:instance-dependent-bound-j-lowdim} and \cref{lemma:instance-dependent-bound-r-lowdim} that the inequality in \eqref{eq:low-dim-tj-tr-bound} holds. We will bound the term $\mathsf{I}_4$ differently.

For first term of $\mathsf{I}_4$, it also follows from the fact $xy \le \frac{1}{2}(x^2 + y^2)$ that
\begin{align*}
&|S^*\setminus \supp(\bbeta)| \times \kappa_U \sigma_x^2\sigma_\varepsilon^2  \frac{t+ \log(2|\mathcal{E}||S^*|)}{\bar{n}} \\
&~~~~~~~~~ = \left(\sqrt{|S^*\setminus \supp(\bbeta)|} \min_{j\in S^*}|\beta_j^*| \sqrt{\kappa_L/16C\gamma}\right) \times \left\{\frac{2\sqrt{|S^*|}\kappa_U \sigma_x^2\sigma_\varepsilon^2}{\sqrt{\kappa_L/16C\gamma} \min_{j\in S^*} |\bbeta_j^ *|}\frac{t+ \log(2|\mathcal{E}||S^*|)}{\bar{n}}\right\} \\
&~~~~~~~~~ \le \frac{\kappa_L}{16C\gamma} |S^*\setminus \supp(\bbeta)| \min_{j\in S^*} |\bbeta_j^*|^2 + \frac{16C \gamma \kappa_U^2 \sigma_x^4\sigma_\varepsilon^2 |S^*|\{t + \log(2|\mathcal{E}||S^*|)\}^2}{\kappa_L (\min_{j\in S^*} |\beta_j^*|^2) \bar{n}^2} \\
&~~~~~~~~~ \le  \frac{\kappa_L}{16C\gamma} \|\bbeta - \bbeta^*\|_2^2 + 16C \kappa_U^2 \sigma_x^4\sigma_\varepsilon^2 \frac{(\gamma/\kappa_L) |S^*|}{\min_{j\in S^*} |\beta_j^*|^2} \left(\frac{t + \log(2|\mathcal{E}||S^*|)}{\bar{n}}\right)^2, 
\end{align*} 

At the same time, we also have $\mathsf{I}_1 \le \frac{\kappa_L}{16} \|\bbeta - \bbeta^*\|_2^2$ if 
\begin{align}
\label{eq:low-dim-faster-rate-c1}
n_* \ge \underbrace{\left(16C \kappa_U \sigma_x^2\right)}_{c_1} (p + t) \left(\gamma/\kappa_L\right)^2.
\end{align}

Putting these pieces together with the error bounds we have for $\mathsf{I}_2$--$\mathsf{I}_3$ in the proof of \cref{thm:rate-low-dim}, the following holds for all the $\bbeta \in \mathbb{R}^p$,
\begin{align*}
\mathsf{T}_{\mathsf{R}}(\bbeta,\bbeta^*) + \mathsf{T}_{\mathsf{J}}(\bbeta, \bbeta^*) &\le 4\times \frac{\kappa_L}{16} \|\bbeta - \bbeta^*\|_2^2 + \frac{\gamma}{6} \mathsf{J}(\bbeta;\omega) \\
&~~~~~~~~~~ + C_1 \frac{\gamma^2}{\kappa_L} \kappa_U^3 \sigma_x^4 \sigma_\varepsilon^2 \left\{\frac{p + t}{n_*} + \sigma_x^2 \left(\frac{p + \log(2|\mathcal{E}|) + t}{\bar{n}}\right)^2 \right\} \\
&~~~~~~~~~~ + C_1 \frac{\gamma^2}{\kappa_L} \kappa_U^2 \sigma_x^4\sigma_\varepsilon^2 \frac{|S^*|}{\min_{j\in S^*} |\beta_j^*|^2} \left(\frac{t + \log(2|\mathcal{E}||S^*|)}{\bar{n}}\right)^2
\end{align*}

Plugging $\hat{\bbeta}_{\mathsf{Q}}$ into the decomposition \eqref{eq:q-decomposition}, it then follows from \cref{prop:population-lb} that
\begin{align*}
    \frac{\kappa_L}{2}\|\hat{\bbeta}_{\mathsf{Q}} - \bbeta^*\|_2^2 + \frac{\gamma}{6} \mathsf{J}(\hat{\bbeta}_{\mathsf{Q}}; \bomega) + 0 &\le \mathsf{Q}(\hat{\bbeta}_{\mathsf{Q}};\gamma, \bomega) - \mathsf{Q}(\hat{\bbeta}_{\mathsf{Q}};\gamma, \bomega) \\
    &= \hat{\mathsf{Q}}(\hat{\bbeta}_{\mathsf{Q}};\gamma, \bomega) - \hat{\mathsf{Q}}(\hat{\bbeta}_{\mathsf{Q}};\gamma, \bomega) + \mathsf{T}_{\mathsf{R}}(\hat{\bbeta}_{\mathsf{Q}},\bbeta^*) + \mathsf{T}_{\mathsf{J}}(\hat{\bbeta}_{\mathsf{Q}}, \bbeta^*) \\
    &\le 0 + \mathsf{T}_{\mathsf{R}}(\hat{\bbeta}_{\mathsf{Q}},\bbeta^*) + \mathsf{T}_{\mathsf{J}}(\hat{\bbeta}_{\mathsf{Q}}, \bbeta^*)
\end{align*} Combining it with the upper bound on $\mathsf{T}_{\mathsf{R}}(\bbeta,\bbeta^*) + \mathsf{T}_{\mathsf{J}}(\bbeta, \bbeta^*)$ we derived, we conclude that
\begin{align}
\label{eq:low-dim-faster-rate-c2-4}
\begin{split}
    \frac{\|\hat{\bbeta}_{\mathsf{Q}} - \bbeta^*\|_2^2}{\sigma_\varepsilon^2} &~\le  \underbrace{4C_2\kappa_U^3 \sigma_x^6}_{c_2^2} \frac{\gamma^2}{\kappa_L^2} \Bigg( \frac{p + t}{n_*} + \frac{(p + \log(2|\mathcal{E}|) + t)^2}{\bar{n}^2}\Bigg)\\
    &~~~~~~~~~~ + \underbrace{4C_2\kappa_U^2 \sigma_x^4}_{c_3^2} \frac{(\gamma^2/\kappa_L^2)|S^*|}{\min_{j\in S^*} |\beta_j^*|^2} \left(\frac{t + \log(2|\mathcal{E}||S^*|)}{\bar{n}}\right)^2
\end{split}
\end{align} under $\mathcal{A}_{1,t}\cap \mathcal{A}_{2,t}$. This completes the proof. \qed

\subsubsection{Proof of the Rate \eqref{eq:low-dim-rate1}}

We assume that conditions \eqref{eq:low-dim-req-n-bar} and \eqref{eq:low-dim-req-n-star} are satisfied. Thus, the variable selection property \eqref{eq:proof-low-dim-variable-selection} also holds. Now that we only focus on the case of no pooled linear spurious variables. The covariate dimension is $p_0 = |(G_{\bomega})^c|$. Now we apply \cref{lemma:instance-dependent-bound-r-lowdim}-\ref{lemma:instance-dependent-bound-j-lowdim} with $\bx_{(G_{\bomega})^c}$. When $\mathcal{A}_{1,t} \cap \mathcal{A}_{2,t} \cap \mathcal{A}_{1,t}^{p_0} \cap \mathcal{A}_{2,t}^{p_0}$ occurs, the inequality 
\eqref{eq:low-dim-tj-tr-bound} with $p=p_0$ also holds for $\bbeta = \hat{\bbeta}_\mathsf{Q}$. Observe that $|S^* \setminus \supp(\bbeta)|=0$ given the variable selection property \eqref{eq:proof-low-dim-variable-selection}. The rest of the proof follows similarly to that for \eqref{eq:low-dim-rate2}.

\subsection{Proof of \cref{thm:high-dim}}

We first define some additional notations. Define $s^*=|S^*|$. Let 
\begin{align*}
    V(s,t) = s\log (ep/s) + s^* + t,
\end{align*} and define $0\times \log(1/0) = 0$.
Similar to the low-dimension counterpart, we also need the following two Lemmas.

\begin{lemma}[Instance-dependent One-side Bound for ${\mathsf{J}}$, High-dimensional]
\label{lemma:instance-dependent-bound-j-highdim}
Assume Conditions \ref{cond0-model} -- \ref{cond3-sub-gaussian-eps} hold. For any $1\le s\le p$, define the following event,
\begin{align}
\label{eq:instance-dependent-bound-j-highdim}
\begin{split}
    \mathcal{A}_{3,t}(s) = \Bigg\{ \forall \bbeta \in \mathcal{B}_s, ~~& \frac{1}{c_1} \left(\mathsf{J}(\bbeta) - \mathsf{J}(\bbeta^*) - \hat{\mathsf{J}}(\bbeta) + \hat{\mathsf{J}}(\bbeta^*)\right) \\
    &\le \|\bbeta - \bbeta^*\|_2^2 \times \kappa_U^2\sigma_x^2 \Bigg(\sqrt{\frac{V(s,t)}{n_{\bomega}}} + \frac{V(s,t)}{n_*}\Bigg) \\
    &~~~~~~ + \|\bbeta - \bbeta^*\|_2 \times \kappa_U^{3/2}\sigma_x^2 \sigma_\varepsilon \Bigg(\sqrt{\frac{V(s,t)}{n_{\bomega}}} + \frac{V(s,t)}{n_*}\Bigg)  \\
    &~~~~~~ + \|\bbeta - \bbeta^*\|_2 \times \kappa_U^{3/2}\sigma_x^3 \sigma_\varepsilon \frac{\sqrt{V(s,t+\log(2|\mathcal{E}|)) \cdot V(0, t+\log(2|\mathcal{E}|))}}{\bar{n}}\\
    &~~~~~~ + \|\bbeta - \bbeta^*\|_2 \times \kappa_U^{3/2}\sigma_x^3 \sigma_\varepsilon \frac{V(s,t+\log(2|\mathcal{E}|))\sqrt{V(0, t+\log(2|\mathcal{E}|))}}{n_\dagger} \\
    &~~~~~~ + \sqrt{\sum_{e\in \mathcal{E}} \omega^{(e)} \left\|\mathbb{E} [\bx_{\supp(\bbeta)}^{(e)} \varepsilon^{(e)}]\right\|_2^2} \times \kappa_U^{1/2} \sigma_x \sigma_\varepsilon \sqrt{\frac{V(|\supp(\bbeta) \setminus S^*|,t)}{n_*}}  \\
    &~~~~~~ + |S^*\setminus \supp(\bbeta)| \times \kappa_U \sigma_x^2\sigma_\varepsilon^2  \frac{t+ \log(2s^*|\mathcal{E}|)}{\bar{n}} \\
    &~~~~~~ + \kappa_U \sigma_x \sigma_\varepsilon^2 \frac{V(|\supp(\bbeta)\setminus S^*|,t)}{n_*}
    \Bigg\}
\end{split}
\end{align} for some universal constant $c_1$. Then we have $\mathbb{P} [\mathcal{A}_{3,t}(s)] \ge 1-6e^{-t}$ for any $t\in (0, n_{\min} - \log(2|\mathcal{E}|) - s^*]$.
\end{lemma}
\begin{proof}[Proof of \cref{lemma:instance-dependent-bound-j-highdim}]
    See \cref{subsec:proof-instance-dependent-bound-j-highdim}.
\end{proof}

\begin{lemma}[Instance-dependent Two-side Bound for ${\mathsf{R}}$, High-dimensional]
\label{lemma:instance-dependent-bound-r-highdim}
Assume Conditions \ref{cond0-model} -- \ref{cond3-sub-gaussian-eps} hold. Define the event
\begin{align}
\begin{split}
    \mathcal{A}_{4,t}(s) = \Bigg\{ \forall \bbeta \in \mathcal{B}_s, ~~ &~ \left|\mathsf{R}(\bbeta) - \mathsf{R}(\bbeta^*) - \hat{\mathsf{R}}(\bbeta) + \hat{\mathsf{R}}(\bbeta^*)\right| \\
    & ~~~~~~~ \le c_1 \left(\kappa_U \sigma_x^2 \delta_1 \|\bbeta - \bbeta^*\|_2^2 + \kappa_U^{1/2} \sigma_x \sigma_\varepsilon \delta_1 \|\bbeta - \bbeta^*\|_2 \right)\Bigg\}.
\end{split} 
\end{align} with $\delta_1 = \sqrt{\frac{V(s,t)}{n_{\bomega}}} + \frac{V(s,t)}{n_*}$ and some universal constants $c_1$. We have $\mathbb{P}[\mathcal{A}_{4,t}(s)] \ge 1-e^{-t}$ for any $t>0$.
\end{lemma}
\begin{proof}[Proof of \cref{lemma:instance-dependent-bound-r-highdim}]
    See \cref{subsec:proof-instance-dependent-bound-r-highdim}.
\end{proof}

We start with a lemma stating that if we choose some large $\lambda$, then the $\hat{S} = \supp(\hat{\bbeta}_{\mathsf{L}})$ will satisfy $|\hat{S}| \le 2s^*$ with high probability.

\begin{proposition}
\label{proposition:small-hat-s}
    Let $0<t\le n_{\min} - \log(2ep|\mathcal{E}|) - s^*$ be arbitrary. Assume Conditions \ref{cond0-model}--\ref{cond-ident} hold, and $\gamma \ge 3\gamma^*\lor 1$. Let $\zeta = \log(ep/s^*) + t$. Suppose further that $\log(|\mathcal{E}|) \le c_1 \log p$ for some universal constant $c_1>0$. There exist some universal constants $c_2$--$c_3$ depending only on $C$ such that if \begin{align*}
        \bar{n} &\ge c_2 \kappa_U\sigma_x^2 \gamma (t + \log p) \left\{ s^* +  \sigma_\varepsilon^2 / (\kappa_L \min_{j\in S^*} |\beta_j^*|^2)\right\}, \\
        n_* &\ge c_2 \kappa_U^4\sigma_x^4 (\gamma/\kappa_L)^2  \zeta s^*, \\
        n_\dagger &\ge c_2 \kappa_U^{3/2} \sigma_x^2 (\gamma/\kappa_L) \zeta s^* \sqrt{\zeta + s^*},
    \end{align*}
    and the choice of $\lambda$ satisfies
    \begin{align*}
        \lambda \ge c_3 \sigma_x^2 \sigma_\varepsilon^2 (\gamma/\kappa_L) \zeta \left\{\frac{\kappa_U^4\sigma_x^2 (\gamma/\kappa_L)}{n_*} + \frac{\kappa_U^{3/2} \sqrt{\zeta + s^*}}{n_\dagger} + \frac{\kappa_{U}^3\sigma_x^2 (\gamma/\kappa_L)(\zeta + s^*)}{\bar{n}^2}\right\}.
    \end{align*} Then the following holds
    \begin{align*}
        \mathbb{P}\left[|\hat{S}| \le 2s^*\right] \ge 1-3e^{-t}.
    \end{align*}
\end{proposition}
\begin{proof}[Proof of \cref{proposition:small-hat-s}] 
    See \cref{subsec:proof-small-hat-s}.
\end{proof}

We are ready to prove \cref{thm:high-dim}.

\begin{proof}[Proof of Theorem~\ref{thm:high-dim}]
We prove the Theorem in a more general form, including the tail probability $t$. Denote $\zeta = \log(ep/s^*) + t$. We first provide a detailed condition for different sample sizes presented in \cref{cond6-n-req},
\begin{align}
\label{eq:req-n-high-dim}
\begin{split}
    n_{\min} &\ge C^* (s^* + \log p) \\
    n_* &\ge C^* \Big\{\kappa_U^4\sigma_x^4 (\gamma/\kappa_L)^2 \zeta s^* \Big\} \lor \Big\{\kappa_U^3 \sigma_x^4 \sigma_\varepsilon^2 \zeta s^* (\gamma/\kappa_L)^2 (1/\kappa_L\beta_{\min}^2)\Big\} \\
    \bar{n} &\ge C^* \left\{\kappa_U\sigma_x^2 \gamma (t + \log p) \left\{ s^* +  \sigma_\varepsilon^2 / (\kappa_L \beta_{\min}^2)\right\} \right\}\lor \left\{\kappa_U^{3/2}\sigma_x^2\sigma_\varepsilon (\gamma/\kappa_L) \frac{\sqrt{s^* \zeta (s^* + \zeta)}}{\kappa_L \beta_{\min}}\right\} \\
    n_\dagger &\ge C^* \kappa_U^{3/2} \sigma_x^2 (\gamma/\kappa_L) s^* \zeta \sqrt{s^* + \zeta} \{1 + \sigma_\varepsilon/(\sqrt{\kappa_L} \beta_{\min})\}
\end{split}
\end{align} together with a detailed lower bound on the regularization parameter $\lambda$
\begin{align}
\label{eq:req-lambda-high-dim}
\begin{split}
    \lambda \ge C^*\kappa_U^3 \sigma_x^4 \sigma_\varepsilon^2 (\gamma/\kappa_L)^2 \Bigg\{ &\left(\frac{\kappa_U \zeta}{n_*} \lor \frac{s^* \zeta}{n_*}\right) + \frac{s^* \zeta(s^* + \zeta)}{\bar{n}^2} \\ 
    &~~~~ + \kappa_U^{-3/2} \sigma_x^{-2} (\gamma/\kappa_L)^{-1} \frac{\zeta\sqrt{s^* + \zeta}}{n_\dagger} + \frac{(s^*\zeta)^2(s^* + \zeta)}{n_\dagger^2} \Bigg\}.
\end{split}
\end{align}
Define the two events
\begin{align*}
    \mathcal{C}_{1,t} = \left\{|\hat{S}| \le 2s^*\right\} ~~~~\text{and}~~~~ \mathcal{C}_{2,t} = \mathcal{A}_{3,t}(2s^*) \cap \mathcal{A}_{4,t}(2s^*).
\end{align*} We can see that the conditions in \cref{lemma:instance-dependent-bound-j-highdim}, \cref{lemma:instance-dependent-bound-r-highdim},  and \cref{proposition:small-hat-s} are satisfied by \eqref{eq:req-n-high-dim} and \eqref{eq:req-lambda-high-dim} with large universal constant $C^*>0$. Denote $\phi_* = s^* + \log(ep/s^*) + t = s^* + \zeta$ and $\phi_s = s^* \log(ep/s^*) + t \le s^* \zeta$. We follow a similar strategy as that in \emph{Case 1} of \cref{proposition:small-hat-s}. Under $\mathcal{C}_{1,t} \cap \mathcal{C}_{2,t}$, it follows our choice of $\gamma$ and \cref{prop:population-lb} that, there exists some universal constant $C_1>0$,
\begin{align*}
    \hat{\mathsf{Q}}(\hat{\bbeta}_{\mathsf{L}}) - \hat{\mathsf{Q}}(\bbeta^*) &= \hat{\mathsf{Q}}(\hat{\bbeta}_{\mathsf{L}}) - \mathsf{Q}(\hat{\bbeta}_{\mathsf{L}}) + \mathsf{Q}(\hat{\bbeta}_{\mathsf{L}}) - \mathsf{Q}(\bbeta^*) + \mathsf{Q}(\bbeta^*) - \hat{\mathsf{Q}}(\bbeta^*) \\
    &\ge \frac{\kappa_L}{2} \|\hat{\bbeta}_{\mathsf{L}} - \bbeta^*\|_2^2 + \frac{\gamma}{6} \mathsf{J}(\hat{\bbeta}_{\mathsf{L}};\bomega) + \hat{\mathsf{R}}(\hat{\bbeta}_{\mathsf{L}}) - \mathsf{R}(\hat{\bbeta}_{\mathsf{L}}) + \mathsf{R} (\bbeta^*) - \hat{\mathsf{R}} (\bbeta^*) \\
    & ~~~~~~~~~~~~ - \gamma \left(\mathsf{J}(\hat{\bbeta}_{\mathsf{L}}) - \hat{\mathsf{J}}(\hat{\bbeta}_{\mathsf{L}}) - \mathsf{J}(\bbeta^*) + \hat{\mathsf{J}}(\bbeta^*)\right)\\
    &\overset{(a)}{\ge} \left(\frac{\kappa_L}{2} - C_1 \kappa_U^2\sigma_x^2 \gamma \sqrt{\frac{\phi_s}{n_*}}\right) \|\hat{\bbeta}_{\mathsf{L}} - \bbeta^*\|_2^2 \\
    & ~~~~~~~~~~~~ + \frac{\gamma}{6} \mathsf{J}(\hat{\bbeta}_{\mathsf{L}};\bomega) - C_1 \gamma \sqrt{\sum_{e\in \mathcal{E}} \omega^{(e)} \left\|\mathbb{E}[\bx_{\hat{S}}^{(e)} \varepsilon^{(e)}]\right\|_2^2} \kappa_U^{1/2} \sigma_x\sigma_\varepsilon \sqrt{\frac{|\hat{S}\setminus S^*|\log (ep/s^*)+s^* + t}{n_*}} \\
    & ~~~~~~~~~~~~ - C_1 \|\hat{\bbeta}_{\mathsf{L}} - \bbeta^*\|_2 \times \gamma \kappa_U^{3/2} \sigma_x^2 \sigma_\varepsilon \sqrt{\frac{\phi_s}{n_*}} \\
    & ~~~~~~~~~~~~ - C_1 \|\hat{\bbeta}_{\mathsf{L}} - \bbeta^*\|_2 \times \gamma \kappa_U^{3/2} \sigma_x^3 \sigma_\varepsilon \left(\frac{\sqrt{\phi_*} \sqrt{\phi_s}}{\bar{n}} + \frac{\phi_s \sqrt{\phi_*}}{n_\dagger} \right) \\
    & ~~~~~~~~~~~~ - C_1 \gamma |S^*\setminus \hat{S}| \kappa_U \sigma_x^2 \sigma_\varepsilon^2 \frac{t+ \log p}{\bar{n}} - C_1 \gamma \kappa_U \sigma_x \sigma_\varepsilon^2 \frac{\phi_s}{n_*} \\
    &= \mathsf{I}_1 + \mathsf{I}_2 + \mathsf{I}_3 + \mathsf{I}_4 + \mathsf{I}_5,
\end{align*} where $(a)$ follows from the fact that $\mathcal{C}_{1,t} \cap \mathcal{C}_{2,t}$ holds such that we can apply the result of \cref{lemma:instance-dependent-bound-j-highdim} and \cref{lemma:instance-dependent-bound-r-highdim} with $s=2s^*$ to $\hat{\bbeta}_{\mathsf{L}}$. Denote $\Diamond=0.5 (6C_1)^2 (\gamma/\kappa_L)^2 \kappa_U^3 \sigma_x^4\sigma_\varepsilon^2$. Following a similar strategy of lower bounds on $\mathsf{I}_1$ -- $\mathsf{I}_5$ and using the fact that $\phi_s \le 2s^*\zeta$ and
\begin{align*}
    \mathsf{I}_2 \ge - \frac{\kappa_L}{12} \|\bbeta - \bbeta^*\|_2^2 - 4\Diamond \frac{|\hat{S} \setminus S^*| \zeta}{n_*},
\end{align*} we obtain
\begin{align*}
    &~\hat{\mathsf{Q}}(\hat{\bbeta}_{\mathsf{L}}) - \hat{\mathsf{Q}}(\bbeta^*) + \lambda \|\hat{\bbeta}_{\mathsf{L}}\|_0 - \lambda \|\bbeta^*\|_0 \\
    &~~~~ = \hat{\mathsf{Q}}(\hat{\bbeta}_{\mathsf{L}}) - \hat{\mathsf{Q}}(\bbeta^*) + \lambda |\hat{S} \setminus S^*| - \lambda |S^* \setminus \hat{S}| \\
    &~~~~\ge \frac{\kappa_L}{12} \|[\hat{\bbeta}_{\mathsf{L}}]_{\hat{S}} - [\bbeta^*]_{\hat{S}}\|_2^2 \\
    &~~~~~~~~~~~~ + \frac{\kappa_L}{12} \|[\hat{\bbeta}_{\mathsf{L}}]_{S^*\setminus \hat{S}} - [\bbeta^*]_{S^*\setminus \hat{S}}\|_2^2 - \lambda |S^* \setminus \hat{S}| \\
    &~~~~~~~~~~~~ + \frac{\lambda}{2} |\hat{S} \setminus S^*| - 4\Diamond \cdot \frac{\zeta}{n_*} |\hat{S}\setminus S^*| \\ 
    &~~~~~~~~~~~~ + \frac{\lambda}{2} |\hat{S}\setminus S^*| - 2 \Diamond \cdot (2s^*) \zeta \left(\frac{1}{n_*} + \frac{\phi_*}{\bar{n}^2} + \frac{2s^* \phi_* \zeta^2}{n_\dagger^2}\right) \\
    &~~~~ \overset{(a)}{\ge} |S^* \setminus \hat{S}| \left(\frac{\kappa_L}{12} \min_{j\in S^*}|\beta^*_j|^2 - \lambda \right) \\
    &~~~~~~~~~~~~ + \frac{\lambda}{2} |\hat{S}\setminus S^*| - 4 \Diamond \cdot s^* \zeta \left(\frac{1}{n_*} + \frac{\phi_*}{\bar{n}^2} + \frac{2s^* \phi_* \zeta}{n_\dagger^2}\right) + \left(\frac{\lambda}{2} - 4\Diamond \cdot \frac{\zeta}{n_*}\right) |\hat{S}\setminus S^*|,
\end{align*} 
Here $(a)$ follows from the facts $\|[\hat{\bbeta}_{\mathsf{L}}]_{\hat{S}} - [\bbeta^*]_{\hat{S}}\|_2^2 \ge 0$, $\|[\hat{\bbeta}_{\mathsf{L}}]_{S^*\setminus \hat{S}} - [\bbeta^*]_{S^*\setminus \hat{S}}\|_2^2 \ge |S^* \setminus \hat{S}| \min_{j\in S^*}|\beta^*_j|$. We argue that if
\begin{align}
\label{eq:proof-thm-highdim-n-req}
    \frac{\kappa_L}{36} \min_{j\in S^*} |\beta_j^*|^2 \ge  4 \Diamond \cdot s^* \zeta \left(\frac{1}{n_*} + \frac{\phi_*}{\bar{n}^2} + \frac{2s^* \phi_* \zeta}{n_\dagger^2}\right),
\end{align} which is satisfied by \cref{cond6-n-req}, and $\lambda \le \frac{\kappa_L}{36} \min_{j\in S^*} |\beta_j^*|^2$, then $|S^*\setminus \hat{S}|=0$. This is because if $|S^*\setminus \hat{S}| \ge 1$, we will have
\begin{align*}
    \hat{\mathsf{Q}}(\hat{\bbeta}_{\mathsf{L}}) - \hat{\mathsf{Q}}(\bbeta^*) + \lambda \|\hat{\bbeta}_{\mathsf{L}}\|_0 - \lambda \|\bbeta^*\|_0 \ge |S^*\setminus \hat{S}| \frac{\kappa_L}{36} \min_{j\in S^*}|\beta^*_j|^2 > 0.
\end{align*} This is contrary to the fact that $\hat{\bbeta}_{\mathsf{L}}$ is the minimizer of \eqref{eq:eills-objective-l}. Moreover, we also have $|\hat{S} \setminus S^*| = 0$ if
\begin{align}
\label{eq:proof-thm-highdim-lambda-req2}
    \lambda \ge 9 \Diamond \cdot s^* \zeta \left(\frac{1}{n_*} + \frac{\phi_*}{\bar{n}^2} + \frac{2s^* \phi_* \zeta}{n_\dagger^2}\right).
\end{align} This is because if $|\hat{S} \setminus S^*| \ge 1$, we will have
\begin{align*}
    \hat{\mathsf{Q}}(\hat{\bbeta}_{\mathsf{L}}) - \hat{\mathsf{Q}}(\bbeta^*) + \lambda \|\hat{\bbeta}_{\mathsf{L}}\|_0 - \lambda \|\bbeta^*\|_0 \ge \frac{\lambda}{2} |\hat{S} \setminus S^*| - 4 \Diamond \cdot s^* \zeta \left(\frac{1}{n_*} + \frac{\phi_*}{\bar{n}^2} + \frac{2s^* \phi_* \zeta}{n_\dagger^2}\right) > 0,
\end{align*} which also contradicts the fact that $\hat{\bbeta}_{\mathsf{L}}$ minimizes \eqref{eq:eills-objective-l}. In conclusion, letting $\lambda \le \frac{\kappa_L}{36} \min_{j\in S^*} |\beta_j^*|^2$ and choosing some large enough constant $C^*$ such that \eqref{eq:proof-thm-highdim-n-req} and \eqref{eq:proof-thm-highdim-lambda-req2} are satisfied by \eqref{eq:req-n-high-dim} and \eqref{eq:req-lambda-high-dim}, we can then argue that under the event $\mathcal{C}_{1,t} \cap \mathcal{C}_{2,t}$, which occurs with probability $1-11e^{-t}$, one has $|\hat{S}\setminus S^*| + |S^*\setminus \hat{S}| = 0$, implying that $\hat{S} = S^*$. This completes the proof via setting $t=10\log (ep)$. 
\end{proof}

\subsection{Proof of \cref{lemma:instance-dependent-bound-j-highdim}}
\label{subsec:proof-instance-dependent-bound-j-highdim}

Let $S=\supp(\bbeta)$, $I = S\cap S^*$. It follows from the fact that $\mathsf{J}(\bbeta^*)=0$ and the definition of $\mathsf{J}$ that
\begin{align*}
	\mathsf{J}(\bbeta) - \mathsf{J}(\bbeta^*) - \hat{\mathsf{J}}(\bbeta) + \hat{\mathsf{J}}(\bbeta^*) = \sum_{e\in \mathcal{E}} \frac{\omega^{(e)}}{4} \left\{\|\nabla_S \mathsf{R}^{(e)}(\bbeta)\|^2_2 - \|\nabla_S \hat{\mathsf{R}}^{(e)}(\bbeta)\|_2^2 + \|\nabla_{S^*} \hat{\mathsf{R}}^{(e)}(\bbeta^*)\|_2^2\right\}.
\end{align*}
Observe that, for any fixed $e\in \mathcal{E}$, 
\begin{align}
&\frac{1}{8}\left(\|\nabla_S \mathsf{R}^{(e)}(\bbeta)\|^2_2 - \|\nabla_S \hat{\mathsf{R}}^{(e)}(\bbeta)\|_2^2 + \|\nabla_{S^*} \hat{\mathsf{R}}^{(e)}(\bbeta^*)\|_2^2\right) \nonumber \\
=&~ \frac{1}{8}\left\{\|\nabla_S \mathsf{R}^{(e)}(\bbeta)\|^2_2 - \|\nabla_S \hat{\mathsf{R}}^{(e)}(\bbeta) - \nabla_S \mathsf{R}^{(e)}(\bbeta) + \nabla_S \mathsf{R}^{(e)}(\bbeta)\|_2^2 + \|\nabla_{S^*} \hat{\mathsf{R}}^{(e)}(\bbeta^*)\|_2^2\right\} \nonumber\\
=&~ -\frac{1}{4} \{\nabla_S \mathsf{R}^{(e)}(\bbeta)\}^\top \left\{\nabla_S \hat{\mathsf{R}}^{(e)}(\bbeta) - \nabla_S {\mathsf{R}}^{(e)}(\bbeta)\right\} + \frac{1}{8} \|\nabla_{S^*\setminus S} \hat{\mathsf{R}}^{(e)}(\bbeta^*)\|_2^2 \nonumber\\
&~~~~~~~~~~~ - \frac{1}{8} \|\nabla_{S\setminus S^*} \hat{\mathsf{R}}^{(e)}(\bbeta) - \nabla_{S\setminus S^*} {\mathsf{R}}^{(e)}(\bbeta)\|_2^2 \nonumber\\
&~~~~~~~~~~~ + \frac{1}{8}\left\{\|\nabla_{I} \hat{\mathsf{R}}^{(e)}(\bbeta^*)\|_2^2 - \|\nabla_{I} \hat{\mathsf{R}}^{(e)}(\bbeta) - \nabla_{I} \mathsf{R}^{(e)}(\bbeta)\|_2^2\right\} \nonumber\\
\overset{(a)}{\le}&  \left(\bSigma^{(e)}_{S,:}(\bbeta - \bbeta^*) - \mathbb{E}[\bx_S^{(e)} \varepsilon^{(e)}]\right)^\top \Bigg\{\hat{\mathbb{E}}[\bx_S^{(e)} \varepsilon^{(e)}] - \mathbb{E}[\bx_S^{(e)} \varepsilon^{(e)}]  \nonumber\\
&~~~~~~~~~~~~~~~~ - \left\{\hat{\mathbb{E}}[\bx_S^{(e)} (\bx^{(e)})^\top (\bbeta - \bbeta^*)] - \bSigma^{(e)}_{S,:} (\bbeta- \bbeta^*)\right\}\Bigg\} \nonumber\\
&~~~~~~~~~~~~~~~~ + \left(\hat{\mathbb{E}}[\bx_{I}^{(e)} \varepsilon^{(e)}] \right)^\top \left(\hat{\mathbb{E}} [\bx_{I}^{(e)} (\bx^{(e)})^\top (\bbeta - \bbeta^*)] - \bSigma^{(e)}_{I,:} (\bbeta - \bbeta^*)\right) + \frac{1}{2} \left\|\hat{\mathbb{E}}[\bx_{S^*\setminus S}^{(e)} \varepsilon^{(e)}]\right\|_2^2 \nonumber\\
=&~  (\bbeta - \bbeta^*)^\top \bSigma^{(e)}_{:,S} \left\{\hat{\mathbb{E}}[\bx_S^{(e)} \varepsilon^{(e)}] - \mathbb{E}[\bx_S^{(e)} \varepsilon^{(e)}]\right\} \nonumber\\
&~~~~~~~~~ -\mathbb{E}[\bx^{(e)}_S \varepsilon^{(e)}]^\top \left\{\hat{\mathbb{E}}[\bx_S^{(e)} \varepsilon^{(e)}] - \mathbb{E}[\bx_S^{(e)} \varepsilon^{(e)}]\right\} \nonumber\\
&~~~~~~~~~ +(\bbeta - \bbeta^*)^\top \bSigma_{:,S}^{(e)} \left\{\hat{\mathbb{E}}[\bx_S^{(e)} (\bx^{(e)})^\top (\bbeta - \bbeta^*)] - \bSigma^{(e)}_{S,:} (\bbeta- \bbeta^*)\right\} \nonumber\\
&~~~~~~~~~ -\mathbb{E}[\bx_S^{(e)}\varepsilon^{(e)}]^\top \left\{\hat{\mathbb{E}}[\bx_S^{(e)} (\bx^{(e)})^\top (\bbeta - \bbeta^*)] - \bSigma^{(e)}_{S,:} (\bbeta- \bbeta^*)\right\} \nonumber\\
&~~~~~~~~~ +\left(\hat{\mathbb{E}}[\bx_{I}^{(e)} \varepsilon^{(e)}] \right)^\top \left(\hat{\mathbb{E}} [\bx_{I}^{(e)} (\bx^{(e)})^\top (\bbeta - \bbeta^*)] - \bSigma^{(e)}_{I,:} (\bbeta - \bbeta^*)\right) + \frac{1}{2} \left\|\hat{\mathbb{E}}[\bx_{S^*\setminus S}^{(e)} \varepsilon^{(e)}]\right\|_2^2\nonumber\\
=&~ \mathsf{T}^{(e)}_1(\bbeta) + \mathsf{T}^{(e)}_2(\bbeta) + \mathsf{T}^{(e)}_3(\bbeta) + \mathsf{T}^{(e)}_4(\bbeta) + \mathsf{T}^{(e)}_5(\bbeta) + \mathsf{T}_6^{(e)}(\bbeta), \label{eq:low-dim-id-decomposition}
\end{align}
where $(a)$ follows from the definition of $\nabla \mathsf{R}(\bbeta)$, $\nabla \hat{\mathsf{R}}(\bbeta)$ together with the facts
\begin{align*}
	&\frac{1}{8}\left\{\|\nabla_{I} \hat{\mathsf{R}}^{(e)}(\bbeta^*)\|_2^2 - \|\nabla_{I} \hat{\mathsf{R}}^{(e)}(\bbeta) - \nabla_{I} {\mathsf{R}}^{(e)}(\bbeta)\|_2^2\right\} \\
	=& ~\frac{1}{2} \sum_{j\in I} \left(\hat{\mathbb{E}}[x_j^{(e)} \varepsilon^{(e)}] \right)^2 - \left(- \hat{\mathbb{E}}[x_j^{(e)} \varepsilon^{(e)}] + \hat{\mathbb{E}}[x_j^{(e)} (\bx^{(e)})^\top (\bbeta - \bbeta^*)] - {\mathbb{E}}[x_j^{(e)} (\bx^{(e)})^\top (\bbeta - \bbeta^*)]\right)^2\\
	=& ~\frac{1}{2} \sum_{j \in I} \left(\hat{\mathbb{E}}[x_j^{(e)} (\bx^{(e)})^\top (\bbeta - \bbeta^*)] - {\mathbb{E}}[x_j^{(e)} (\bx^{(e)})^\top (\bbeta - \bbeta^*)]\right) \times \\
	&~~~~~~~~~~~~~~~~~~~~~ \left\{2\hat{\mathbb{E}}[x_j^{(e)} \varepsilon^{(e)}] -\left(\hat{\mathbb{E}}[x_j^{(e)} (\bx^{(e)})^\top (\bbeta - \bbeta^*)] - {\mathbb{E}}[x_j^{(e)} (\bx^{(e)})^\top (\bbeta - \bbeta^*)]\right)\right\} \\
	\le& ~ \sum_{j\in I} \hat{\mathbb{E}}[x_j^{(e)} \varepsilon^{(e)}] \times \left(\hat{\mathbb{E}}[x_j^{(e)} (\bx^{(e)})^\top (\bbeta - \bbeta^*)] - {\mathbb{E}}[x_j^{(e)} (\bx^{(e)})^\top (\bbeta - \bbeta^*)]\right) \\
	=& ~\left(\hat{\mathbb{E}}[\bx_{I}^{(e)} \varepsilon^{(e)}] \right)^\top \left(\hat{\mathbb{E}} [\bx_{I}^{(e)} (\bx^{(e)})^\top (\bbeta - \bbeta^*)] - \bSigma^{(e)}_{I,:} (\bbeta - \bbeta^*)\right),
\end{align*} and $-\frac{1}{8} \|\nabla_{S\setminus S^*} \hat{\mathsf{R}}^{(e)}(\bbeta) - \nabla_{S\setminus S^*} {\mathsf{R}}^{(e)}(\bbeta)\|_2^2 \le 0$.

The rest of the proof will be divided into several pieces deriving the instance-dependent high-probability upper bounds on $\sum_{e\in \mathcal{E}} \omega^{(e)} \mathsf{T}_k^{(e)}(\bbeta)$ for each $k\in [6]$.

\noindent {\sc Step 1. Upper bound on $\mathsf{T}_1^{(e)}$.} Define the event
\begin{align}
\label{eq:low-dim-id-bound1}
	\mathcal{C}_{1,t} = \left\{\forall \bbeta \in \mathcal{B}_s, ~~\sum_{e\in \mathcal{E}} \omega^{(e)} \mathsf{T}_1^{(e)} \le C_1 \kappa_U^{3/2} \sigma_x \sigma_\varepsilon \|\bbeta - \bbeta^*\|_2 \Bigg(\sqrt{\frac{V(s,t)}{n_{\bomega}}} + \frac{V(s,t)}{n_*}\Bigg)\right\}
\end{align} for some constant $C_1$ to be determined. We will claim that $\mathbb{P}(\mathcal{C}_{1,t}) \ge 1-e^{-t}$ in this step. We further assume $\bbeta \neq \bbeta^*$ since the inequality holds trivially when $\bbeta = \bbeta^*$. It then follows from Cauchy-Schwarz inequality that
\begin{align*}
	\sup_{\bbeta \in \mathcal{B}_s, \bbeta \neq \bbeta^*} \frac{\sum_{e\in \mathcal{E}} \omega^{(e)} \mathsf{T}_1^{(e)}}{\|\bbeta - \bbeta^*\|_2} &= \sup_{\bbeta \in \mathcal{B}_s, \bbeta \neq \bbeta^*} \frac{(\bbeta - \bbeta^*)^\top}{\|\bbeta - \bbeta^*\|_2} \sum_{e\in \mathcal{E}} \omega^{(e)} \bSigma^{(e)}_{S\cup S^*,S} \left\{\hat{\mathbb{E}}[\bx_S^{(e)} \varepsilon^{(e)}] - \mathbb{E}[\bx^{(e)} \varepsilon^{(e)}]\right\} \\
	&\le \sup_{|S| \le s} \left\|\sum_{e\in \mathcal{E}} \omega^{(e)} \bSigma^{(e)}_{S\cup S^*,S} \left\{\hat{\mathbb{E}}[\bx_S^{(e)} \varepsilon^{(e)}] - \mathbb{E}[\bx^{(e)} \varepsilon^{(e)}]\right\}\right\|_2.
\end{align*}

For any $S\subseteq [p]$ with $|S|\le s$, let $\bv^{(S)}_1, \ldots, \bv^{(S)}_{N_S}$ be an $1/4-$covering of $\mathcal{B}(S \cup S^*)$, that is, for any $\bv\in \mathcal{B}(S\cup S^*)$, there exists some $\pi(\bv) \in [N_S]$ such that
\begin{align}
\label{eq:proof-cover1}
	\|\bv - \bv_{\pi(v)}^{(S)}\|_2 \le 1/4.
\end{align}
	It follows from standard empirical process result that $N_S \le 9^{|S\cup S^*|}$, then 
\begin{align}
\label{eq:calculate-cover-number}
\begin{split}
	N = \sum_{|S| \le s} N_S \le \sum_{|S|\le s} 9^{|S\cup S^*|} &\le \sum_{i=0}^s 9^{i+s^*} \binom{p}{i} \\
            &\le 9^{s^*} \times  \left(\frac{9p}{s}\right)^{s} \sum_{i=0}^s \left(\frac{s}{p}\right)^i \binom{p}{i} \le 9^{s^*} \times  \left(\frac{9p}{s}\right)^{s} \sum_{i=0}^p \left(\frac{s}{p}\right)^i \binom{p}{i} \\
            & \le 9^{s^*} \left(\frac{9p}{s}\right)^{s} \left(1 + \frac{s}{p}\right)^p \le 9^{s^*} \left(\frac{9ep}{s}\right)^{s}.
\end{split}
\end{align} 

At the same time, denote $\bxi = \sum_{e\in \mathcal{E}} \omega^{(e)} \bSigma^{(e)}_{S\cup S^*,S} \left\{\hat{\mathbb{E}}[\bx_S^{(e)} \varepsilon^{(e)}] - \mathbb{E}[\bx^{(e)} \varepsilon^{(e)}]\right\}$. For any $S\in [p]$ with $|S| \le s$, it follows from the variational representation of the $\ell_2$ norm that
\begin{align*}
\|\bxi\|_2 &= \sup_{\bv \in \mathcal{B}(S \cup S^*)} \bv^\top \bxi = \sup_{k \in [N_S]} (\bv_k^{(S)})^\top \bxi + \sup_{\bv \in \mathcal{B}(S\cup S^*)} (\bv - \bv_{\pi(v)}^{(S)})^\top \bxi \le \sup_{k \in [N_S]} (\bv_k^{(S)})^\top \bxi + \frac{1}{4}\|\bxi\|_2,
\end{align*} where the last inequality follows from the Cauchy-Schwarz inequality and our construction of covering in \eqref{eq:proof-cover1}. This implies $\|\bxi\|_2 \le 2 \sup_{k \in [N_S]} (\bv_k^{(S)})^\top \bxi$, thus
\begin{align}
\label{eq:proof-sup1}
\begin{split}
	&\sup_{|S| \le s} \left\|\sum_{e\in \mathcal{E}} \omega^{(e)} \bSigma^{(e)}_{S\cup S^*,S} \left\{\hat{\mathbb{E}}[\bx_S^{(e)} \varepsilon^{(e)}] - \mathbb{E}[\bx^{(e)} \varepsilon^{(e)}]\right\}\right\|_2  \\
	&~~~~~~ \le 2 \sup_{|S|\le s, k\in [N_S]} \underbrace{(\bv_{k}^{(S)})^\top \sum_{e\in \mathcal{E}} \omega^{(e)} \bSigma^{(e)}_{S\cup S^*,S} \left\{\hat{\mathbb{E}}[\bx_S^{(e)} \varepsilon^{(e)}] - \mathbb{E}[\bx^{(e)} \varepsilon^{(e)}]\right\}}_{Z(S,k)}. 
\end{split}
\end{align} 
For given fixed $\bv_k^{(S)} \in \mathcal{B}(S\cup S^*)$, $Z(S,k)$ can be written as the sum of independent zero-mean random variables as
\begin{align*}
	\sum_{e\in \mathcal{E}} \sum_{i=1}^{n^{(e)}} \frac{\omega^{(e)}}{n^{(e)}}\left(X_{e,i} - \mathbb{E}[X_{e,i}] \right) ~~~~ with ~~~~ X_{e,i} = \left((\bv^{(S)}_k)^\top \bSigma^{(e)}_{S\cup S^*,S} \bx_S^{(e)}\right) \left( \varepsilon^{(e)} \right).
\end{align*} Observe that $\varepsilon^{(e)}$ is a zero-mean sub-Gaussian random variable with parameter $\sigma_\varepsilon$ by \cref{cond3-sub-gaussian-eps}, and $(\bv^{(S)}_k)^\top \bSigma^{(e)}_{S\cup S^*, S}\bx_{S}$ is a zero-mean sub-Gaussian random variable with parameter 
\begin{align*}
	\sigma_1 = \|(\bv_{k}^{(S)})^\top \bSigma^{(e)}_{S\cup S^*, S} \bar{\bSigma}_{S,:}^{1/2} \|_2 \sigma_x \le \kappa_U^{3/2} \sigma_x
\end{align*} by \cref{cond2-subgaussain-x}. It then follows from \cref{lemma:product-sub-gaussian} and \cref{lemma:sum-sub-exp} that there exists some universal constant $C'$ such that
\begin{align*}
    \mathbb{P}\left[ |Z(S,k)| \ge C' \sigma_1 \sigma_\varepsilon \left\{\sqrt{\sum_{e\in \mathcal{E}} \sum_{i=1}^{n^{(e)}}  \left(\frac{\omega^{(e)}}{n^{(e)}}\right)^2 u} + \max_{e\in \mathcal{E}} \frac{\omega^{(e)}}{n^{(e)}} u\right\}\right] \le 2e^{-u}
\end{align*} for any $u>0$. Letting $u=t+\log(2N) \le 3 \left(t + s\log(ep/s) + s^*\right)$, we obtain
\begin{align*}
    \mathbb{P}\left[\sup_{|S|\le s, k} |Z(S,k)| \ge 3C' \sigma_1\sigma_\varepsilon \Bigg(\sqrt{\frac{V(s,t)}{n_{\bomega}}} + \frac{V(s,t)}{n_*}\Bigg)\right] \le  N\times 2e^{-\log(2N)-t} \le e^{-t}.
\end{align*} Combining with the argument \eqref{eq:proof-sup1} concludes the proof of the claim with $C_1=6C'$.

\noindent {\sc Step 2. Upper Bound on $\mathsf{T}_2^{(e)}$. } We claim that $\mathbb{P}(\mathcal{C}_{2,t}) \ge 1-e^{-t}$ for any $t>0$, where
\begin{align}
\label{eq:low-dim-id-bound2}
\begin{split}
\mathcal{C}_{2,t} = \Bigg\{\forall \bbeta, ~~~~ \sum_{e\in \mathcal{E}} \omega^{(e)} \mathsf{T}_2^{(e)}  \le &~C_2 \kappa_U^{1/2} \sigma_x \sigma_\varepsilon \sqrt{\frac{V(|S\setminus S^*|,t)}{n_*}} \times \sqrt{\sum_{e\in \mathcal{E}} \omega^{(e)} \left\|\mathbb{E} [\bx_S^{(e)} \varepsilon^{(e)}]\right\|_2^2} \\
	&~~~~ + C_2 \kappa_U \sigma_x \sigma_\varepsilon^2 \frac{V(|S\setminus S^*|,t)}{n_*} \Bigg\}
\end{split}
\end{align} for some universal constant $C_2$ to be determined. Note that L.H.S. and R.H.S. of the inequality in \eqref{eq:low-dim-id-bound2} both depend on $S$, which is the support set of $\bbeta$ and satisfies $|S|\le s$. For a fixed $S$, we can write down $\sum_{e\in \mathcal{E}} \omega^{(e)} \mathsf{T}_2^{(e)}$ as sum of independent random variables as
\begin{align*}
	Z(S)=\sum_{e\in \mathcal{E}} \omega^{(e)} \mathsf{T}_2^{(e)} = \sum_{e\in \mathcal{E}} \sum_{i=1}^{n^{(e)}} \frac{\omega^{(e)}}{n^{(e)}} \left(X_{e,i} - \mathbb{E}[X_{e,i}]\right) ~~~~\text{with}~~~~ X_{e,i} = \left(\mathbb{E}[\bx_S^{(e)} \varepsilon^{(e)}]^\top [\bx_{i}^{(e)}]_S \right) (\varepsilon^{(e)}_{i}).
\end{align*}
Observe that $X_{i,e}$ are independent sub-exponential random variables because of \cref{lemma:product-sub-gaussian} and our assumptions \cref{cond2-subgaussain-x}--\ref{cond3-sub-gaussian-eps}. It then follows from \cref{lemma:sum-sub-exp} that the following event,
\begin{align*}
	|Z(S)| &\le C'\kappa_U \sigma_x \sigma_\varepsilon \left\{\sqrt{\sum_{e\in \mathcal{E}} (\omega^{(e)})^2 \frac{1}{n^{(e)}} \left\|\mathbb{E}[\bx_S^{(e)} \varepsilon^{(e)}]\right\|_2^2} \times \sqrt{u} + \max_{e\in \mathcal{E}} \frac{\omega^{(e)}}{n^{(e)}} \left\|\mathbb{E}[\bx_S^{(e)}\varepsilon^{(e)}]\right\|_2^2 \times u\right\}\\
	&\le C'\kappa_U^{1/2} \sigma_x \sigma_\varepsilon \left\{ \sqrt{u\times \max_{e'\in \mathcal{E}} \frac{\omega^{(e')}}{n^{(e')}}} \sqrt{\sum_{e\in \mathcal{E}} \omega^{(e)} \left\|\mathbb{E}[\bx_S^{(e)} \varepsilon^{(e)}]\right\|_2^2} + \frac{u}{n_*} \max_{e\in \mathcal{E}}  \left\|\mathbb{E}[\bx^{(e)}\varepsilon^{(e)}]\right\|_2^2 \right\} \\
	&\le C'\kappa_U^{1/2} \sigma_x \sigma_\varepsilon \left\{ \sqrt{\frac{x}{n_*}} \sqrt{\sum_{e\in \mathcal{E}} \omega^{(e)} \left\|\mathbb{E}[\bx_S^{(e)} \varepsilon^{(e)}]\right\|_2^2}  + \kappa_U^{1/2} \sigma_\varepsilon \frac{x}{n_*} \right\}
\end{align*} occurs with probability at least $1-2e^{-x}$ for any $u>0$, where the last inequality follows from \cref{lemma:x-product-epsilon} and the definition of $n_*$ in \eqref{eq:def-n-low-dim}. Now, define the event 
\begin{align*}
\mathcal{K}_u(r) = \left\{ \forall S, |S\setminus S^*|=r, ~~~ |Z(S)| \le 2C'\kappa_U^{1/2} \sigma_x \sigma_\varepsilon \left\{ \sqrt{\frac{V(r,u)}{n_*}} \sqrt{\sum_{e\in \mathcal{E}} \omega^{(e)} \left\|\mathbb{E}[\bx_S^{(e)} \varepsilon^{(e)}]\right\|_2^2} + \kappa_U^{1/2} \sigma_\varepsilon \frac{V(r,u)}{n_*} \right\}\right\}
\end{align*} for any $r\ge 1$. The total number of $S$ satisfying $|S\setminus S^*|=r$ can be upper-bounded by
\begin{align*}
    N_r = 2^{s^*} \times \binom{p-s^*}{r} \le 2^{s^*} (ep/r)^r.
\end{align*} Applying union bound with $x=u+\log(2N_r) \le 2(u + r\log(ep/r) + s^*)$ then gives $\mathbb{P}[\mathcal{K}_u(r)] \ge 1-2N_r e^{-(u+\log(2N_r))} = 1-e^{-u}$. Therefore, we can argue that, under $\bigcap_{r=1}^p \mathcal{K}_{t+\log p}(r)$, the following holds
\begin{align*}
\forall \bbeta, ~~~ Z(S) \le 4C'\kappa_U^{1/2} \sigma_x \sigma_\varepsilon \left\{ \sqrt{\frac{V(|S\setminus S^*|,t)}{n_*}} \sqrt{\sum_{e\in \mathcal{E}} \omega^{(e)} \left\|\mathbb{E}[\bx_S^{(e)} \varepsilon^{(e)}]\right\|_2^2} + \kappa_U^{1/2} \sigma_\varepsilon \frac{V(|S\setminus S^*|,t)}{n_*} \right\}
\end{align*} by the fact that
\begin{align*}
    \forall r \ge 1, ~~~~~ V(r, t + \log p) = r\log(ep/r) + s^* + (\log p) + t \le 2r\log(ep/r) + s^* + t \le 2V(r, t).
\end{align*} This completes the proof of via setting $C_2 = 4C'$.

\noindent {\sc Step 3. Upper Bound on $\mathsf{T}_3^{(e)}$.} We argue in this step that the event
\begin{align}
\label{eq:low-dim-id-bound3}
	\mathcal{C}_{3,t} = \left\{ \forall \bbeta \in \mathcal{B}_s, ~~~~ \sum_{e\in \mathcal{E}} \omega^{(e)} \mathsf{T}_3^{(e)} \le C_3\kappa_U^2 \sigma_x^2  \|\bbeta - \bbeta^*\|_2^2 \Bigg(\sqrt{\frac{V(s,t)}{n_{\bomega}}} + \frac{V(s,t)}{n_*}\Bigg) \right\}	
\end{align} occurs with probability at least $1-e^{-t}$ for any $t>0$, where $C_3$ is some universal constant to be determined. Without loss of generality, let $\bbeta \neq \bbeta^*$, then it suffices to establish an upper bound for
\begin{align}
\label{eq:step3-sup-argument-1}
\begin{split}
&~\sup_{\bbeta \neq \bbeta^*} \frac{(\bbeta - \bbeta^*)^\top}{\|\bbeta - \bbeta^*\|_2} \sum_{e\in \mathcal{E}} \omega^{(e)} \bSigma_{:,S}^{(e)} \left(\hat{\mathbb{E}}[\bx_S^{(e)} (\bx^{(e)})^\top] - \mathbb{E}[\bx_S^{(e)} (\bx^{(e)})^\top] \right) \frac{(\bbeta - \bbeta^*)}{\|\bbeta - \bbeta^*\|_2^2} \\
&~~~\le \sup_{|S| \le s} \left\| \sum_{e\in \mathcal{E}} \omega^{(e)} \bSigma_{S\cup S^*,S}^{(e)} \left(\hat{\mathbb{E}}[\bx_S^{(e)} (\bx^{(e)}_{S\cup S^*})^\top] - \mathbb{E}[\bx_S^{(e)} (\bx^{(e)}_{S\cup S^*})^\top] \right) \right\|_2 = \sup_{|S| \le s} \|\bA_S\|_2.
\end{split}
\end{align}
We follow a similar strategy as {\sc Step 1}. For any $S\in \mathcal{B}_s$, let $\{(\bv^{(e)}_{k}, \bu^{(e)}_{k})\}_{k=1}^{N_S} \in \mathcal{B}(S\cup S^*) \times \mathcal{B}(S\cup S^*) := \mathcal{B}^2(S\cup S^*)$ be a $1/4$-covering of $\mathcal{B}^2(S\cup S^*)$ in a sense that for any $(\bu, \bv) \in \mathcal{B}^2(S\cup S^*)$, there exists some $\pi(\bu, \bv) \in [N_S]$ such that
\begin{align*}
    \|\bu - \bu^{(S)}_{\pi(\bu, \bv)}\|_{2} + \|\bv - \bv^{(S)}_{\pi(\bu, \bv)}\|_{2} \le \frac{1}{4}.
\end{align*} It follows from standard empirical process theory that $N_S \le 9^{2|S\cup S^*|}$, then
\begin{align*}
    N = \sum_{S\subseteq [p], |S| \le s} N_S &\le \sum_{S\subseteq [p], |S| \le s} N_S 9^{2|S\cup S^*|} \le \sum_{i=0}^{s} 81^{i + s^*} \binom{p}{i}
    \le 81^{s^*} \left(\frac{81ep}{s}\right)^s,
\end{align*} where the last inequality follows from the same procedure as \eqref{eq:calculate-cover-number}.

At the same time, denote $\bu^\dagger = \bu^{(S)}_{\pi(\bu,\bv)}$ and $\bv^\dagger = \bv^{(S)}_{\pi(\bu,\bv)}$. It follows from the variational representation of the matrix $\ell_2$ norm that
\begin{align*}
    \|\bA_S\|_2 &= \sup_{(\bu, \bv) \in \mathcal{B}^2(S\cup S^*)} \bu^\top \bA_S \bv \\
    &\le \sup_{(\bu, \bv) \in \mathcal{B}^2(S\cup S^*)} (\bu^\dagger)^\top \bA_S \bv^\dagger + \sup_{(\bu, \bv) \in \mathcal{B}^2(S\cup S^*)} (\bu-\bu^\dagger)^\top \bA_S \bv^\dagger \\
    &~~~~~~~~ + \sup_{(\bu, \bv) \in \mathcal{B}^2(S\cup S^*)} (\bu-\bu^\dagger)^\top \bA_S (\bv - \bv^\dagger) + \sup_{(\bu, \bv) \in \mathcal{B}^2(S\cup S^*)} (\bu^\dagger)^\top \bA_S (\bv - \bv^\dagger) \\
    &\le  \sup_{k \in [N_S]} (\bu^{(S)}_k)^\top \bA_S \bv^{(S)}_k + \left(\frac{1}{4} + \frac{1}{4} + \frac{1}{16} \right) \|\bA_S\|_2, 
\end{align*} which implies $\|\bA_S\|_2 \le 4 \sup_{k \in [N_S]} (\bu^{(S)}_k)^\top \bA_S \bv^{(S)}_k$, thus
\begin{align}
\label{eq:step3-sup-argument-2}
    \sup_{|S| \le s} \|\bA_S\|_2 \le \sup_{|S| \le s, k \in [N_S]} (\bu_{k}^{(e)})^\top \bA_S (\bv_k^{(e)}) = \sup_{|S| \le s, k\in [N_S]} Z(S,k).
\end{align} For fixed $k$ and $S$, $Z(S,k)$ can be written as the sum of independent zero-mean random variables as
\begin{align*}
    Z(S,k) = \sum_{e\in \mathcal{E}} \sum_{i=1}^{n^{(e)}} \frac{\omega^{(e)}}{n^{(e)}} (X_{e,i} - \mathbb{E} [X_{e,i}]) \qquad \text{with} \qquad X_{e,i} = \left((\bu_{k}^{(e)})^\top \Sigma^{(e)}_{S\cup S^*, S} \bx_S^{(e)} \right) \left((\bx_{S\cup S^*}^{(e)})^\top \bv_k^{(e)}\right).
\end{align*} Here $X_{e,i}$ is the product of two zero-mean sub-Gaussian random variables with parameter $\kappa_U^{3/2}\sigma_x$ and $\kappa_U^{1/2}\sigma_x$, respectively. It then follows from \cref{lemma:product-sub-gaussian} and \cref{lemma:sum-sub-exp} that, for any $u>0$,
\begin{align*}
    \mathbb{P}\left[|Z(S,k)| \le C'\kappa_U^2 \sigma_x^2 \left(\sqrt{\frac{u}{n_{\bomega}}} + \frac{u}{n_*}\right)\right] \ge 1-2e^{-u},
\end{align*} where $C'$ is some universal constant. Finally, we apply union bound with $u=t+\log(2N) \le 6(t + s\log(ep/s) + s^*)$ and obtain
\begin{align*}
    \mathbb{P}\left[\sup_{|S| \le s, k\in [N_S]} |Z(S,k)| \le 6C'\kappa_U^2 \sigma_x^2 \Bigg(\sqrt{\frac{V(s,t)}{n_{\bomega}}} + \frac{V(s,t)}{n_*}\Bigg)\right] \ge 1-N\times 2e^{-\log(2N) + t} = 1-e^{-t}.
\end{align*} Set $C_3 = 24C'$. Combining the above inequality with the suprema arguments \eqref{eq:step3-sup-argument-2} and \eqref{eq:step3-sup-argument-1} completes the proof.

\noindent {\sc Step 4. Upper Bound on $\mathsf{T}_4^{(e)}$.} Our target in this step is to show that, for any $t>0$,
\begin{align}
\label{eq:low-dim-id-bound4}
\mathbb{P}(\mathcal{C}_{4,t}) = \mathbb{P}\left[\forall \bbeta \in \mathcal{B}_s, ~~ \sum_{e\in \mathcal{E}} \omega^{(e)} \mathsf{T}_4^{(e)} \le C_4 \kappa_U^{3/2} \sigma_x^2\sigma_\varepsilon \|\bbeta - \bbeta^*\|_2\Bigg(\sqrt{\frac{V(s,t)}{n_{\bomega}}} + \frac{V(s,t)}{n_*}\Bigg)\right] \ge 1-e^{-t}.
\end{align} The proof is very similar to that in {\sc Step 1}, we only sketch here and highlight the difference. Following a similar strategy, it suffices to derive an upper bound for the quantity
\begin{align*}
    \sup_{|S| \le s, k\in [N_S]} Z(S,k) ~~~~~ \text{with} ~~~~~ Z(S,k) = \sum_{e\in \mathcal{E}} \sum_{i=1}^{n^{(e)}} \frac{\omega^{(e)}}{n^{(e)}} (X_{e,i} - \mathbb{E}[X_{e,i}]) ~~\text{and}~~ N=\sum_{S\subseteq [p]} N_S \le 9^{s^*}(9ep/s)^s,
\end{align*} where $X_{e,i}=\left((\mathbb{E}[\bx_S^{(e)} \varepsilon^{(e)}])^\top [\bx^{(e)}_i]_S\right) \left([\bx_i^{(e)}]_{S\cup S^*}^\top \bv_k^{(S)}\right)$ is the product of two zero-mean sub-Gaussian random variables with parameters $\kappa_U\sigma_x\sigma_\varepsilon$ and $\kappa_{U}^{1/2}\sigma_x$, respectively. It then follows from \cref{lemma:sum-sub-exp} that, for any $u>0$,
\begin{align*}
    \mathbb{P}\left[|Z(S,k)| \le C' \kappa_U^{3/2} \sigma_x^2\sigma_\varepsilon \left(\sqrt{\frac{u}{n_{\bomega}}} + \frac{u}{n_*}\right)\right] \ge 1-2e^{-u}.
\end{align*} So it concludes via applying union bound with $u = t + \log (2N) \le 3V(s,t)$.

\noindent {\sc Step 5. Upper Bound on $\mathsf{T}_5^{(e)}$.} In this step, we claim that the following event
\begin{align}
\label{eq:low-dim-id-bound5}
    \mathcal{C}_{5,t} = \left\{\forall \bbeta \in \mathbb{R}^p, ~~ \sum_{e\in \mathcal{E}} \omega^{(e)} \mathsf{T}_5^{(e)} \le C_5\kappa_U^{3/2} \sigma_x^3 \sigma_\varepsilon \frac{p + \log(2|\mathcal{E}|)+t}{\bar{n}}\|\bbeta - \bbeta^*\|_2\right\}
\end{align} occurs with probability at least $1-e^{-t}$ if $s^*+t+\log(|2\mathcal{E}|) \le n_{\min}$. Define the event $\mathcal{K}_{5,u}(e)$ as
\begin{align*}
    \mathcal{K}_{5,u}(e) = \left\{\left\|\hat{\mathbb{E}}[\bx_{S^*}^{(e)}\varepsilon^{(e)}]\right\|_2 \le C_{1}'\kappa_U^{1/2} \sigma_x \sigma_\varepsilon \Bigg(\sqrt{\frac{s^* + u}{n^{(e)}}} + \frac{s^* + u}{n^{(e)}}\Bigg)\right\}
\end{align*} for some universal constant $C_1'$. Observe that $\|\mathbb{E}[\bx_{S^*}^{(e)}\varepsilon^{(e)}]\|_2 = \sup_{\bv\in \mathbb{R}^{s^*}, \|\bv\|_2=1} \bv^\top \hat{\mathbb{E}}[\bx_{S^*}^{(e)}\varepsilon^{(e)}]$,
and $\bv^\top \hat{\mathbb{E}}[\bx_S^{(e)}\varepsilon^{(e)}]$ with given fixed $\bv\in \mathbb{R}^{s^*}$ is the sum of independent (centered) products of two zero-mean sub-Gaussian random variables with parameter $\kappa_U^{1/2} \sigma_x$ and $\sigma_\varepsilon$, respectively. Then it follows from \cref{lemma:product-sub-gaussian} and \cref{lemma:sum-sub-exp} that, for any fixed $\bv\in \mathbb{R}^{s^*}$ and $x>0$
\begin{align*}
    \mathbb{P}\left[ \left|\bv^\top \hat{\mathbb{E}}[\bx_{S^*}^{(e)}\varepsilon^{(e)}] \right| \ge C'\kappa_U^{1/2}\sigma_x \sigma_\varepsilon\Bigg(\sqrt{\frac{x}{n^{(e)}}} + \frac{x}{n^{(e)}}\Bigg)\right] \le 2e^{-x}.
\end{align*} Following a similar argument as that of {\sc Step 1} and applying union bound gives $\mathbb{P}\left[\mathcal{C}_{5,u}(e)\right] \ge 1-e^{-u}$. 

At the same time, let 
\begin{align*}
    \mathcal{K}_{5,u}'(e) = \left\{ \forall \bbeta \in \mathcal{B}_s, ~~ \left\| \hat{\mathbb{E}}[\bx_{S^*}^{(e)} (\bx^{(e)})^\top] - \bSigma^{(e)}_{S^*,:} \right\|_2 \le C_2'\kappa_U \sigma_x^2 \Bigg(\sqrt{\frac{V(s,u)}{n^{(e)}}} + \frac{V(s,u)}{n^{(e)}}\Bigg)\right\}
\end{align*} for some universal constant $C_2'>0$. It is easy to verify that $\mathbb{P}(\mathcal{C}_{5,u}'(e)) \ge 1-e^{-u}$ for any given fixed $u>0$ and $e\in \mathcal{E}$. 

Under the event $\mathcal{K}_u = \bigcap_{e\in \mathcal{E}} \{\mathcal{K}_{5,u}(e) \cap \mathcal{K}'_{5,u}(e)\}$, which occurs with probability $1-2|\mathcal{E}| e^{-u}$, we obtain
\begin{align*}
    \forall \bbeta \in \mathcal{B}_s, ~~\sum_{e\in \mathcal{E}} \omega^{(e)} \mathsf{T}_5^{(e)} &= \sum_{e\in \mathcal{E}} \omega^{(e)} \left(\hat{\mathbb{E}}[\bx_I^{(e)}\varepsilon^{(e)}]\right)^\top \left(\hat{\mathbb{E}}[\bx_{I}^{(e)} (\bx^{(e)})^\top] - \bSigma^{(e)}_{I,:}\right) \left(\bbeta - \bbeta^*\right) \\
    &\le \sum_{e\in \mathcal{E}} \omega^{(e)} \left\|\hat{\mathbb{E}}[\bx_I^{(e)}\varepsilon^{(e)}]\right\|_2 \left\|\hat{\mathbb{E}}[\bx_{I}^{(e)} (\bx^{(e)})^\top] - \bSigma^{(e)}_{I,:}\right\|_2 \|\bbeta - \bbeta^*\|_2 \\
    &\le \sum_{e\in \mathcal{E}} \omega^{(e)} \left\|\hat{\mathbb{E}}[\bx_{S^*}^{(e)}\varepsilon^{(e)}]\right\|_2 \left\|\hat{\mathbb{E}}[\bx_{S^*}^{(e)} (\bx^{(e)})^\top] - \bSigma^{(e)}_{S^*,:}\right\|_2 \|\bbeta - \bbeta^*\|_2 \\ 
    &\overset{(a)}{\le} C'' \kappa_U^{3/2} \sigma_x^3 \sigma_\varepsilon \sum_{e\in \mathcal{E}} \omega^{(e)} \sqrt{\frac{|S^*| + u}{n^{(e)}}} \Bigg(\sqrt{\frac{V(s, u)}{n^{(e)}}} + \frac{V(s,u)}{n^{(e)}} \Bigg) ~ \|\bbeta - \bbeta^*\|_2 \\
    &\overset{(b)}{\le} C'' \kappa_U^{3/2} \sigma_x^3 \sigma_\varepsilon \frac{\sqrt{V(s,u) V(0,u)}}{\bar{n}} \|\bbeta - \bbeta^*\|_2 \\
    &~~~~~~~~ + C'' \kappa_U^{3/2} \sigma_x^3 \sigma_\varepsilon \frac{V(s,u)\sqrt{V(0,u)}}{n_\dagger} \|\bbeta - \bbeta^*\|_2
\end{align*} provided $s^* + u \le n_{\min}$. Here $(a)$ follows from the definition of the events $\mathcal{K}_{5,u}(e)$ and $\mathcal{K}_{5,u}'(e)$ and the fact that $x \le \sqrt{x}$ when $x\in [0,1]$, $(b)$ follows directly from and the definition of $\bar{n}$ in \eqref{eq:def-n-low-dim} and the definition of $n_\dagger$ in \eqref{eq:def-n-low-dim}. This completes the proof via letting $u = \log(2|\mathcal{E}|) + t$.

\noindent {\sc Step 6. Upper Bound on $\mathsf{T}_6^{(e)}$.} The goal of this step is to derive a high-probability bound for the event
\begin{align}
\label{eq:low-dim-id-bound6}
    \mathcal{C}_{6,t} = \left\{\forall \bbeta \in \mathbb{R}^p, ~~ \sum_{e\in \mathcal{E}} \omega^{(e)} \mathsf{T}_6^{(e)} \le C_6 \kappa_U \sigma_x^2\sigma_\varepsilon^2  \frac{t+ \log(2|\mathcal{E}||S^*|)}{\bar{n}} |S^*\setminus S| \right\}
\end{align} for any $t \in (0, n_{\min} - \log(2|\mathcal{E}||S^*|)]$. Note that both the L.H.S. and R.H.S. of the inequality in \eqref{eq:low-dim-id-bound6} depends on $\bbeta$, or more precisely, $S=\supp(\bbeta)$. Denoting $\delta = C_6 \kappa_U\sigma_x^2\sigma_\varepsilon^2 \{t+ \log(2|\mathcal{E}||S^*|)\}/\bar{n}$, we have the following decomposition
\begin{align*}
    \mathcal{C}_{6,t} &= \bigcup_{T\subseteq S^*} \left\{\forall \bbeta \in \mathbb{R}^p, S^*\setminus \supp(\bbeta) = T, ~~ \sum_{e\in \mathcal{E}} \omega^{(e)} \mathsf{T}_6^{(e)} \le \delta |T| \right\}\\
    &= \bigcup_{T\subseteq S^*} \left\{\forall \bbeta \in \mathbb{R}^p, S^*\setminus \supp(\bbeta) = T, ~~ \sum_{j\in T} \sum_{e\in \mathcal{E}} \omega^{(e)} \left|\hat{\mathbb{E}}[x_j^{(e)}\varepsilon^{(e)}]\right|^2 \le \delta |T| \right\} \\
    &= \bigcup_{T\subseteq S^*} \mathcal{K}(T).
\end{align*}

At the same time, given fixed $j\in S^*$ and $e\in \mathcal{E}$, it follows from \cref{lemma:sum-sub-exp} that,
\begin{align*}
    \mathbb{P}[\mathcal{K}_{6,x}(e, j)] = \mathbb{P}\left[ \left|\hat{\mathbb{E}}[x_j^{(e)} \varepsilon^{(e)}]\right| \le C'\kappa_U^{1/2}\sigma_x\sigma_\epsilon \Bigg(\sqrt{\frac{x}{n^{(e)}}} + \frac{x}{n^{(e)}}\Bigg)\right] \ge 1- 2e^{-x}.
\end{align*} for some universal constant $C'$. We claim that
\begin{align}
\label{eq:low-dim-id-step6-claim1}
    \mathcal{K}(T) \subseteq \bigcup_{j\in S^*, e\in \mathcal{E}} \mathcal{K}_{6,t + \log(2s^*|\mathcal{E}|)}(e,j) 
\end{align} by choosing $C_6=(C')^2$. This is because under $\bigcup_{j\in S^*, e\in \mathcal{E}} \mathcal{K}_{6,t + \log(2s^*|\mathcal{E}|)}(e,j)$, one has
\begin{align*}
    \sum_{j\in T} \sum_{e\in \mathcal{E}} \omega^{(e)} \left|\hat{\mathbb{E}}[x_j^{(e)}\varepsilon^{(e)}]\right|^2 &\le \sum_{j\in T} \sum_{e\in \mathcal{E}} (C')^2 \kappa_U \sigma_x^2\sigma_\varepsilon^2\left(\frac{t + \log(2s^*|\mathcal{E}|)}{n^{(e)}} \omega^{(e)}\right) \\
    &\le |T| (C')^2 \kappa_U \sigma_x^2\sigma_\varepsilon^2\frac{t + \log(2s^*|\mathcal{E}|)}{\bar{n}}
\end{align*} provided $t+\log(2s^*|\mathcal{E}|) \le n_{\min}$, this validates the claim \eqref{eq:low-dim-id-step6-claim1}. Therefore, we have
\begin{align*}
    \mathbb{P}(\mathcal{C}_{6,t}) = \mathbb{P} \left[\bigcup_{T\subseteq S^*} \mathcal{K}(T)\right] &\ge \mathbb{P} \left[ \bigcup_{j\in S^*, e\in \mathcal{E}} \mathcal{K}_{6,t + \log(2s^*|\mathcal{E}|)}(e,j) \right]\\
    &\ge 1 - \sum_{e\in \mathcal{E}, j\in S^*} \left(1-\mathbb{P}\left[\mathcal{K}_{6,t + \log(2s^*|\mathcal{E}|)}(e,j)\right]\right) \\
    &\ge 1-2s^*|\mathcal{E}|e^{-t - \log(2s^*|\mathcal{E}|)} \ge 1-e^{-t}.
\end{align*}

\noindent {\sc Step 7. Conclusion.} We are now ready to conclude the proof by combining results \eqref{eq:low-dim-id-bound1}, \eqref{eq:low-dim-id-bound2}, \eqref{eq:low-dim-id-bound3}, \eqref{eq:low-dim-id-bound4}, \eqref{eq:low-dim-id-bound5}, \eqref{eq:low-dim-id-bound6} we obtained from {\sc Step 1} to {\sc Step 6}. Plugging these upper bounds back into our decomposition in \eqref{eq:low-dim-id-decomposition}, we have that, under the event $\bigcap_{k=1}^6\mathcal{C}_{k,t}$, which occurs with probability at least $1-6e^{-t}$, the inequality in \eqref{eq:instance-dependent-bound-j-highdim} holds provided $n_{\min} \ge (s^* + \log(2|\mathcal{E}|) + t) \lor (\log(2|\mathcal{E}||S^*|) + t) = s^* + \log(2|\mathcal{E}|) + t$. This completes the proof. \qed

\subsection{Proof of \cref{lemma:instance-dependent-bound-r-highdim}}
\label{subsec:proof-instance-dependent-bound-r-highdim}

It follows from the definition of the pooled $L_2$ risk that
\begin{align*}
&~\mathsf{R}(\bbeta) - \mathsf{R}(\bbeta^*) - \hat{\mathsf{R}}(\bbeta) + \hat{\mathsf{R}}(\bbeta^*) \\
&~~~~~~~  = \sum_{e\in \mathcal{E}} \omega^{(e)} \left\{\mathsf{R}^{(e)}(\bbeta) - \mathsf{R}^{(e)}(\bbeta^*) - \left(\hat{\mathsf{R}}^{(e)}(\bbeta) - \hat{\mathsf{R}}^{(e)}(\bbeta^*)\right) \right\}\\
&~~~~~~~  = \sum_{e\in \mathcal{E}} \omega^{(e)} (\bbeta - \bbeta^*)^\top \left(\bSigma^{(e)} - \hat{\bSigma}^{(e)}\right) (\bbeta - \bbeta^*) - 2(\bbeta - \bbeta^*)^\top \left(\mathbb{E}[\bx^{(e)}\varepsilon^{(e)}] - \hat{\mathbb{E}}[\bx^{(e)} \varepsilon^{(e)}]\right) \\
&~~~~~~~  = (\bbeta - \bbeta^*)^\top \bA (\bbeta - \bbeta^*) - 2(\bbeta - \bbeta^*)^\top \bb.
\end{align*}

Recall that $\delta_1 = \sqrt{\frac{V(s,t)}{n_{\bomega}}} + \frac{V(s,t)}{n_*}$, we argue that the following two events
\begin{align*}
    \mathcal{C}_{1,t} &= \left\{ \sup_{(\bx, \by)\in \mathcal{B}_s \times \mathcal{B}_s, \|\bx\|_2=\|\by\|_2=1} \bx^\top \bA \by \le C_1 \kappa_U \sigma_x^2 \delta_1\right\} \\
    \qquad \mathcal{C}_{2,t} &= \left\{ \sup_{\bx \in \mathcal{B}_s, \|\bx\|_2=1} \bx^\top \bb \le C_2 \kappa_U^{1/2} \sigma_x\sigma_\varepsilon \delta_1\right\}
\end{align*} satisfies $\mathbb{P}[\mathcal{C}_{1,t}] \land \mathbb{P}[\mathcal{C}_{2,t}] \ge 1-0.5 e^{-t}$ for any $t>0$, where $C_1, C_2$ are some universal constants to be determined. If the two claims are verified, then we have, under the event $\mathcal{C}_{1,t} \cap \mathcal{C}_{2,t}$ which occurs with probability at least $1-e^{-t}$, the following holds
\begin{align*}
&~|\mathsf{R}(\bbeta) - \mathsf{R}(\bbeta^*) - \hat{\mathsf{R}}(\bbeta) + \hat{\mathsf{R}}(\bbeta^*)| \\
&~~~~~~ \le \left|(\bbeta - \bbeta^*)^\top \bA (\bbeta - \bbeta^*)\right| + 2\left|(\bbeta - \bbeta^*)^\top \bb \right| \\
&~~~~~~ \le \{C_1 \lor (2C_2) \}\left(\kappa_U \sigma_x^2 \delta_1 \|\bbeta - \bbeta^*\|_2^2 + \kappa_U^{1/2} \sigma_x\sigma_\varepsilon \delta_1 \|\bbeta - \bbeta^*\|_2\right)
\end{align*} for some universal constant $C'$. This completes the proof of \cref{lemma:instance-dependent-bound-r-lowdim}.

Now, we prove the two claims separately.

\noindent \emph{Proof of the Claim $\mathbb{P}[\mathcal{C}_{1,t}]\ge 1-0.5e^{-t}$. } The proof strategy is very similar to {\sc Step 3} in the proof of \cref{lemma:instance-dependent-bound-j-lowdim}. To be specific, similar to the derivation in \eqref{eq:step3-sup-argument-2}, there exist $N \le 81^{s^*} (81ep/s)^s$ pairs of $p$-dimensional unit vectors $(\bv_1, \bu_1),\ldots, (\bv_N, \bu_N)$ such that
\begin{align*}
    \sup_{(\bx, \by)\in \mathcal{B}_s \times \mathcal{B}_s, \|\bx\|_2=\|\by\|_2=1} \bx^\top \bA \by \le 4\sup_{k\in [N]} \bv_k^\top \bA \bu_k.
\end{align*}
For fixed $(\bv_k, \bu_k)$, it follows from Conditions \ref{cond0-model}, \ref{cond1-well-condition} and \ref{cond2-subgaussain-x} that $\bv_k^\top \bA \bu_k$ is the sum of independent variables $X_{e,i}-\mathbb{E}[X_{e,i}]$, and $X_{e,i}$ is the product of two zero-mean sub-Gaussian random variables with parameter $\kappa_U^{1/2}\sigma_x$. Then applying \cref{lemma:product-sub-gaussian} and \cref{lemma:sum-sub-exp} gives
\begin{align*}
    \mathbb{P}\left[|\bv_k^\top \bA \bu_k| \ge C' \kappa_U^{1/2} \sigma_x \left(\sqrt{\frac{u}{n_{\bomega}}} + \frac{u}{n_*}\right)\right] \le 2e^{-u}
\end{align*} for any $u>0$ and some universal constant $C'$. Using the union bounds over all the $k$, we have the following event
\begin{align*}
    \sup_{(\bx, \by)\in \mathcal{B}_s \times \mathcal{B}_s} \le 4\sup_{k\in [N]} \bv_k^\top \bA \bu_k \le 4C' \kappa_U^{1/2} \sigma_x \left(\sqrt{\frac{u}{n_{\bomega}}} + \frac{u}{n_*}\right)
\end{align*} will occurs with probability at least $1-2Ne^{-u}$. Letting $u=t+\log(4N) \le 6V(s,t)$ completes the proof.

\noindent \emph{Proof of the Claim $\mathbb{P}[\mathcal{C}_{2,t}]\ge 1-0.5e^{-t}$.} The proof strategy is also similar to the first part of {\sc Step 5} in the proof of \cref{lemma:instance-dependent-bound-j-lowdim}. To be specific, there exist $N=90^p$ $p$-dimensional unit vectors $\bu_1,\ldots, \bu_N$ such that $\sup_{\bx \in \mathcal{B}_s, \|\bx\|_2=1} \bx^\top \bb \le 2\sup_{k\in [N]} \bu_k^\top \bb$. Moreover, for fixed $\bu_k$, it follows from Conditions \ref{cond0-model} -- \ref{cond3-sub-gaussian-eps} such that $\bu_k^\top \bb$ is the sum of independent variables $X_{e,i} - \mathbb{E}[X_{e,i}]$, and $X_{e,i}$ is the product of two zero-mean sub-Gaussian random variables with parameters $\kappa_U^{1/2}\sigma_x$ and $\sigma_\varepsilon$, respectively. Following a similar procedure of applying \cref{lemma:product-sub-gaussian}, \cref{lemma:sum-sub-exp}, and the union bound as above concludes the proof.

\subsection{Proof of \cref{proposition:small-hat-s}}
\label{subsec:proof-small-hat-s}

We need the following lemma. 
\begin{lemma}
\label{lemma:lemma-qhat-beta-star}
    Assume Conditions \ref{cond0-model}--\ref{cond3-sub-gaussian-eps} hold. Define the following event
    \begin{align}
        \mathcal{A}_{5,t} = \left\{\hat{\mathsf{Q}}(\bbeta^*;\gamma,\bomega) \le \sigma_\varepsilon^2 \left( 1 + c_1 \sqrt{\frac{t}{n_*}} + c_1 \kappa_U\sigma_x^2 \gamma \frac{s^*( \log(2s^*|\mathcal{E}|) + t)}{\bar{n}}\right)\right\}
    \end{align} for some universal constant, we have $\mathbb{P}(\mathcal{A}_{5,t}) \ge 1-2e^{-t}$. 
\end{lemma}

\begin{proof}[Proof of \cref{lemma:lemma-qhat-beta-star}]
It follows from the definition that
\begin{align*}
    \hat{\mathsf{Q}}(\bbeta^*) &= \sum_{e\in \mathcal{E}} \omega^{(e)} \left\{\hat{\mathbb{E}} |\varepsilon^{(e)}|^2 + \gamma \sum_{j\in S^*} \left(\hat{\mathbb{E}} [x_j^{(e)} \varepsilon^{(e)}]\right)^2\right\} \\
    &= \sum_{e\in \mathcal{E}} \omega^{(e)} \mathbb{E} |\varepsilon^{(e)}|^2 + \sum_{e\in \mathcal{E}} \omega^{(e)} \left(\hat{\mathbb{E}} |\varepsilon^{(e)}|^2 - {\mathbb{E}} |\varepsilon^{(e)}|^2\right) + \gamma \sum_{e\in \mathcal{E}} \omega^{(e)} \left(\hat{\mathbb{E}} [x_j^{(e)} \varepsilon^{(e)}]\right)^2.
\end{align*} It follows from \cref{cond3-sub-gaussian-eps} that
\begin{align}
\label{eq:proof-qhat-beta-star-c0}
    \sum_{e\in \mathcal{E}} \omega^{(e)} \mathbb{E} |\varepsilon^{(e)}|^2 \le \sum_{e\in \mathcal{E}} \omega^{(e)} \sigma_\varepsilon^2 = \sigma_\varepsilon^2,
\end{align} and for any $t>0$,
\begin{align}
\label{eq:proof-qhat-beta-star-c1}
    \mathbb{P}\left[\sum_{e\in \mathcal{E}} \omega^{(e)} \left(\hat{\mathbb{E}} |\varepsilon^{(e)}|^2 - {\mathbb{E}} |\varepsilon^{(e)}|^2\right) \le C' \sigma_\varepsilon^2 \left(\sqrt{\frac{t}{n_{\bomega}}} + \frac{t}{n_{*}}\right)\right] \ge 1-e^{-x}
\end{align} for some universal constant $C'>0$. At the same time, for any fixed $e\in \mathcal{E}$, it follows from \cref{lemma:product-sub-gaussian} and \cref{lemma:sum-sub-exp} that, for any $x>0$ and $j\in [S^*]$,
\begin{align*}
    \mathbb{P}[\mathcal{K}_{t}(e,j)] = \mathbb{P}\left[ \left|\hat{\mathbb{E}} [x_j^{(e)} \varepsilon^{(e)}] - \mathbb{E}[x_j^{(e)} \varepsilon^{(e)}]\right| \ge C' \kappa_U^{1/2}\sigma_x\sigma_\varepsilon \left(\sqrt{\frac{x}{n^{(e)}}} + \frac{x}{n^{(e)}}\right)\right] \le 1-2e^{-x}.
\end{align*} Note $\mathbb{E}[x_j^{(e)} \varepsilon^{(e)}] = 0$. Then under $\mathcal{K}_t = \bigcap_{j\in S^*, e\in \mathcal{E}} \mathcal{K}_{t+\log(2s^*|\mathcal{E}|)}(e,j)$, which occurs with probability at least $1-e^{-t}$, we have
\begin{align}
\label{eq:proof-qhat-beta-star-c2}
    \sum_{j\in S^*} \sum_{e\in \mathcal{E}} \omega^{(e)} \left(\hat{\mathbb{E}} [x_j^{(e)} \varepsilon^{(e)}] \right)^2 \le (C')^2 \kappa_U \sigma_x^2 \sigma_\varepsilon^2 \frac{s^*(\log(2s^*|\mathcal{E}|) + t)}{\bar{n}}
\end{align} provided $t+\log(2s^*|\mathcal{E}|) \le n_{\min}$. Putting the three bounds \eqref{eq:proof-qhat-beta-star-c0}, \eqref{eq:proof-qhat-beta-star-c1} and \eqref{eq:proof-qhat-beta-star-c2} together completes the proof.
\end{proof}

Now we are ready to prove \cref{proposition:small-hat-s}.
\begin{proof}[Proof of \cref{proposition:small-hat-s}]
Observing that \begin{align*}
    \left\{|\hat{S}| \le 2s^*\right\} \subseteq \left\{ \forall \bbeta~\text{with}~\supp(\bbeta) > 2s^*,~~ \hat{\mathsf{Q}}(\bbeta) + \lambda \|\bbeta\|_0 > \hat{\mathsf{Q}}(\bbeta^*) + \lambda \|\bbeta^*\|_0 \right\} := \mathcal{C}_t,
\end{align*} it remains to show that $\mathbb{P}(\mathcal{C}_t) \ge 1-e^{-t}$. To this end, we use a peeling device. Let $\alpha_\ell=2^{\ell} s^*$, then 
\begin{align*}
    \mathcal{C}_t = \bigcup_{\ell=1}^{\lceil \log_2(p/s^*) \rceil - 1} \left\{\forall \bbeta \in \mathcal{B}_{\alpha_{\ell+1}} \setminus \mathcal{B}_{\alpha_{\ell}}, ~~ \hat{\mathsf{Q}}(\bbeta) + \lambda \|\bbeta\|_0 > \hat{\mathsf{Q}}(\bbeta^*) + \lambda \|\bbeta^*\|_0\right\} = \bigcup_{\ell=1}^{\lceil \log_2(p/s^*) \rceil - 1} \mathcal{C}_{t}(\ell).
\end{align*}
We claim that if $\lambda$ satisfies the conditions presented in the statement, then \begin{align}
\label{eq:proof-small-hats-peeling-argument}
    \bigcup_{\ell=1}^{\lceil \log_2(p/s^*) \rceil - 1} \mathcal{C}_{t}(\ell) \subseteq \mathcal{A}_{5,t} \cup \left\{\bigcup_{\ell=1}^{\lceil \log_2(p/s^*) \rceil - 1} \mathcal{A}_{3,u_{\ell}}(\alpha_{\ell+1} s^*) \cup \mathcal{A}_{4,u_{\ell}}(\alpha_{\ell+1} s^*)\right\}.
\end{align} with $u_\ell = t + \log (\lceil \log_2(p/s^*) \rceil) \le t + \log (ep/s^*)$.
Denote $\phi_* = s^* + \log(ep/s^*)+t$ and
\begin{align*}
    \tilde{s} = \frac{n_*}{2\{\log(ep/s^*) + t\}} \land \frac{\left(\frac{1}{12C_1} \frac{\kappa_L}{\gamma \kappa_U^2\sigma_x^2}\right)^2n_*}{\log(ep/s^*) + t} \land \frac{n_\dagger}{\{\log(ep/s^*)+t\} \sqrt{\phi_*} (\gamma/\kappa_L)\kappa_U^{3/2} \sigma_x^2 },
\end{align*} where $C_1$ is some universal constant to be determined. We prove this claim by considering the following two cases.

\noindent \emph{Case 1, $\alpha_{\ell+1} s^* \le \tilde{s}$.} Let $S=\supp(\bbeta)$. Denote $\phi_s=\alpha_{\ell+1} s^* \log(ep/s^*) + t$. In this case, there exists some universal constant $C_1$ such that if $\mathcal{A}_{3,u_{\ell}}(\alpha_{\ell+1}s^*) \cap \mathcal{A}_{4,u_{\ell}}(\alpha_{\ell+1}s^*)$ occurs, then we have, for any $\bbeta \in \mathcal{B}_{\alpha_{\ell+1}}$,
\begin{align*}
    \hat{\mathsf{Q}}(\bbeta) - \hat{\mathsf{Q}}(\bbeta^*) &= \hat{\mathsf{Q}}(\bbeta) - \mathsf{Q}(\bbeta) + \mathsf{Q}(\bbeta) - \mathsf{Q}(\bbeta^*) + \mathsf{Q}(\bbeta^*) - \hat{\mathsf{Q}}(\bbeta^*) \\
    &\overset{(a)}{\ge} \frac{\kappa_L}{2} \|\bbeta - \bbeta^*\|_2^2 + \frac{\gamma}{6} \mathsf{J}(\bbeta; \bomega) + \hat{\mathsf{R}}(\bbeta) - \mathsf{R}(\bbeta) + \mathsf{R} (\bbeta^*) - \hat{\mathsf{R}} (\bbeta^*) \\
    & ~~~~~~~~~~~~ - \gamma \left(\mathsf{J}(\bbeta) - \hat{\mathsf{J}}(\bbeta) - \mathsf{J}(\bbeta^*) + \hat{\mathsf{J}}(\bbeta^*)\right)\\
    &\overset{(b)}{\ge} \left(\frac{\kappa_L}{2} - C_1 \kappa_U^2\sigma_x^2 \gamma \sqrt{\frac{\phi_s}{n_*}}\right) \|\bbeta - \bbeta^*\|_2^2 \\
    & ~~~~~~~~~~~~ + \frac{\gamma}{6} \mathsf{J}(\bbeta;\bomega) - C_1 \gamma \sqrt{\sum_{e\in \mathcal{E}} \omega^{(e)} \left\|\mathbb{E}[\bx_{S}^{(e)} \varepsilon^{(e)}]\right\|_2^2} \kappa_U^{1/2} \sigma_x\sigma_\varepsilon \sqrt{\frac{\phi_s}{n_*}} \\
    & ~~~~~~~~~~~~ - C_1 \|\bbeta - \bbeta^*\|_2 \times \gamma \kappa_U^{3/2} \sigma_x^2 \sigma_\varepsilon \sqrt{\frac{\phi_s}{n_*}} \\
    & ~~~~~~~~~~~~ - C_1 \|\bbeta - \bbeta^*\|_2 \times \gamma \kappa_U^{3/2} \sigma_x^3 \sigma_\varepsilon \left(\frac{\sqrt{\phi_*} \sqrt{\phi_s}}{\bar{n}} + \frac{\phi_s \sqrt{\phi_*}}{n_\dagger} \right) \\
    & ~~~~~~~~~~~~ - C_1 \gamma |S^*\setminus S| \kappa_U \sigma_x^2 \sigma_\varepsilon^2 \frac{t+ \log p}{\bar{n}} - C_1 \gamma \kappa_U \sigma_x \sigma_\varepsilon^2 \frac{\phi_s}{n_*} \\
    &= \mathsf{I}_1 + \mathsf{I}_2 + \mathsf{I}_3 + \mathsf{I}_4 + \mathsf{I}_5,
\end{align*} where $(a)$ follows from \cref{prop:population-lb}, $(b)$ follows from the bounds in \cref{lemma:instance-dependent-bound-j-highdim} \& \ref{lemma:instance-dependent-bound-r-highdim} together with the facts that $\log(|\mathcal{E}|) \lesssim \log p$, $\phi_s/n_* \le \tilde{s}(\log(ep/s^*)+t)/ n_* \le 1$, and $|S\setminus S^*| \log(ep/|S\setminus S^*|) + s^* \le 2|S\setminus S^*| \log(ep/s^*) $ because $|S\setminus S^*| \ge s^*$. For $\mathsf{I}_1$, we have $\mathsf{I}_1 \ge \frac{5}{12}\kappa_L \|\bbeta - \bbeta^*\|_2^2$ whenever
\begin{align*}
    C_1 \kappa_U^2\sigma_x^2 \gamma \sqrt{\frac{\phi_s}{n_*}} \le C_1 \kappa_U^2 \sigma_x^2 \gamma \sqrt{\frac{\tilde{s} \{\log(ep/s^*) + t\}}{n_*}} \le \frac{\kappa_L}{12}.
\end{align*} Denote $\Diamond=0.5 (6C_1)^2 (\gamma/\kappa_L)^2 \kappa_U^3 \sigma_x^4\sigma_\varepsilon^2$. 
For $\mathsf{I}_2$, it follows from \cref{lemma:removecondition7} that
\begin{align*}
    \mathsf{I}_2 &\ge \frac{\gamma}{6} \mathsf{J}(\bbeta; \bomega) - C_1 \gamma \sqrt{2\kappa_U^2\|\bbeta - \bbeta^*\|_2^2 + 2\mathsf{J}(\bbeta; \bomega)} \kappa_U^{1/2} \sigma_x \sigma_\varepsilon \sqrt{\frac{\phi_s}{n_*}} \\
    &\ge \frac{\gamma}{6} \mathsf{J}(\bbeta;\bomega) - \frac{\kappa_L}{12} \|\bbeta-\bbeta^*\|_2^2 - 4\Diamond \cdot \frac{\phi_s}{n_*} - \frac{\gamma}{6} \mathsf{J}(\bbeta;\bomega) \\
    &\ge - \frac{\kappa_L}{12} \|\bbeta-\bbeta^*\|_2^2 - 4\Diamond \cdot \frac{(\alpha_{\ell+1} s^*) \{\log(ep/s^*)+t\}}{n_*}
\end{align*}
For $\mathsf{I}_3$ and $\mathsf{I}_4$, it follows from the fact $xy \le \frac{1}{2} (x^2 + y^2)$ that
\begin{align*}
    \mathsf{I}_3 \ge -\frac{\kappa_L}{12} \|\bbeta - \bbeta^*\|_2^2 - \Diamond \cdot \frac{(\alpha_{\ell+1} s^*) \{\log(ep/s^*) + t\} }{n_*} 
\end{align*} and
\begin{align*}
    \mathsf{I}_4 \ge & ~-\frac{\kappa_L}{12} \|\bbeta - \bbeta^*\|_2^2 - 2\Diamond \cdot \frac{(\alpha_{\ell+1} s^*) \cdot \phi_* \left\{\log(ep/s^*) + t\right\}}{\bar{n}^2} \\
    &~~~~~~~~~~~~ - 2\Diamond \cdot \frac{(\alpha_{\ell+1} s^*) \cdot \phi_* \tilde{s} \{\log(ep/s^*) + t\}^2}{n_\dagger^2}.
\end{align*} For $\mathsf{I}_5$, it follows from \eqref{eq:low-dim-bound-i4-1} that
\begin{align*}
\mathsf{I}_5 &\ge -C_1 \|\bbeta - \bbeta_*\|_2^2  \gamma \left\{\frac{\kappa_U \sigma_x^2\sigma_\varepsilon^2}{ \min_{j\in S^*} |\bbeta_j^ *|^2}\frac{t + \log(es^*p)}{\bar{n}}\right\} - C_1 \gamma \kappa_U\sigma_x\sigma_\varepsilon^2 \frac{(\alpha_{\ell+1} s^*) \{\log(ep/s^*) + t\}}{n_*} \\
&\ge -\frac{\kappa_L}{12} \|\bbeta - \bbeta^*\|_2^2 - C_1 \gamma \kappa_U\sigma_x\sigma_\varepsilon^2 \frac{(\alpha_{\ell+1} s^*) \{\log(ep/s^*) + t\}}{n_*}
\end{align*} if $\bar{n} \ge 12C_1 (\gamma/\kappa_L) \kappa_U\sigma_x^2\sigma_\varepsilon^2 (t + \log p) / \min_{j\in S^*} |\beta_j^*|^2$. 
Recall that $\zeta = t + \log(ep/s^*)$. 
Putting all the pieces together and plugging $\zeta$ in, we obtain that, for any $\bbeta \in \mathcal{B}_{\alpha_{\ell+1}s^*} \setminus \mathcal{B}_{\alpha_{\ell}s^*}$
\begin{align*}
    & ~\lambda \|\bbeta\|_0 - \lambda \|\bbeta^*\|_0 + \hat{\mathsf{Q}}(\bbeta) - \hat{\mathsf{Q}}(\bbeta^*) \\
    & ~~~~ \overset{(a)}{\ge} \lambda \left(\alpha_{\ell} -  1\right) s^* + \hat{\mathsf{Q}}(\bbeta) - \hat{\mathsf{Q}}(\bbeta^*) \ge  \frac{\lambda}{4} \alpha_{\ell+1} s^* + \hat{\mathsf{Q}}(\bbeta) - \hat{\mathsf{Q}}(\bbeta^*) \\
    & ~~~~ \ge \frac{1}{2}\left(\kappa_L - \frac{\kappa_L}{6} \times 5\right) \|\bbeta - \bbeta^*\|_2^2 \\
    & ~~~~~~~~~~~~ + \alpha_{\ell+1} s^* \Bigg\{\frac{\lambda}{4} - 6\Diamond \cdot \left(\frac{\zeta}{n_*} + \frac{\phi_* \cdot \zeta}{\bar{n}^2} + \frac{\phi_* \tilde{s}\cdot \zeta^2}{n_\dagger^2}\right)\Bigg\},
\end{align*} where $(a)$ follows from the fact that $\bbeta \notin \mathcal{B}_{\alpha_\ell s^*}$. Recall our choice of $\tilde{s}$. We can conclude that if
\begin{align}
\label{eq:proof-lambda-req2}
\begin{split}
    \lambda &\ge 432C_1^2 \kappa_U^{3} \sigma_x^4 (\gamma/\kappa_L)^2 \sigma_{\varepsilon}^2 \left(\frac{\zeta}{n_*} + \frac{\zeta \phi_*}{\bar{n}^2}\right) + 432C_1^2 \kappa_U^{3/2} \sigma_x^2 (\gamma/\kappa_L) \sigma_\varepsilon^2 \frac{\zeta \sqrt{\phi_*}}{n_\dagger} \\
    &\ge 24\Diamond \cdot \left(\frac{\zeta}{n_*} + \frac{\phi_* \cdot \zeta}{\bar{n}^2} + \frac{\phi_* \tilde{s}\cdot \zeta^2}{n_\dagger^2} \right),
\end{split}
\end{align} then 
\begin{align*}
    \underbrace{\left\{\forall \bbeta \in \mathcal{B}_{\alpha_{\ell+1} s^*} \setminus \mathcal{B}_{\alpha_{\ell} s^*}, ~~ \hat{\mathsf{Q}}(\bbeta) + \lambda \|\bbeta\|_0 - \hat{\mathsf{Q}}(\bbeta^*) - \lambda \|\bbeta^*\|_0 > 0\right\}}_{C_t(\ell)} \subseteq \mathcal{A}_{3,u_{\ell}}(\alpha_{\ell+1} s^*) \cup \mathcal{A}_{4,u_{\ell}}(\alpha_{\ell+1} s^*).
\end{align*}

\noindent \emph{Case 2, $\alpha_{\ell+1} s^* \ge \tilde{s}$.} Observe that $\alpha_{\ell} s^* \ge \frac{1}{2} \alpha_{\ell+1} s^* \ge \frac{1}{2}\tilde{s}$ in this case. Hence the following holds
\begin{align*}
    \lambda \|\bbeta\|_0 - \lambda \|\bbeta^*\| + \hat{\mathsf{Q}}(\bbeta) - \hat{\mathsf{Q}}(\bbeta^*) \ge \frac{\lambda}{2} \tilde{s} - \hat{\mathsf{Q}}(\bbeta^*) - \lambda s^* \ge \frac{\lambda}{4} \tilde{s} - \hat{\mathsf{Q}}(\bbeta^*)
\end{align*} provided $4s^* \le \tilde{s}$. At the same time, if $t\le n_*$ and $\bar{n} \ge \kappa_U\sigma_x^2 \gamma s^*( \log(2s^*|\mathcal{E}|) + t)$, then it follows from \cref{lemma:lemma-qhat-beta-star} that under $\mathcal{A}_{5,t}$, 
\begin{align*}
    \hat{\mathsf{Q}}(\bbeta^*) \le (2 + C_2) \sigma_\varepsilon^2,
\end{align*} where $C_2 = c_1$ is the constant in \cref{lemma:lemma-qhat-beta-star}. At the same time, it follows from our construction of $\tilde{s}$ that
\begin{align*}
    (\tilde{s})^{-1} \le (12C_1)^2 \kappa_U^4\sigma_x^4 (\gamma/\kappa_L)^2 \frac{\zeta}{n_*} + \kappa_U^{3/2} \sigma_x^2 (\gamma/\kappa_L) \frac{\zeta \sqrt{\phi_*} }{n_\dagger}.
\end{align*}
Therefore, we can argue that $\mathcal{C}_{t}(\ell) \subseteq \mathcal{A}_{5,t}$ when 
\begin{align}
\label{eq:lambda-req-3}
\lambda \ge 8(2+C_2) (16C_1)^2 \kappa_U^4\sigma_x^4 \sigma_\varepsilon^2 (\gamma/\kappa_L)^2 \frac{\zeta}{n_*} + 8(2+C_2) \kappa_U^{3/2} \sigma_x^2 \sigma_\varepsilon^2 (\gamma/\kappa_L) \frac{\zeta \sqrt{\phi_*} }{n_\dagger} \ge 8(2+C_2) \sigma_\varepsilon^2 \tilde{s}^{-1}.
\end{align}

Now, we combine the two cases. According to the above discussion, we find that under the conditions 
\begin{align*}
    \tilde{s} \ge 4s^*\qquad \text{and} \qquad \bar{n} \ge C' \kappa_U\sigma_x^2 \gamma (t + \log p) \left\{ s^* +  \sigma_\varepsilon^2 / (\kappa_L \min_{j\in S^*} |\beta_j^*|^2)\right\},
\end{align*} we have \eqref{eq:proof-small-hats-peeling-argument} holds if our choice of $\lambda$ satisfies \eqref{eq:proof-lambda-req2} and \eqref{eq:lambda-req-3}. It then follows from the union bound that
\begin{align*}
    \mathbb{P}\left[\mathcal{C}_t^c\right] \le \sum_{\ell=1}^{\lceil \log_2(p/s^*) \rceil - 1} \left\{\mathbb{P}[\mathcal{A}_{3,u_\ell}(\alpha_{\ell+1} s^*)^c] + \mathbb{P}[\mathcal{A}_{4,u_\ell} (\alpha_{\ell+1} s^*)] \right\}+ \mathbb{P}(\mathcal{A}_{5,t}^c) \le 3e^{-t}.
\end{align*} This completes the proof.

\end{proof}

\subsection{Technical Lemmas}

\begin{lemma}
\label{lemma:removecondition7}
Under Conditions \ref{cond0-model} and \ref{cond1-well-condition}, we have
\begin{align*}
    \forall \bbeta \in \mathbb{R}^p \qquad \sum_{e\in \mathcal{E}} \omega^{(e)} \left\|\mathbb{E}[\bx_{\supp(\bbeta)}^{(e)} \varepsilon^{(e)}] \right\|_2^2 \le 2 \mathsf{J}(\bbeta;\bomega) + 2\kappa_U^2 \|\bbeta - \bbeta^*\|_2^2
\end{align*}
\end{lemma}
\begin{proof}[Proof of \cref{lemma:removecondition7}]
Let $S=\supp(\bbeta)$, it follows from the definition of $\mathsf{J}$ that
\begin{align*}
    \mathsf{J}(\bbeta;\bomega) = \sum_{e\in \mathcal{E}} \omega^{(e)} \|\mathbb{E}[(y^{(e)} - \bbeta^\top \bx^{(e)}) \bx_S^{(e)}]\|_2
    &= \sum_{e\in \mathcal{E}} \omega^{(e)} \left\|\mathbb{E}\left[(\varepsilon^{(e)} + (\bbeta^*)^\top \bx^{(e)} - \bbeta^\top \bx^{(e)}) \bx_S^{(e)}\right]\right\|_2.
\end{align*} Denote $\overline{S} = S^* \cup S$. Then it follows from the fact $(a+b)^2 \le 2(a^2+b^2)$ and \cref{cond1-well-condition} that
\begin{align*}
    \qquad \sum_{e\in \mathcal{E}} \omega^{(e)} \left\|\mathbb{E}[\bx_{S}^{(e)} \varepsilon^{(e)}] \right\|_2^2 &= \sum_{e\in \mathcal{E}} \omega^{(e)} \left\|\mathbb{E}[\bx_S^{(e)} \varepsilon^{(e)} + \bx_S^{(e)} (\bx^{(e)}_{\overline{S}})^\top  (\bbeta^*-\bbeta)_{\overline{S}} ] - \bSigma^{(e)}_{S,\overline{S}} (\bbeta^*-\bbeta)_{\overline{S}}\right\|_2^2 \\
    &\le 2 \sum_{e\in \mathcal{E}} \omega^{(e)} \left\|\mathbb{E}[\bx_S^{(e)} \varepsilon^{(e)} + \bx_S^{(e)} (\bx^{(e)}_{\overline{S}})^\top  (\bbeta^*-\bbeta)_{\overline{S}} ]\right\|_2^2 + 2\kappa_U^2 \|\bbeta - \bbeta^*\|_2^2 \\
    &\le 2\mathsf{J}(\bbeta;\bomega) + 2\kappa_U^2 \|\bbeta - \bbeta^*\|_2^2.
\end{align*} This completes the proof.
\end{proof}

\begin{lemma}
\label{lemma:x-product-epsilon}
Under \cref{cond1-well-condition} and \cref{cond3-sub-gaussian-eps}, we have that, for any $e\in \mathcal{E}$,
\begin{align*}
\left\| \mathbb{E}[\varepsilon^{(e)} \bx^{(e)}]	\right\|_2 \le \sigma_\varepsilon \kappa_U^{1/2}.
\end{align*}
\end{lemma}

\begin{proof}[Proof of \cref{lemma:product-sub-gaussian}]
	Let $\bz = ((\bx^{(e)})^\top, \varepsilon^{(e)})^\top$ be a $p+1$-dimensional random vector. Define the matrix $\bA = \mathbb{E}[\bz \bz^\top]$, we have
    \begin{align*}
		\bA = \begin{bmatrix}	
			\bSigma^{(e)} & \mathbb{E}[\bx^{(e)} \varepsilon^{(e)}] \\
			(\mathbb{E}[\bx^{(e)} \varepsilon^{(e)}])^{\top} & \mathbb{E}[|\varepsilon^{(e)}|^2]
 		\end{bmatrix}.
	\end{align*}
    The matrix is positive semi-definite. Combining this with the fact that $\bSigma^{(e)}$ is invertible indicates that the Schur complement is non-negative, that is,
	\begin{align*}
		\mathbb{E}[|\varepsilon^{(e)}|^2] - \left(\mathbb{E}[\varepsilon^{(e)} \bx^{(e)}]\right)^\top \left(\bSigma^{(e)} \right)^{-1} \left(\mathbb{E}[\varepsilon^{(e)} \bx^{(e)}] \right) \ge 0.
	\end{align*} Hence we have
	\begin{align*}
		\left\|\mathbb{E}[\varepsilon^{(e)} \bx^{(e)}]\right\|^2 \kappa_U^{-1} &\overset{(a)}{\le} \left\|\mathbb{E}[\varepsilon^{(e)} \bx^{(e)}]\right\|^2 \lambda_{\min}\{(\bSigma^{(e)})^{-1}\} \\
		&\le \left(\mathbb{E}[\varepsilon^{(e)} \bx^{(e)}]\right)^\top \left(\bSigma^{(e)} \right)^{-1} \left(\mathbb{E}[\varepsilon^{(e)} \bx^{(e)}] \right) \le \mathbb{E}[|\varepsilon^{(e)}|^2] \overset{(b)}{\le} \sigma_\varepsilon^2.
	\end{align*} where $(a)$ follows from \cref{cond1-well-condition} and $(b)$ follows from \cref{cond3-sub-gaussian-eps}. This completes the proof.
\end{proof}

\section{Omitted Discussions in Main Text}
\label{sec:discussion}

\subsection{Comparison with \cite{fan2014endogeneity}}
\label{subsec:discussionliao}

Both \cite{fan2014endogeneity} and our work aim to deal with the problem of endogeneity from a technical perspective. Our constructed focused linear invariance regularizer is similar to their developed FGMM criterion function from the view of the \emph{over-identification} idea they proposed. 

Briefly speaking, we say a parameter $\tilde{\bbeta}$ is over-identified if there are more restrictions than the degree of freedom. When $|\mathcal{E}|=1$, there will be exponential number of distinct $\bbeta$ satisfying $\mathsf{J}(\bbeta;\bomega)=0$. To see this, for any $S \subseteq [p]$, one can find some $\bbeta$ with $\supp(\bbeta) \subseteq S$ satisfying $|S|$ constraints that $\mathbb{E}[(y^{(1)} - \bbeta_S^\top \bx_S^{(1)}) x_j^{(1)}] = 0$ for any $j \in [S]$ because the degree of freedom and the number of constraints are both $|S|$. However, things may be different when $|\mathcal{E}| \ge 2$. In this case, for a given $S \subseteq [p]$, one need to find some $\bbeta$ with a degree of freedom $|S|$ satisfies $|\mathcal{E}| \cdot |S|$ constraints 
\begin{align*}
    \forall e \in \mathcal{E}, j\in [S], ~~~~~~ \mathbb{E}\left[x^{(e)}_j (y^{(e)} - \bbeta_S^\top \bx_S^{(e)})\right] = 0
\end{align*}
simultaneously to let $\mathsf{J}(\bbeta;\bomega)=0$, which does not hold in general.

According to the above discussion, the focused linear invariance regularizer shares a similar spirit with their proposed FGMM from a technical viewpoint since the over-identification of our regularizer comes from multiple environments. In contrast, theirs come from two instrumental variables or two marginal features of the covariate, for example, $x_j$ and $x_j^2$. However, the statistical models the two papers work on are different. We briefly remark on the differences between our method and theirs using marginal nonlinear features as follows:
\begin{itemize}
    \item[1.] We are working with multiple environment settings. The necessity of heterogeneous environments and the potential violation of their identification condition are illustrated by \cref{prop:necessary-multiple-envs}. In particular, if two marginal features of the covariate are used in their method, the corresponding identification condition is a sufficient condition of the condition that $S^*$ is the only CE-invariant set among $\mathcal{E}=\{1\}$ other than $\emptyset$.
    \item[2.] We use a linear combination of invariance regularizer and the $L_2$ loss to avoid collapsing to conservative solutions. At the same time, theirs will suffer from the case where two groups of important variables are independent.
    \item[3.] We provide a more explainable identification condition in the context of multiple environments and a non-asymptotic upper bound on the critical threshold of the regularization hyper-parameter $\gamma$; see the intuition explanations in \cref{subsubsec:debias} and a formal presentation in \cref{subsec:strong-convexity}. While \cite{fan2014endogeneity} has few discussions on the population-level conditions.
    \item[4.] The finite sample analyses are completely different because (1) the objective functions are different; and (2) we establish the variable selection consistency for the global minimizer while they only show that some good local minimizer satisfying the variable selection consistency exists. We should carefully deal with the dependence on the number of environments and apply a novel localization argument to obtain a fast rate or a weak condition for variable selection consistency. 
\end{itemize}

\subsection{Comparison with IRM}
\label{sec:irm}

Our constructed EILLS objective is similar to the invariant risk minimization criterion \citep{arjovsky2019invariant} by letting $\gamma \to \infty$. To see this, when $\gamma\to \infty$, the population-level EILLS objective becomes
\begin{align*}
    \min_{\bbeta \in \mathbb{R}^p} \mathsf{R}(\bbeta;\bomega) ~~~~~~s.t.~~~~~~ \mathsf{J}(\bbeta;\bomega)=0.
\end{align*} This optimization problem finds the most effective solution in the sense of small $L_2$ risk among all the LLS-invariant solutions as discussed in \cref{subsec:focused-regularizer}. However, it is unclear (1) what it implies when $\mathsf{J}(\bbeta;\bomega) \approx 0$; and (2) how large $\gamma$ must be to estimate $\beta^*$ consistently. Such two concerns are important for the non-asymptotic analysis because of finite-sample data and finite choice of $\gamma$. We provide an intuitive explanation in \cref{subsubsec:debias} addressing the above concerns, which can be treated as an informal presentation of our theory in \cref{subsec:strong-convexity}.

\subsection{Debiasing by Regularizer via Bias Differences} 
\label{subsubsec:debias} 

\cref{subsec:focused-regularizer} and \cref{sec:irm} provide a sketchy glimpse at the effects of two losses $\mathsf{J}$ and $\mathsf{R}$ on the solution. The following concerns remain: (1) we only analyze the implication of focused linear invariance regularizer $\mathsf{J}(\bbeta;\bomega)$ on the solution when $\mathsf{J}(\bbeta;\bomega)=0$, but we can not expect $\mathsf{J}=0$ in practice especially when it comes to the analysis of its empirical analogs with finite sample data; (2) the above paragraph characterizes that property of proposed EILLS estimator in $\gamma \to \infty$, while it is still unknown how large $\gamma$ is enough. In this part, we will provide a different perspective. In a high-level viewpoint, the $L_2$ risk will have some bias in the presence of the \emph{pooled linear spurious variables} satisfying $\sum_{e\in \mathcal{E}} \omega^{(e)} \mathbb{E} [\varepsilon^{(e)} x_j^{(e)}] \neq 0$. This is where our proposed regularizer $\mathsf{J}$ comes into play: it debiases the pooled least squares loss using the difference of biases between heterogeneous environments, as illuminated below. Such insight also answers the question of how large $\gamma$ is enough.

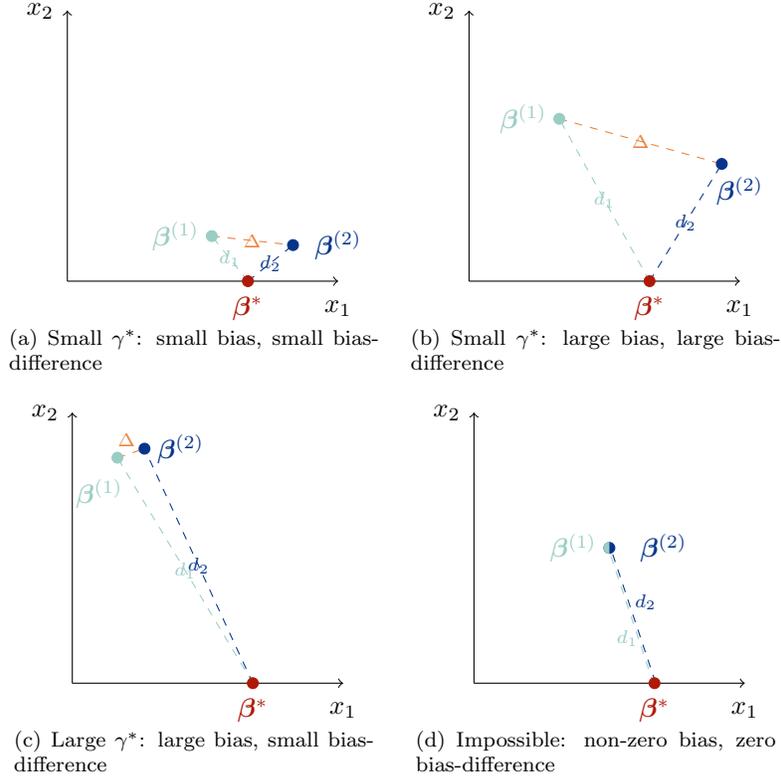
\begin{figure}[!t]
    \centering
    \begin{tabular}{cc}
    \subfigure[Small $\gamma^*$: small bias, small bias-difference]{
        \begin{tikzpicture}[scale=1.2]
        \draw[->] (0,0) -- (0,3);
        \draw[->] (0,0) -- (3,0);
        \draw (3, -0.3) node{$x_1$};
        \draw (-0.3, 3) node{$x_2$};
        \draw[mylightblue, dashed] (2, 0) -- (1.8, 0.25) node{\scriptsize $d_1$} -- (1.6, 0.5);
        \draw[myblue, dashed] (2, 0) -- (2.25, 0.2) node{\scriptsize $d_2$} -- (2.5, 0.4);
        \draw[myorange, dashed] (1.6, 0.5) -- (2.05, 0.45) node{\scriptsize $\Delta$} -- (2.5, 0.4);
        \draw[myred, fill] (2, 0) circle[radius=0.06];
        \draw[myred] (2,-0.3) node{${\bbeta}^*$};
        \draw[mylightblue, fill] (1.6,0.5) circle[radius=0.06];
        \draw[mylightblue] (1.2, 0.5) node{$\bbeta^{(1)}$};
        \draw[myblue, fill] (2.5, 0.4) circle[radius=0.06];
        \draw[myblue] (3, 0.4) node {$\bbeta^{(2)}$};
        \end{tikzpicture} 
    }
    &
        \subfigure[Small $\gamma^*$: large bias, large bias-difference]{
        \begin{tikzpicture}[scale=1.2]
        \draw[->] (0,0) -- (0,3);
        \draw[->] (0,0) -- (3,0);
        \draw (3, -0.3) node{$x_1$};
        \draw (-0.3, 3) node{$x_2$};
        \draw[mylightblue, dashed] (2, 0) -- (1.5, 0.9) node{\scriptsize $d_1$} -- (1, 1.8);
        \draw[myblue, dashed] (2, 0) -- (2.4, 0.65) node{\scriptsize $d_2$} -- (2.8, 1.3);
        \draw[myorange, dashed] (1, 1.8) -- (1.9, 1.55) node{\scriptsize $\Delta$} -- (2.8, 1.3);
        \draw[myred, fill] (2, 0) circle[radius=0.06];
        \draw[myred] (2,-0.3) node{${\bbeta}^*$};
        \draw[mylightblue, fill] (1,1.8) circle[radius=0.06];
        \draw[mylightblue] (0.6, 1.8) node{$\bbeta^{(1)}$};
        \draw[myblue, fill] (2.8, 1.3) circle[radius=0.06];
        \draw[myblue] (3, 1) node {$\bbeta^{(2)}$};
        \end{tikzpicture} 
    }
    \\
    \subfigure[Large $\gamma^*$: large bias, small bias-difference]{
        \begin{tikzpicture}[scale=1.2]
        \draw[->] (0,0) -- (0,3);
        \draw[->] (0,0) -- (3,0);
        \draw (3, -0.3) node{$x_1$};
        \draw (-0.3, 3) node{$x_2$};
        \draw[myblue, dashed] (2, 0) -- (1.4, 1.3) node{\scriptsize $d_2$} -- (0.8, 2.6);
        \draw[mylightblue, dashed] (2, 0) -- (1.25, 1.25) node{\scriptsize $d_1$} -- (0.5, 2.5);
        \draw[myorange, dashed] (0.5, 2.5) -- (0.8, 2.6);
        \draw[myred, fill] (2, 0) circle[radius=0.06];
        \draw[myred] (2,-0.3) node{${\bbeta}^*$};
        \draw[mylightblue, fill] (0.5, 2.5) circle[radius=0.06];
        \draw[mylightblue] (0.3, 2.1) node{$\bbeta^{(1)}$};
        \draw[myblue, fill] (0.8, 2.6) circle[radius=0.06];
        \draw[myblue] (1.2, 2.6) node {$\bbeta^{(2)}$};
	\draw[myorange] (0.6, 2.7) node{\scriptsize $\Delta$};
        \end{tikzpicture} 
    }
    &
    \subfigure[Impossible: non-zero bias, zero bias-difference]{
        \begin{tikzpicture}[scale=1.2]
        \draw[->] (0,0) -- (0,3);
        \draw[->] (0,0) -- (3,0);
        \draw[mylightblue, dashed] (2-0.01, 0-0.01) -- (1.5-0.01, 1.5-0.01);
        \draw[myblue, dashed] (2+0.01, 0+0.01) -- (1.5+0.01,1.5+0.01);
        
        \draw (3, -0.3) node{$x_1$};
        \draw (-0.3, 3) node{$x_2$};
        \draw[myred, fill] (2, 0) circle[radius=0.06];
        \draw[myred] (2,-0.3) node{${\bbeta}^*$};
        \draw[myblue, fill] (1.5,1.5) circle[radius=0.06];
        \draw[mylightblue] (1.1, 1.5) node{$\bbeta^{(1)}$};
        \draw[mylightblue, fill=mylightblue] (1.5,1.5+0.06) arc [start angle=90, end angle=270, radius=0.06];
        \draw[myblue] (2.1, 1.5) node{$\bbeta^{(2)}$};
        \draw[mylightblue] (1.7, 0.5) node{\scriptsize $d_1$};
        \draw[myblue] (1.9, 0.9) node{\scriptsize $d_2 $};
        \end{tikzpicture} 
    }
    \end{tabular}
    \caption{A geometric illustration of the bias-difference debiasing idea. We consider the same case where $|\mathcal{E}|=p=2$, $x_1$ is the important variable, and $x_2$ is the pooled linear spurious variable. In each subplot, ${\color{myred} \bbeta^*}$ is the true parameter and $\bbeta^{(e)}$ with $e\in \{{\color{mylightblue} 1}, {\color{myblue} 2}\}$ is the population risk minimizer in each environment. Following the discussion in the text, $d_e = \|\bbeta^* - \bbeta^{(e)}\|_2 \asymp \|\bb^{(e)}\|_2$ quantifies the bias of each environment and $\Delta = \|\bbeta^{(1)} - \bbeta^{(2)}\|_2$ represents the bias-difference. The four plots demonstrate four cases in which the magnitudes of bias and bias-difference vary, leading to different thresholds $\gamma^*$ satisfying $\gamma^* \asymp \left(\frac{ {\color{mylightblue} d_1 } + { \color{myblue} d_2} }{{\color{myorange} \Delta}}\right)^{2}$. The above two plots (a) and (b) are the cases where $\gamma^*$ is of reasonable, constant order. We can see when both bias and bias-difference are relatively small in plot (a) or relatively large in plot (b), and the ratio of the two quantities is within constant order that ${\color{mylightblue} d_1} + {\color{myblue} d_2} \asymp {\color{myorange} \Delta}$, the choice of $\gamma^*$ is also of constant order. However, when the bias is much larger than the bias difference that ${\color{mylightblue} d_1} + {\color{myblue} d_2} \gg {\color{myorange} \Delta}$ in (c), one needs to use a large $\gamma^*$ to accommodate the gain in loss decrease from selecting pooled linear spurious variable $x_2$. (d) present a case where the variable set $\{1,2\}$ is also LLS-invariance across the two environments because ${\color{mylightblue} \bbeta^{(1)}}$ and ${\color{myblue} \bbeta^{(2)}}$ coincides. In this case, our proposed EILLS approach will fail and converge to the spurious solution ${\color{mylightblue} \bbeta^{(1)}}={\color{myblue} \bbeta^{(2)}}$ instead of recovering ${\color{myred} \bbeta^*}$.}
    \label{fig:bias-difference-debiasing}
\end{figure}

We illustrate the phenomenon of debiasing at the population level by considering the simplest case of $|\mathcal{E}|=p=2$ and defer a thorough and rigorous analysis to \cref{subsec:strong-convexity}. Here we let $\omega^{(1)}=\omega^{(2)}=1/2$ and $\bbeta^*=(1,0)$ such that the first variable $x_1$ is an important and invariant variable and the second variable $x_2$ is a linear spurious variable. Moreover, suppose $\lambda(\bSigma^{(e)}) \asymp 1$ where $\lambda(\bS)$ represent the eigenvalues of the symmetric matrix $\bS$.

Under the above toy model, the population-level excess pooled $L_2$ risk is given 
\begin{align*}
    \mathsf{R}(\bbeta) - \mathsf{R}(\bbeta^*) = (\bbeta - \bbeta^*)^\top \bar{\bSigma} (\bbeta - \bbeta^*) - 2\times (\bbeta - \bbeta^*)^\top \bb
\end{align*} 
where $\bar{\bSigma} = \frac{1}{2} (\bSigma^{(1)} + \bSigma^{(2)})$ and  $\bb = (0, b_2)$ with  $b_2 = \frac{1}{2} \{\mathbb{E}[\varepsilon^{(1)} x_2^{(1)}] + \mathbb{E}[\varepsilon^{(2)} x_2^{(2)}]\}$.
This indicates that $\bbeta^*$ is not the minimizer of $\mathsf{R}(\bbeta)$ when $b_2 \not = 0$
and one can further decrease the loss in the direction of $(\bar{\bSigma})^{-1} \bb$. Furthermore, the  decrease in $\mathsf{R}(\bbeta)$ by moving towards this direction is  bounded from below by
\begin{align*}
    \mathsf{R}(\bbeta) - \mathsf{R}(\bbeta^*) \ge C_1^{-1} \|\bbeta - \bbeta^*\|_2^2 - C_1 \|\bb\|_2^2
\end{align*} for some constant $C_1>0$. Let us see how the invariance regularizer $\mathsf{J}(\bbeta)$ can compromise the bias. Denote $\bb^{(e)}=(0,\mathbb{E}[\varepsilon^{(e)} x_2^{(e)}])$ be the bias vector in each environment $e$. When $\bbeta$ is supported on $\{1,2\}$, we find
\begin{align*}
    \mathsf{J}(\bbeta) - \mathsf{J}(\bbeta^*) &= \frac{1}{2} \|\bSigma^{(1)} (\bbeta - \bbeta^*) - \bb^{(1)} \|_2^2 + \frac{1}{2} \|\bSigma^{(2)} (\bbeta - \bbeta^*) - \bb^{(2)} \|_2^2 \\
    & \ge 2 C_2^{-1} \Big \{\| \bbeta - \bbeta^* - (\bSigma^{(1)})^{-1} \bb^{(1)} \|_2^2 +  \|\bbeta - \bbeta^* - (\bSigma^{(2)})^{-1}\bb^{(2)} \|_2^2 \Big \}\\
    &\ge C_2^{-1} \|(\bSigma^{(1)})^{-1} \bb^{(1)} - (\bSigma^{(2)})^{-1} \bb^{(2)}\|_2
\end{align*} for another constant $C_2>0$. This demonstrates that one needs to pay an extra cost of bias-difference $\geq \gamma C_2^{-1}\|(\bSigma^{(1)})^{-1} \bb^{(1)} - (\bSigma^{(2)})^{-1} \bb^{(2)}\|_2$ when adopting pooled linear spurious variables. Such a cost can compromise the gain of the decrease in $\mathsf{R}(\bbeta)$ when a large $\gamma$ is used. Formally, we have
\begin{align*}
    \forall \bbeta ~~\text{with}~~ \|\bbeta \|_0 = 2, ~~~~~~\mathsf{Q}(\bbeta) - \mathsf{Q}(\bbeta^*) \ge C_1^{-1}\| \bbeta - \bbeta^*\|_2^2
\end{align*} provided
\begin{align}
\label{eq:gamma-star-sketch}
    \gamma \ge \gamma^* = C_1 C_2 \frac{\|\bb\|_2^2}{\|(\bSigma^{(1)})^{-1} \bb^{(1)} - (\bSigma^{(2)})^{-1} \bb^{(2)}\|_2} \asymp \frac{\|(\bbeta^{(1)} - \bbeta^*) + (\bbeta^{(2)} - \bbeta^*)\|_2^2}{\|(\bbeta^{(1)} - \bbeta^*) - (\bbeta^{(2)} - \bbeta^*)\|_2^2}
\end{align} where $\bbeta^{(e)} = \argmin_{\bbeta \in \mathbb{R}^2} \mathsf{R}^{(e)}(\bbeta)$ is the  population risk minimizer for environment $e\in \mathcal{E}$. The R.H.S. of \eqref{eq:gamma-star-sketch} follows from the fact that $\bbeta^{(e)} - \bbeta^* = (\bSigma^{(e)})^{-1} \bb^{(e)}$. We present a geometric illustration of the above discussion and how the bias and bias-difference jointly affect the critical threshold $\gamma^*$ in \cref{fig:bias-difference-debiasing}.

\subsection{Implementation Details of Simulations in \cref{sec:simulation}}
\label{subsec:implementation}

The structural assignment of the SCM in $e=1$ and $e=2$ are as follows
\begin{align*}
    x^{(e)}_1 &\gets u_1^{(e)} \\
    x^{(1)}_4 &\gets u_4^{(1)} \qquad \qquad \qquad \qquad &x^{(2)}_4 \gets (u_4^{(2)})^2-1\\
    x^{(e)}_2 &\gets \sin(x_4^{(e)}) + u_2^{(e)} \\
    x^{(e)}_3 &\gets \cos(x_4^{(e)}) + u_3^{(e)} \\
    x^{(e)}_5 &\gets \sin(x_3^{(e)} + u_5^{(e)}) \\
    x^{(e)}_{10} &\gets 2.5 x_1^{(e)} + 1.5 x_2^{(e)} + u_{10}^{(e)} \\
    y^{(e)} &\gets 3 x_1^{(e)} + 2x_2^{(e)} - 0.5 x_3^{(e)} + u_{13}^{(e)} \\
    x_6^{(e)} &\gets 0.8 y^{(e)} u_6^{(e)} \\
    x_7^{(1)} &\gets 0.5 x_3^{(1)} + y^{(1)} + u_7^{(1)} & x_7^{(2)} \gets 4 x_3^{(2)} + \tanh(y^{(2)}) + u_7^{(2)} \\
    x_8^{(e)} &\gets 0.5 x_7^{(e)} - y^{(e)} + x_{10}^{(e)} + u_8^{(e)} \\
    x_9^{(e)} &\gets \tanh(x_7^{(e)}) + 0.1 \cos(x_8^{(e)}) + u_9^{(e)} \\
    x^{(e)}_{11} &\gets 0.4 (x_7^{(e)} + x_8^{(e)}) * u_{11}^{(e)} \\
    x_{12}^{(e)} &\gets u^{(e)}_{12}
\end{align*} 
Here $u^{(e)}_1,\ldots, u^{(e)}_{13} \sim \mathcal{N}(\bm{0}, I_{13\times 13})$ for all the $e\in \mathcal{E}$.

\noindent \textbf{Implementation.} We use brute force search to calculate $\hat{\bbeta}_{\mathsf{Q}}$. To be specific, we enumerate all the possible support set $S\in 2^{[p]}$, for each fixed $S$, the EILLS objective \eqref{eq:eills-objective-q} is a quadratic function of $[\bbeta]_S$, whose minimum value $\hat{\bbeta}^{(S)}$ can be found by setting the first order condition to be held, that is, 
\begin{align*}
        \hat{\bbeta}^{(S)} &= \mathop{\mathrm{argmin}}_{\mathrm{supp}(\bbeta) \subseteq S} \sum_{e\in \mathcal{E}} \omega^{(e)} \hat{\mathbb{E}}[|y^{(e)} - \bbeta^\top \bx^{(e)}|^2] + \gamma \sum_{e\in \mathcal{E}} \left\|\hat{\mathbb{E}}[\{y^{(e)} - \bbeta^\top \bx^{(e)}\}\bx_{S}^{(e)}]\right\|_2^2 \\
        &= \left[\hat{\bbeta}^{(S)}_S, \hat{\bbeta}^{(S)}_{S^c}\right] \\
        &= \left[\left\{\sum_{e\in \mathcal{E}} \omega^{(e)} \hat{\bSigma}^{(e)}_S + \gamma \sum_{e\in \mathcal{E}} \omega^{(e)} (\hat{\bSigma}^{(e)}_S)^2\right\}^{-1} \left\{ \sum_{e\in \mathcal{E}} \omega^{(e)} \hat{\mathbb{E}}[\bx_S^{(e)} y^{(e)}] + \gamma \sum_{e\in \mathcal{E}} \omega^{(e)} \hat{\bSigma}_{S} \hat{\mathbb{E}}[\bx_S^{(e)} y^{(e)}]\right\}, \bm{0} \right].
\end{align*}
Then $\hat{\bbeta}_{\mathsf{Q}}$ is assigned to be $\hat{\bbeta}^{(S)}$ with the minimum empirical objective value $\hat{\mathsf{Q}}$. The total computational complexity is $\mathcal{O}(2^p p^3 + \sum_{e\in \mathcal{E}} n^{(e)} p^2)$.

We argue here it is possible to do some relaxation on the objective \eqref{eq:eills-objective-q} such that it may be efficiently solved via gradient descent or other methods. However, minimizing $\hat{\mathsf{Q}}$ is still a non-convex problem. We leave an efficient implementation as future work since this paper focused on a thorough statistical analysis for the multiple environment linear regression.

\noindent \textbf{Implementation for Other Invariance Based Methods.} It is noteworthy that all the invariance-based methods have hyper-parameters related to ``invariance'' besides. For example, the hyper-parameters balancing least squares and invariance regularizer in IRM and Anchor regression, and the Type-I error threshold $\alpha$ in ICP. The criterion for choosing this type of hyper-parameter is fundamentally different from that for hyper-parameters controlling the statistical complexity, such as $L_1$/$L_2$ regularization. We can use a train/valid split for the latter and choose the hyper-parameters using the validation dataset. On the contrary, the optimal hyper-parameter for the former should be tuned using the ``test dataset''. In the comparison experiment, the choice of hyper-parameters for ICP, IRM, and Anchor regression are picked in an oracle manner, that is, we enumerate all the possible hyper-parameters and choose the one that minimizes the $\ell_2$ prediction error $\|\bar{\bSigma}^{1/2}(\hat{\bbeta} - \bbeta^*)\|_2^2$. We set the candidate set to be $\{0,1e-5,1e-4,1e-3,1e-2,1e-1,1\}$ for IRM, $\{0, 1, 2, 4, 8, 10, 15, 20, 30, 40, 60, 80, 90, 100, 150, 200, 500, 1000, 5000, 10000\}$ for Anchor regression. For ICP, we try the Type-I error parameter in $\{0.9, 0.95, 0.99, 0.995\}$. We use the dummy variable $1\{E=1\}$ as the anchor variable for Anchor regression.

\noindent \textbf{Configurations for figures.} \cref{fig:numerical} (a) plots the estimated coefficients $\hat{\beta}_j$ over $\gamma \in \{0, 1, 2, 3, 8\} \cup \{5k: k\in [19]\} \cup \{100j:j\in [9]\} \cup \{1000\ell: \ell\in [10]\}$. \cref{fig:numerical} (b) (c) enumerates $n$ in $\{100k: k\in [10]\} \cup \{1500, 2000\}$ and all the plotted quantities are the average over 500 replications.

\subsection{A Diagnosis on Why Refitting is Worse than Vanilla EILLS} 
\label{sec:diagnosis}

In \cref{fig:numerical} (b), we see that the performance of EILLS with refitting is slightly worse than that of the vanilla EILLS estimator. Here, we provide a possible explanation for why there is a noticeable gap when $n=100$ from the observations in simulation studies and leave the finer analysis for future studies. \cref{fig:refit} visualizes the solution EILLS and EILLS with refitting produced in 100 replications when $n=100$. In the plot, each point $(x, y)$ with marker $m$ represents a solution $\hat{\bbeta}$ that the method $m$ produces. Here, $x$ denotes the relative $\ell_2$ norm restricted to the true important variables set $S^*$, calculated as $\|\hat{\bbeta}_{S^*}\|_2/\|{\bbeta}^*_{S^*}\|_2$; and $y$ represents the relative $\ell_2$ norm restricted to the pooled linear spurious variables set $G$, expressed as $\|\hat{\bbeta}_{G}\|_2/\|\bar{\bbeta}_{G}\|_2$ with $\bar{\bbeta} = \bar{\bbeta}^{([12])}$.  The solutions in each trial are connected by an arrow from EILLS to EILLS with refitting, indicating that the latter is the refitted version of the former. 

As shown in \cref{fig:refit}, some solutions falsely select variables in $G=\{7,8,9\}$. For these cases, the EILLS solutions tend to be more dependent on the variables in $S^*$ and less dependent on the variables in $G$ than the EILLS with refitting solution. In other words, the EILLS method yields a more robust solution when the variable selection property fails to hold. We guess that will contribute to the $\ell_2$ estimation error gap between the two methods.

    \begin{figure*}[t!]
    \centering
        \includegraphics[scale=0.5]{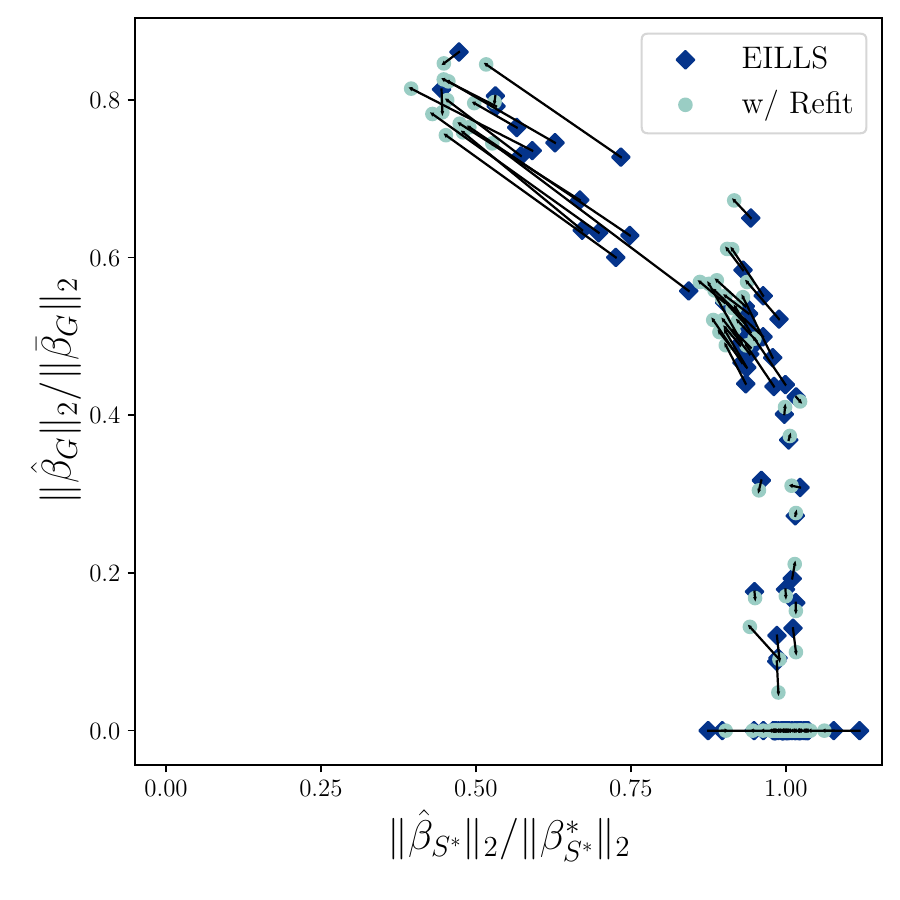}
    \caption{Visualization of solutions EILLS and EILLS with refitting produce in 100 replications when $n=100$.}
    \label{fig:refit}
    \end{figure*}

\subsection{EILLS by Gumbel Approximation}
\label{sec:gumbel}

In this section, we briefly describe how to use stochastic approximation and Gumbel approximation to handle the combinatorial nature of optimization and let a variant of gradient descent continue to work. This is supported by simulation results when $p=70$. We follow the notations in \citep{gu2024causality}.

The original EILLS objective can be written as 
\begin{align}
\label{fan5}
    (\hat{\beta}, \hat{a}) \in \mathop{\mathrm{argmin}}_{\beta \in \mathbb{R}^p, a\in \{0,1\}^p} \underbrace{\sum_{e\in \mathcal{E}} \hat{\mathbb{E}}[|Y^{(e)} - (\beta \odot a)^\top X^{(e)}|^2] + \gamma \sum_{j=1}^p a_j \left|\hat{\mathbb{E}}[\{Y^{(e)} - (\beta \odot a)^\top X^{(e)}\} X_j] \right|^2}_{\hat{\mathsf{Q}}_\gamma(\beta, a)},
\end{align} where $\odot$ is the point-wise multiplication, i.e., $[x \odot y]_j = x_j y_j$ for any $x, y\in \mathbb{R}^d$. This step is used to disentangle the effect of variable selection and parameter estimation, but the combinatorial nature remains. To avoid this, we first rewrite the optimization as a ``continuous'' optimization:
\begin{align*}
    (\hat{\beta}, \hat{w}) \in \argmin_{\beta \in \mathbb{R}^p, w\in \mathbb{R}^p} \mathbb{E}_{B(w)} \left [ \hat{\mathsf{Q}}_\gamma(\beta, B(w)) \right ],
\end{align*}
where the $j^{th}$ component of $B(w) \in \{0, 1\}^d$ follows an independent Bernoulli with probability of success $\sigma(w_j) = \exp(w_j)/(1+\exp(w_j))$.  This is easily seen by taking $\hat{w} = \mathrm{logit}(\hat a)= \log (\frac{\hat a}{1 - \hat a})$ (taking values $\pm \infty$).  Note that $B_j(w_j) = I(\mathrm{logit} (U_j) \leq w_j)$ is discontinuous in $w_j$ where $U_j \sim$ Uniform[0,1], but can be approximated as
\begin{equation} \label{fan6}
    B_j(w_j) \approx \frac{1}{1 + e^{ (\mbox{\scriptsize logit}(U_j) - w_j))/\tau}} \equiv V_\tau(U_j, w_j) ~~~~ \text{as} ~~~~\tau \to 0^+,
\end{equation}
for which its gradient can be taken.  Let
\begin{align*}
    A_\tau(U, w) = (V_\tau(U_{1}, w_1),\ldots, V_\tau(U_{d}, w_d))^\top \in \mathbb{R}^d
\end{align*} with $\{U_j\}_{j=1}^d$ being i.i.d. uniform random variables.
One can approximate the original objective \eqref{fan5} by
\begin{align}
\label{eq:final}
    (\hat{\beta}, \hat{w}) \argmin_{\beta \in \mathbb{R}^p, w\in \mathbb{R}^p} \mathbb{E}_{U} \left [ \hat{\mathsf{Q}}_\gamma(\beta, A_\tau(U, w)) \right ].
\end{align} 
Since $\mathrm{logit}(U_j) \stackrel{d}{=} U_{j,1}- U_{j, 2}$ with $\{U_{j,1}, U_{j,2}\}_{j=1}^d$ being i.i.d. Gumbel(0,1) random variables, the approximation \eqref{fan6} is also referred to as the Gumbel approximation. Given \eqref{eq:final}, one can adopt the following variants of gradient descent in Algorithm \ref{algo1}.

\begin{algorithm}
\caption{Scaling EILLS to High-dimension by Gumbel Trick}
\begin{algorithmic}[1]
\State \textbf{Hyper-parameter: } number of iterations $T$, hyper-parameter $\gamma$
\State \textbf{Gumbel parameters: } initial/final temperature $(\tau_0,\tau_T)$, anneal rate $\rho$, anneal iteration $T_\tau$.
\State \textbf{Input:} data $\{(X^{(e)}_i, Y^{(e)}_i)\}_{i\in [n], e\in \mathcal{E}}$
\State Initialize $\beta, w$
\State Set $\tau = \tau_0$
\For {$t \in \{1,\ldots, T\}$} 
    \State $\tau = \max(\tau_T, \tau \times \rho)$ \textbf{if} $t ~\mathrm{mod}~ T_\tau = 0$.
    \State Sample $\{U_{j,1}, U_{j,2}\}_{j=1}^p$ from Gumbel(0,1).
    \State Update $\beta, w$ by descending its gradient
        \begin{align*}
                \nabla_{(\beta, w)} \left[\hat{\mathsf{Q}}_\gamma(\beta, A_\tau(U, w))\right]
        \end{align*}
\EndFor
\State \textbf{Output:} estimate $\beta \odot \sigma(w)$. 
\end{algorithmic}
\label{algo1}
\end{algorithm}

\subsubsection{Simulation Results}

This section is to illustrate the performance of the above Gumbel-trick optimized EILLS estimator under moderate high dimension $p=70$, where brute force search is not feasible. The data-generating process and the experimental setting are the same as Section 5.2.1 in \cite{gu2024causality}. 

\noindent \textbf{Data Generating Process.} We consider the case where $|\mathcal{E}|=2$ and the data $(X^{(e)}, Y^{(e)})$ in each environment $e\in \{0,1\}$ are generated from two SCMs sharing the same causal relationship between variables. For each trial, we first generate the parent-children relationship among the variables. We enumerate all the $i\in [p+1]$. For each $i\in [p+1]$, we randomly pick at most $4$ parents for the variable $Z_i$ from $\{Z_1,\ldots, Z_{i-1}\}$, this step ensures that the induced graph is a DAG. We use fixed $p=70$, and let the variable $Z_{36}$ be $Y$ and the rest variables constitute the covariate $X$, that is, we let $(Z_1,\ldots, Z_{35}, Z_{36}, Z_{37}, \ldots, Z_{71})=(X_1,\ldots, X_{35}, Y, X_{36}, \ldots, X_{70})$. We also enforce that $Y$ has at least $5$ parents and at least $5$ children by adding parents and children when needed.  The structural assignment for each variable $Z_j$ is defined as
\begin{align*}
    Z_j^{(e)} \gets \sum_{k\in \mathtt{pa}(j)} C_{j,k}^{(e)} f_{j,k}^{(e)}(Z_k^{(e)}) + C_{j, j}^{(e)} \varepsilon_j
\end{align*} where $(\varepsilon_1,\ldots, \varepsilon_{71})$ are independent standard normal random variables. For $j\neq 36$, $f_{j,k}^{(e)}$ are sampled randomly from the candidate functions $\{\cos(x), \sin(x), \sin(\pi x), x, 1/(1+e^{-x})\}$, $C_{j,k}^{(e)}$ are sampled from the uniform distribution on $[-1.5, 1.5]$ with $|C_{j,j}^{(e)}| \ge 0.5$. For $j=36$ and $k<j$, we have $f_{36,k}^{(e)}(x)=x$ and $C_{36,k}^{(0)}\equiv C_{36,k}^{(1)}$ for linearity and invariance. The above data-generating process can be regarded as one observation environment $e=0$ and an interventional environment $e=1$ where the random and simultaneous interventions are applied to all the variables other than the variable $Y$, while the assignment from $Y$'s parent to $Y$ remains and furnishes the target regression function $f^*(x) = \sum_{k\in \mathtt{pa}(36)} C_{36, k}^{(e)} x_k$ in pursuit. In this case, we let $S^* = \mathtt{pa}(36)$ and $\beta^*$ with support set $S^*$ be such that $\beta_j^*=C^{(0)}_{36, k}=C^{(1)}_{36, k}$ for any $k\in S^*$. We also let the noise variance be different for the two environments, i.e., $C_{36,36}^{(0)} \neq C_{36,36}^{(1)}$. Now, the model only has conditional expectation invariance rather than the full conditional distribution invariance. \cref{fig:visualize-linear} (a) visualizes the induced graph in one trial. The complex cause-effect relationships in high-dimensional variables make the problem of estimating $\beta^*$ very challenging.

\begin{figure}[!t]
\centering
\begin{tabular}{cc}
\subfigure[]{
\includegraphics[scale=0.4]{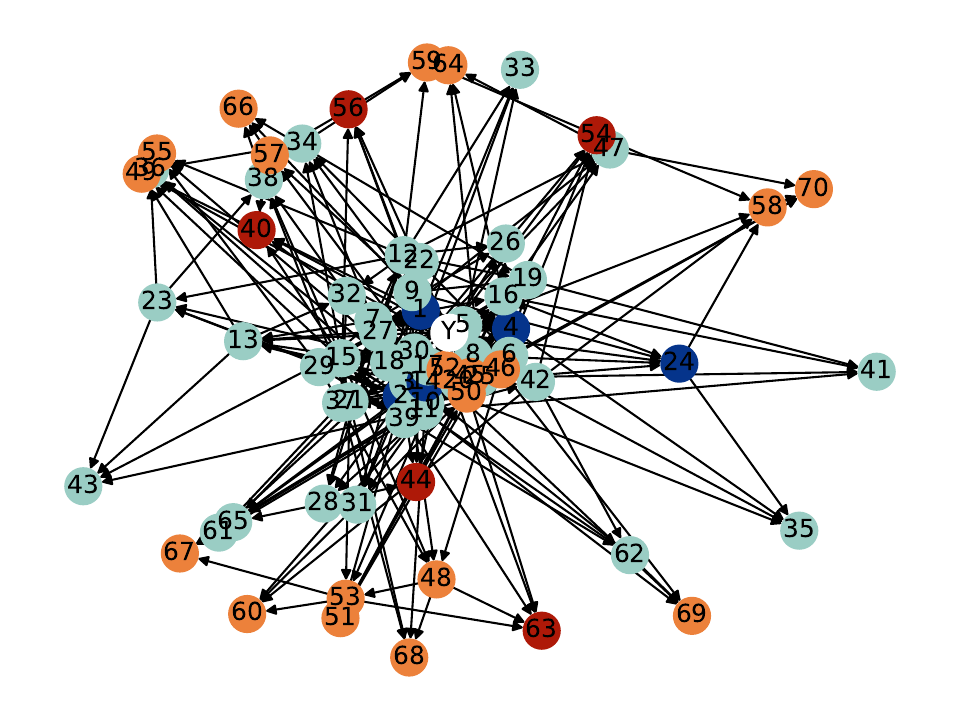}
}&
\subfigure[]{
\includegraphics[scale=0.3]{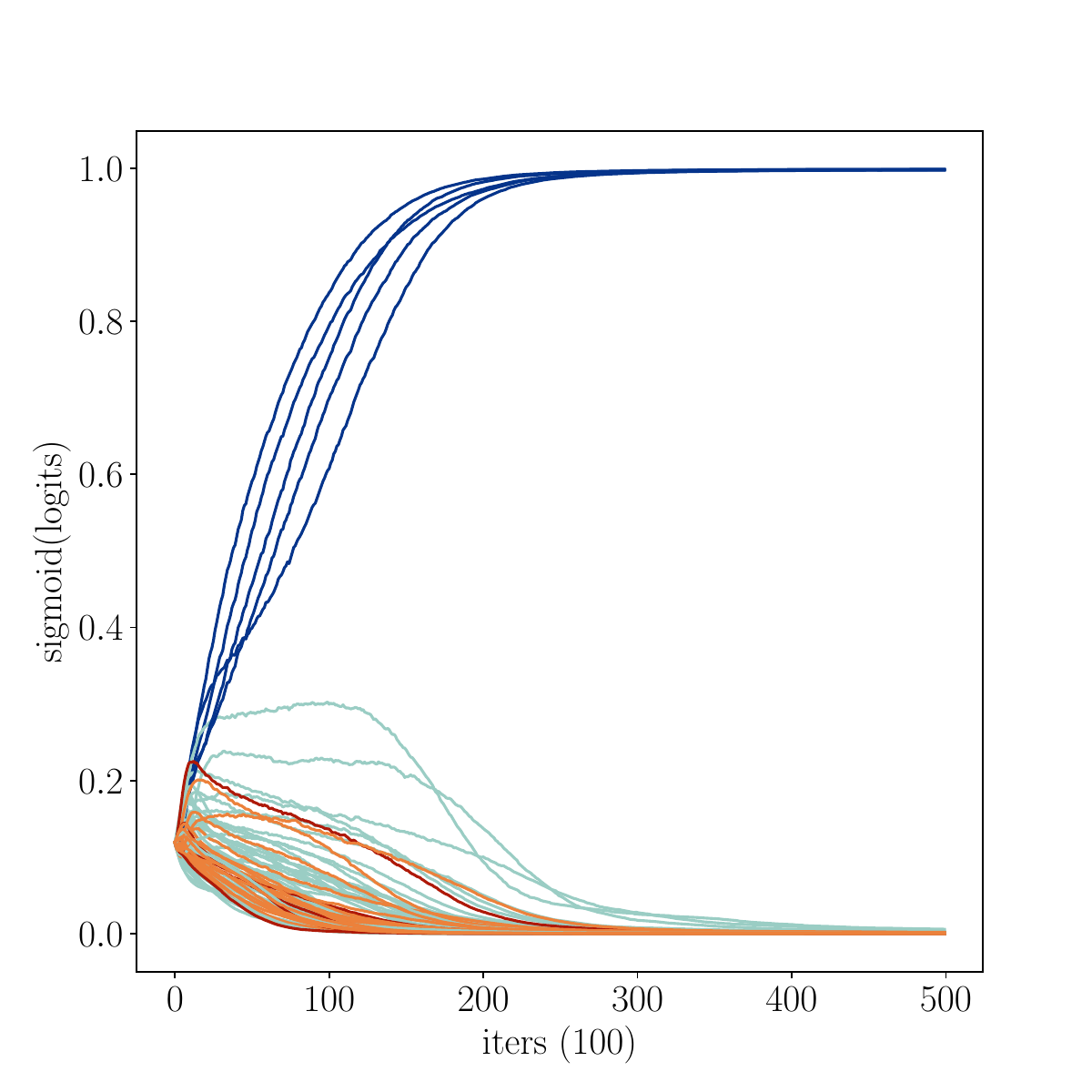}
}
\end{tabular}
\caption{The visualization of (a) the SCM and (b) the $\sigma({w})$ during training in one trial for the FAIR-Linear estimator. We use different colors to represent the different relationships with $Y$: \myblue{blue} = parent, \myred{red} = child, \myorange{orange} = offspring, \mylightblue{lightblue} = other.}
\label{fig:visualize-linear}
\end{figure}

\medskip
\noindent \textbf{Implementation.} For the EILLS estimator realized by gradient descent with Gumbel trick, we run gradient descent ascent using Adam optimizer with a learning rate of 1e-3, batch size $64$ for $50k$ iterations. We adopt a fixed hyper-parameter $\gamma=10$ and report the performance of the following estimators using the median of the estimation error $\|\hat{\beta} - \beta^*\|_2^2$ over $50$ replications and varying $n\in \{200, 500, 1000, 2000, 5000\}$.
\begin{itemize}
    \item[(1)] Pool-LS: it simply runs least squares on the full covariate $X$ using all the data. 
    \item[(2)] EILLS-GB: Our EILLS estimator with Gumbel approximation that outputs $\beta \odot \sigma(w)$.
    \item[(3)] EILLS-RF: it selects the variables $x_j$ with $\sigma({w_j})>0.7$ of the fitted model in (2), i.e., $\hat{S}=\{j: \sigma(w_j)>0.7\}$, and refits least squares again on $X_{\hat{S}}$ using all the data.
    \item[(4)] Oracle: it runs least squares on $X_{S^*}$ using all the data.
    \item[(5)] Semi-Oracle: it runs least squares on $X_{G^c}$ using all the data, where $G$ is the set of all the descendants of $Y$. Compared with the ERM, it manually removes all the variables that will lead to a biased estimation, but it will also keep uncorrelated variables compared with the full Oracle estimation.  
\end{itemize}

\cref{fig:visualize-linear} (b) visualizes how the Gumbel gate values for different covariables $\sigma(w)$ evolve during training in one trial. We can see that $\sigma(w_j)$ for $j\in S^*$ quickly increases and dominates the values for other variables like children/offspring of $Y$ during the whole training process. 

\medskip
\noindent \textbf{Results. } The results are shown in \cref{fig:simulation-result}. We can see that the square of the $\ell_2$ estimation error $\|\hat{\beta} - \beta^*\|_2^2$ for the pooled least squares estimator (\myyellow{$\times$}) does not decrease and remains to be very large ($\approx 1.5$) as $n$ increases, indicating that it converges to a biased solution. At the same time, the estimation error for EILLS-GB/EILLS-RF (\mylightblue{$\blacklozenge$}/\myblue{$\blacksquare$}) decays as $n$ grows and lies in between that for least squares on $X_{G^c}$ (Semi-Oracle \myorange{$\blacktriangledown$}) and least squares on $X_{S^*}$ (Oracle \myred{$\blacktriangle$}).

\begin{figure}[!t]
\centering
\includegraphics[scale=0.47]{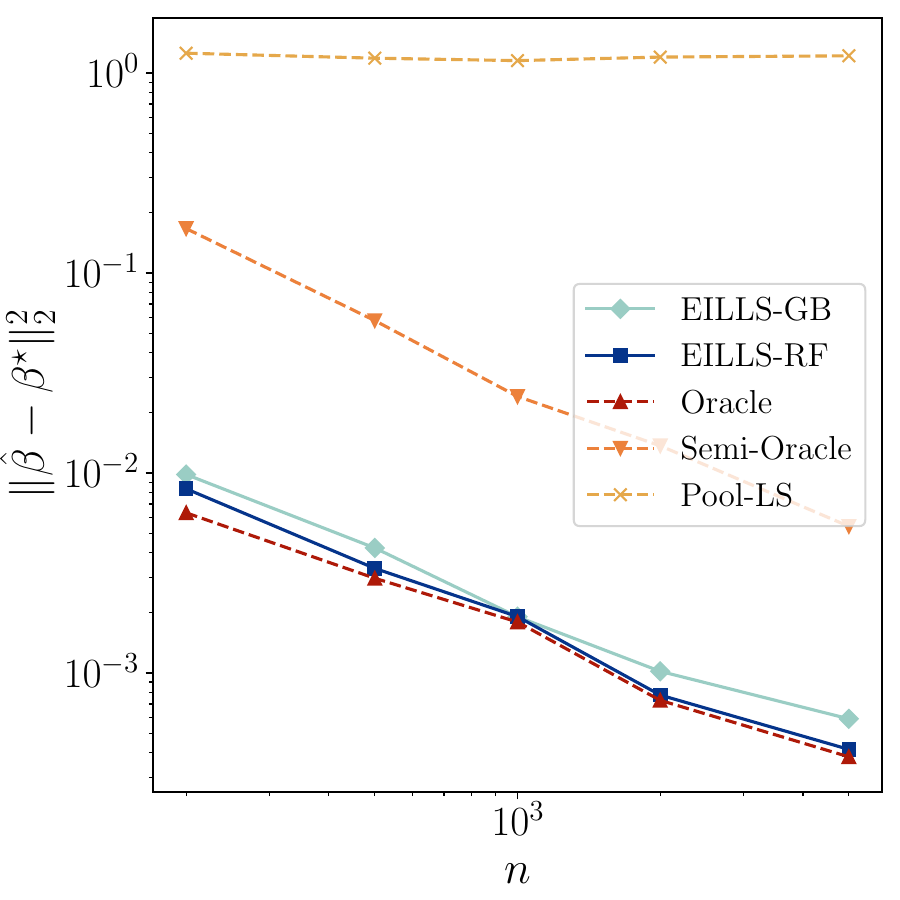}
\caption{The simulation results for linear models with $p=70$. It depicts how the median estimation errors (based on $50$ replications, shown in log scale) for different estimators (marked with different shapes and colors) change when $n$ varies in $\{200, 500, 1000, 2000, 5000\}$ respectively. 
}
\label{fig:simulation-result}
\end{figure}

\subsection{Choice of $\gamma$ in Practice}
\label{sec:gamma-tune}

We argue that the EILLS estimator is not very sensitive to the choice of $\gamma$ when the true causal signal is strong enough, thus one can, for example, adopt $\gamma = 36$, or $\gamma=100$. Such insensitivity is attributed to the discontinuity nature of the loss discussed in \cref{sec:small-gamma}. For example, in the simulation considered in this paper where there exists weak signal $\beta_3 = -0.5$, EILLS discards correct variables only when $\gamma$ is relatively large $\gamma > 3\times 10^3$ and it works pretty well in a wide range of $\gamma \in [15, 3\times 10^3]$ as shown in \cref{fig:numerical} (a).
    
    When the target is to produce the best predictions on unseen future data, we can choose a proper $\gamma$ from a set of candidate $\gamma$-values that has the best out-of-sample performance in certain validation environment(s). To be specific, we consider the two cases, for the first case, we have some training environments $\mathcal{E}_{train}$ and also some validation environments $\mathcal{E}_{valid}$ whose associations between $y$ and $\bx$ are similar to those to be confronted in the future by our belief. We can adopt the following procedure in this case:

    \begin{itemize} 
        \item[Step 1] Set candidate hyper-parameter set $\Gamma$.
        \item[Step 2] For each $\gamma \in \Gamma$, run EILLS estimator using data in $\mathcal{E}_{train}$ with $\gamma$ and get the estimated $\hat{\bbeta}^\gamma$, calculate the worse-case out-of-sample empirical $L_2$ risk among the validation set as
        \begin{align*}
            \hat{R}_\gamma = \max_{e\in \mathcal{E}_{valid}} \frac{1}{n^{(e)}} \sum_{i=1}^{n^{(e)}} \left\{y_i^{(e)} - (\hat{\bbeta}^\gamma)^\top \bx_i^{(e)} \right\}^2.
        \end{align*}
        \item[Step 3] We choose $\gamma$ from $\Gamma$ that corresponds to the minimum out-of-sample empirical $L_2$ risk, that is, $\hat{\gamma} = \argmin_{\gamma \in \Gamma} \hat{R}_\gamma$.
    \end{itemize}

    For the second case, suppose we only have $\mathcal{E}_{train}$, we can adopt the following leave-one-out procedure. 

    \begin{itemize} 
        \item[Step 1] Set candidate hyper-parameter set $\Gamma$.
        \item[Step 2] For fixed $\gamma \in \Gamma$, enumerate all the $e\in \mathcal{E}_{train}$. For each $e \in \mathcal{E}_{train}$, run EILLS estimator using data in $\mathcal{E}_{train} \setminus \{e\}$ with $\gamma$ and get the estimated $\hat{\bbeta}^{\gamma, e}$, calculate its out-of-sample empirical $L_2$ risk in $e$ as
        \begin{align*}
            \hat{R}_{\gamma, e} = \frac{1}{n^{(e)}} \sum_{i=1}^{n^{(e)}} \left\{y_i^{(e)} - (\hat{\bbeta}^{\gamma,e})^\top \bx_i^{(e)} \right\}^2,
        \end{align*} and get its maximum value $\hat{R}_{\gamma} = \max_{e\in \mathcal{E}_{train}} \hat{R}_{\gamma, e}$. 
        \item[Step 3] We choose $\gamma$ from $\Gamma$ that corresponds to the minimum out-of-sample empirical $L_2$ risk, that is, $\hat{\gamma} = \argmin_{\gamma \in \Gamma} \hat{R}_\gamma$.
    \end{itemize}

\subsection{The Interpretation of Small $\gamma$}
\label{sec:small-gamma}

When $\gamma$ is not large enough, i.e., $\gamma < \gamma^*$, the EILLS objective is somewhat similar to running (penalized) pooled least squares on variables excluding some spurious variable whose spuriousness to heterogeneity ratio is smaller than $\gamma$. We use the following variant of the thought experiment to illustrate the idea. Suppose we still do the cow/camel classification using three features $x_1=$~shape, $x_2=$~backgrouond, and $x_3=$~whether the object stands or not. In the two-environment training dataset, 95\% camels/cows on sand/grass, and 95\% camels/cows sit/stand in $\mathcal{D}_1$. Moreover, 90\% camels/cows on sand/grass, and 70\% camels/cows sit/stand in $\mathcal{D}_2$. Under mild conditions akin to \cref{cond-ident-main}, the regularization path of the population-level minimizer of EILLS $\bbeta_\gamma = \argmin_{\bbeta\in \mathbb{R}^p} \mathsf{Q}(\bbeta;\gamma,\bomega)$ can be interpreted as follows. There are change points $\gamma_2>\gamma_1>0$. When $\gamma \in [0,\gamma_1)$, $\bbeta_\gamma$ will be similar to regressing $Y$ on $(x_1,x_2, x_3)$ using all the data; it will threshold $X_3$ and thus be similar to regressing $y$ on $(x_1,x_2)$ using all the data as $\gamma \in (\gamma_1, \gamma_2)$; it will finally recover the ground-truth causal parameter $\bbeta^*$ with $\mathrm{supp}(\bbeta^*)=\{1\}$ when $\gamma > \gamma_2$.

\end{document}